\documentclass{amsart}
\usepackage{mathrsfs,,amsmath,amssymb,amsthm,amsfonts,mathtools}
\usepackage[margin=1in]{geometry}
\usepackage{xcolor}
  \definecolor{b+}{rgb}{0.5, 0.5, 1}
  \definecolor{b-}{rgb}{0, 0, 0.5}
  \definecolor{r|b-}{rgb}{0.5, 0, 0.25}
  \definecolor{g|b-}{rgb}{0, 0.5, 0.25}
  \definecolor{shadecolor}{rgb}{0.75, 0.75, 1}
\usepackage[
  colorlinks,
  urlcolor = black,
  linkcolor = b-,
  anchorcolor = r|b-,
  citecolor = g|b-
]{hyperref}

\def\bbC{\mathbb{C}}

\def\({\left(}
\def\){\right)}

\def\lbk{\left\{}
\def\rbk{\right\}}
\def\lmdl{\left\vert}
\def\rmdl{\right\vert}
\def\lnm{\left\Vert}
\def\rnm{\right\Vert}

\def\Om{\Omega}

\newcommand{\cl}[1]{\overline{#1}}
\DeclareMathOperator{\dom}{dom}

\usepackage{comment}
\theoremstyle{definition}
\newtheorem{deft}{Definition}[section]

\theoremstyle{definition}

\theoremstyle{definition}
\newtheorem{thm}{Theorem}[section]

\theoremstyle{definition}
\newtheorem{lem}{Lemma}[section]

\theoremstyle{definition}
\newtheorem{coro}{Corollary}[section]

\theoremstyle{definition}
\newtheorem{prop}{Proposition}[section]

\theoremstyle{definition}

\theoremstyle{definition}

\usepackage{comment}

\theoremstyle{definition}
\newtheorem{rmk}{Remark}[section]

\theoremstyle{definition}

\numberwithin{equation}{section}

\begin{document}

\title[Sobolev estimates for $\bar\partial$]{Sobolev regularity for the $\bar\partial$-Neumann operator and transverse vector fields}
\author{Qianyun Wang, Yuan Yuan, and Xu Zhang}
% \subjclass[2010]{}
%\date{\today}
\address{School of Statistics and Mathematics, Shanghai Lixin University of Accounting and Finance, Shanghai 201209, China}
\email{wangqy1226@gmail.com}
\address{Institute for Theoretical Sciences, Westlake University, Hangzhou, Zhejiang 310030, China}
\email{yuanyuan@westlake.edu.cn}
\address{School of Mathematical Sciences,  Tongji University, Shanghai 200092, China.}
\email{xzhangmath@tongji.edu.cn}
\keywords{pseudoconnvex domain, $\bar\partial$-Neumannn operator, Sobolev regularity}
\subjclass[2020]{32W05, 35B65, 32T99}
\maketitle

\begin{abstract}
On a bounded smooth pseudoconvex domain in $\mathbb{C}^n$ with $n >2$, inspired by the compactness condition introduced by Yue Zhang, we present the new sufficient condition for the exact regularity of the $\bar\partial$-Neumann operator via the transverse vector fields. 
\end{abstract}

\section{introduction}
On a  bounded smooth pseudoconvex domain $\Omega$ in $\mathbb{C}^n$, $\bar\partial: L^2_{(0, q)}(\Omega) \rightarrow L^2_{(0, q+1)}(\Omega)$ and its Hilbert adjoint $\bar\partial^*: L^2_{(0, q+1)}(\Omega) \rightarrow L^2_{(0, q)}(\Omega)$ are both densely defined, closed, unbounded operators. The complex Laplacian $\Box _{q}: \bar\partial \bar\partial^* + \bar\partial^* \bar\partial:  L^2_{(0, q)}(\Omega) \rightarrow L^2_{(0, q)}(\Omega)$ is then a self-adjoint, densely defined, closed operator, whose domain of definition $\dom( \Box _{q})$ consists of elements in $L^2_{(0, q)}(\Omega)$ satisfying the $\bar\partial$-Neumann condition. By the fundamental $L^2$ estimates due to H\"ormander and Kohn,  $\Box _{q}$ admits a bounded inverse $N_q: L^2_{(0, q)}(\Omega) \rightarrow L^2_{(0, q)}(\Omega)$, called the $\bar\partial$-Neumann operator. The regularity of $N_q$ in various function spaces is one of the fundamental problems in several complex variables and partial differential equations. A linear operator $\mathcal T$ is called exactly regular on the $L^2$-Sobolev space (resp.  globally regular) if  $ \mathcal T: W^k_{(0, q)}(\overline{\Om}) \rightarrow W^k_{(0, q)}(\overline{\Om})$ is continuous for all $k \geq 0$   \big(resp.         $\mathcal T: C^\infty_{(0, q)}(\overline{\Om}) \rightarrow C^\infty_{(0, q)}(\overline{\Om})$      is continuous\big). In this paper, we focus on the exact %$L^2$-Sobolev 
regularity of the $\bar\partial$-Neumann operator $N_q$ on bounded smooth pseudoconvex domains in $\mathbb{C}^n$.

If $\Omega$ is strongly pseudoconvex, Kohn obtained the sub-elliptic estimates for the $\bar\partial$-Neumann operator $N_q$ in his foundational works \cite{Kohn63, Kohn64}. In particular, $N_q$ is exactly regular. However, Christ showed that  $N_q$ is not globally regular \cite{C96} on the worm domains of Diederich
and Fornaess \cite{DF77}, based on Barrett's striking irregularity result of the Bergman projection \cite{B92}. Since then, there have been extensive studies to characterize the domains on which the exact regularity of $N_q$ holds. 

On the other hand, Kohn and Nirenberg discovered that $N_q$ is exactly regular if it is compact on $L^2_{(0, q)}(\Omega)$ \cite{KN65}. The sufficient conditions for $N_q$ to be compact include property ($P_q$) defined by Catlin \cite{C84} and the generalization called property ($\widetilde{P}_q$) by McNeal  \cite{M02}, where both conditions require the existence of a family of functions with Hessian sufficiently large on $\partial\Omega$.  For the later progress on compactness of $N_q$, the interested reader may refer to the survey by Fu-Straube \cite{FS01}.  More recently, in order to find a sufficient condition for the compactness that is  invariant under biholomorphisms, Yue Zhang was led to define the variant of  %property ($P_q$) and the 
property ($\widetilde{P}$), called %property ($P^{\#}_q$) and 
property ($\widetilde{P}^{\#}$) in \cite{YZ2021}. Moreover, Yue Zhang showed that if the domain satisfies %property ($P^{\#}_q$) or the 
property ($\widetilde{P}^{\#}_q$), then $N_q$ is compact \cite{YZ2021}.  
 
In the fundamental work \cite{BS91}, Boas and Straube introduced a family of vector fields $v_\varepsilon$ of $(1, 0)$-type on $\partial\Om$ that is almost holomorphic in the sense that $\left|d\rho([\bar\partial, v_\varepsilon])\right|< \varepsilon$ and derived the exact regularity of $N_q$ on $\Omega$. Boas-Staube's vector field method has many geometric consequences. For example,   Boas and Straube were able to deduce the exact regularity of $N_q$ via the cohomology of the submanifold in $\partial\Om$ consisting of infinite type points \cite{BS93}. Straube later introduced a sufficient condition for the exact regularity of $N_q$ by unifying the vector field
method in Boas-Straube and the compactness estimates \cite{S10}. 
In \cite{Kohn99}, Kohn discovered a profound relation between 
 the exact regularity of $N_q$ and the  Diederich–Fornæss index and his result was later improved in \cite{PZ14, L22}.
More recently, Harrington introduced in \cite{H11} another family of vector fields $v_\varepsilon$ of $(1, 0)$-type and a family of plurisubharmonic functions $\lambda_\varepsilon$ with self-bounded gradient, called a family 
of transverse vector fields satisfying property ($\widetilde{P}$),  such that the complex Hessian of  $\lambda_\varepsilon$  controls the $(0,1)$-form $\theta ^{v_{\varepsilon}}$ associated to $v_\varepsilon$ on $\partial\Om$. Harrington further obtained a fundamental result that if $\Omega$ satisfies his vector field condition, then $N_q$ is exactly regular \cite{H11}. 
Since Harrington's condition does not assume that $v_\varepsilon (\rho)$ is uniformly bounded from below by a positive constant, the result generalizes the  exact regularity result of $N_q$ in \cite{BS91} and also improve Kohn's result  in \cite{Kohn99} on the Diederich–Fornæss index. A more precise
conjecture states that if the Diederich–Fornæss index of $\Om \subset \mathbb{C}^n$ is 1, then $N_q$ is exactly regular.  Very recently, this conjecture was proved by Liu and Straube for $n=2$, and for general $n$, they also made substantial progress \cite{LS25}. 
 For other related results on the regularity of the $\bar\partial$-Neumann problem, the interested reader may refer to \cite{BS99, HM06, AH19, G19, HL20, CH21, Z24, S25} and references therein. 

\medskip

% applies to the case where the domain has certain bounded plurisubharmonic exhaustion function. 
%One important step in Harrington's argument is to estimate the commutators $[\bar\partial^*, v_\varepsilon], [\bar\partial, v_\varepsilon]$ using  $\theta ^{v_{\varepsilon}}$.  
% applied the duality argument to derive the Sobolev estimate $$\|u\|_k \lesssim \|\bar\partial u\|_k + \|\bar\partial^* u\|_k.$$
%which avoids the method of the integration by parts in \cite{BS91}. Therefore, 
%In particular, there is no need to assume that the normal component of $\lambda_\varepsilon$ is uniformly bounded from below by a positive constant as in \cite{BS91}. 

In this paper, inspired by % property ($P^{\#}_q$) or 
property ($\widetilde{P}^{\#}$) introduced by Yue Zhang in \cite{YZ2021} and Harrington's vector fields condition in \cite{H11}, we define a family of vector fields $v_{\alpha, \varepsilon}$ satisfying %property ($P^{\#}_q$) or 
property ($\widetilde{P}^{\#}$). 
In order to derive the exact regularity, we use the method of elliptic regularization (cf. \cite{KN65, S10}) and the key is to estimate the commutators $[\bar\partial^*, v_{\alpha, \varepsilon}], [\bar\partial, v_{\alpha, \varepsilon}]$. Using Harrington's strategy in \cite{H11},  the commutators can be controlled using $\theta^{v_{\alpha, \varepsilon}}$ via another operator $D_{X_{\alpha, \varepsilon}}$, %transformed to $[\bar\partial^*, D_{X_{\alpha, \varepsilon}}], [\bar\partial, D_{X_{\alpha, \varepsilon}}]$ and thus be controlled via $\theta^{v_{\alpha, \varepsilon}}$, 
so that our unweighted $L^2$ estimates are applicable.  
The reader may refer to \autoref{vf} for notations. 

%By following Harrington's strategy in \cite{H11}, we obtain the following sufficient condition for the $\bar\partial$-Neumann problem. 

\begin{thm}   
  Let $n > 2$ and $\Om \subset \bbC ^{n}$ be a smooth bounded pseudoconvex domain 
 that possesses a family of transverse vector fields satisfying property strong  (${\widetilde P}_{q}^{\#}$) %or property ($P _{q} ^{\#}$) 
 with $1 < q \leq n-1$. 
  Then the $\bar{\partial}$-Neumann operator $N _{q'}$ is continuous on $W _{\(0, q'\)}  ^{k} \(\Om\)$ for any $k \geq 0$ and any $q'\geq q$.
 \end{thm}

 \begin{thm}   
  Let $n > 2$ and $\Om \subset \bbC ^{n}$ be a smooth bounded pseudoconvex domain 
 that possesses a family of transverse vector fields satisfying property (${\widetilde P}_{n-1}^{\#}$). %or property ($P _{q} ^{\#}$) 
  Then the $\bar{\partial}$-Neumann operator $N _{n-1}$ is continuous on $W _{\(0, n-1\)}  ^{k} \(\Om\)$ for any $k \geq 0$. %and any $q'\geq q$.
 \end{thm}
 
 We would like to point out that %a particularly interesting case is $q=n-1$ in the above result, where 
 if $\Omega$ satisfies Yue Zhang's %property ($P^{\#}_{n-1}$) or 
 property ($\widetilde{P}^{\#}_{q}$) (resp. property ($\widetilde{P}^{\#}_{n-1}$)), then $\Omega$ possesses
 a family of vector fields satisfying %property ($P^{\#}_{n-1}$) or 
 property strong ($\widetilde{P}^{\#}_{q}$) (resp. property ($\widetilde{P}^{\#}_{n-1}$)). %can be implied by Yue Zhang's %property ($P^{\#}_{n-1}$) or  property ($\widetilde{P}^{\#}_{n-1}$). 
Moreover, the condition  of vector fields satisfying property ($\widetilde{P}^{\#}_{n-1}$) is preserved under biholomorphisms in the following sense: if $\Omega$ possesses
 a family of vector fields satisfying property ($\widetilde{P}^{\#}_{n-1}$) and $F: \Omega \rightarrow \Omega'$ is a biholomorphism that extends to a CR diffeomorphism from $\partial\Omega$ to $\partial\Omega'$, then 
 $\Omega'$ also possesses
 a family of vector fields satisfying property ($\widetilde{P}^{\#}_{n-1}$).
 
 \medskip

 %A deep consequence of the aforementioned Harrington's sufficient condition in \cite{H11} is to improve Kohn's result in \cite{Kohn99}.  Moreover, 
 \medskip

{\bf Notation:} Throughout the paper, we use $Q_1 \leq C_\varepsilon Q_2$ to indicate that the positive constant $C_\varepsilon$ depends on the parameter $\varepsilon$. If the positive constant is independent of parameters $\varepsilon, \delta$, we simply denote the estimate by %$Q_1 \leq C Q_2$ or 
$Q_1 \lesssim Q_2$. 
 
 %   if the Diederich-Fornæss index of the domain is 1 on $\Omega$ in $\mathbb{C}^2$, then $N_1$ is exactly regular. 

%Motivated by Harrington's insights, the objective of this paper is to combine the approaches in \cite{YZ2021}, \cite{BS91}, and \cite{H11} to derive sufficient conditions for exact regularity of $\bar{\partial}$-Neumann. The paper is structured as follows. In Section 2, we will introduce the necessary background on the $\bar{\partial}$-Neumann operator, key definitions, and some useful estimates. Section 3 will derive basic estimates. Section 4 and section 5 will prove the exact regularity $N _{q}$ when the domain satisfies property $\widetilde{P} ^{\#}$ in the paper. Finally, Section 6 concludes domain satisfy property $P ^{\#}$ then $N _{q}$ is exact regularity.

\section{Preliminaries and classical estimates}\label{prel}

\subsection{Preliminaries}
Let $\Om$ be a bounded pseudoconvex domain in $\bbC ^{n}$ %. %with boundary $\partial \Om$. $\Om$ is smooth if there exists
and $\rho$ be a smooth defining function satisfying $\Om = \{z: \rho < 0 \}$ and $d\rho \neq 0$ on the boundary $\partial \Om$. 
For $\(0, q\)$-forms $u, v \in L _{\(0, q\)} ^{2} \(\Om\)$ with $u = \sum' _{J} u _{J} d \bar{z} _{J}, v = \sum' _{I} v _{I}d\bar{z} _{I}$,
where the summations are taken over 
$I$ and $J$ in increasing order, with antisymmetric coefficients $v_I$ and $u_J$, 
 the $L^2$ inner product is defined as 
\begin{align*}
  \(u, v\) 
  = \sideset{}{'}{\sum} _{J}\(u _{J} d \bar{z} _{J}, v _{J} d \bar{z} _{J}\) 
  = \sideset{}{'}{\sum} _{J} \int _{\Om} u _{J} \bar{v} _{J}dV.
\end{align*}
%Let $\lbk\chi _{j} \rbk_{j = 0} ^{N}$ be a partition of unity on $\cl{\Om}$., such that $\chi _{0}$ is compactly supported in $\Om$ and each $\chi _{j}$ is compactly supported in $U _{j} for j \geq 1$. 
The corresponding $L ^{2}$-norm is given by  
%\begin{align*}
 $ \lnm u \rnm ^{2} =  \(u, u\).$
%  \lnm \sideset{}{'}{\sum} _{J} u _{J} \bar{\omega} _{J}\rnm ^{2} = \sideset{}{'}{\sum} _{J} \int _{\Om} |u _{J}| ^{2} dV.
%\end{align*}
For any $u \in W _{\(0,q\)}^{k} \(\Om\)$, the $L^2$-Sobolev norm is defined by 
$$\lnm u \rnm _{k} ^{2} = \sideset{}{'}{\sum}_{J}\sum _{|\beta| = 0} ^{k} \int _{\Om} \left|\frac{\partial^{|\beta|} u _{J}}{\partial x_1^{\beta_1} \cdots \partial x_{2n}^{\beta_{2n}}}  \right|^2 dV,$$
%where $\nabla^i =\frac{\partial^{|i|}}{\partial x_1^{i_1} \cdots \partial x_{2n}^{i_{2n}}}$ denotes the real differential operator 
with $\beta=(\beta_1, \cdots, \beta_{2n})$.  The $\bar{\partial}$-operator is given by
\begin{align*}
  \bar{\partial}u = \sideset{}{'}\sum_{|J| = q}  \sum_{j = 1}^{n} \frac{\partial u _{J}}{\partial \bar{z}_{j}} d\bar{z}_{j} \wedge d\bar{z}_{J},
\end{align*}
and the domain of $\bar{\partial}$ is
\begin{align*}
  \dom \(\bar{\partial}\) = \lbk u \in L _{\(0, q\)} ^{2} \(\Om\): \bar{\partial} u \in L _{\(0, q + 1\)} ^{2} \(\Om\)\rbk.
\end{align*}
Let $\vartheta$ be the formal adjoint of $\bar{\partial}$ 
and the formal adjoint $\vartheta$ can be calculated as 
\begin{align*}
  \vartheta u = - \sideset{}{'}\sum _{|K| = q-1}\(\sum _{j = 1} ^{n} \frac{\partial u _{jK}}{\partial z_{j}}\) d\bar{z} _{K} .
\end{align*}
The Hilbert space adjoint of $\bar{\partial}$ is denoted by $\bar{\partial} ^{*}$, and its domain is defined by
\begin{align*}
  \dom \(\bar{\partial} ^{*}\) = \lbk v \in L _{\(0, q + 1\)} ^{2} \(\Om\) :  \left| \(\bar{\partial} u, v\) \right| \lesssim  \lnm u \rnm {\rm ~for~ all~} u \in \dom \(\bar{\partial} \)\rbk.
\end{align*}
For $1 \leq q \leq n$, the Kohn Laplacian $\square _{q}$ is defined by $\square _{q} = \bar{\partial} \bar{\partial} ^{*} + \bar{\partial} ^{*} \bar{\partial} : L _{\(0, q\)} ^{2} \(\Om\) \to L _{\(0, q\)} ^{2} \(\Om\)$ with %domain %of $\square _{q}$ is
\begin{align*}
  \dom\(\square _{q}\) = \lbk u \in \dom \(\bar{\partial}\) \cap \dom \(\bar{\partial} ^{*}\) \cap L _{\(0, q\)} ^{2} \(\Om\):  \bar{\partial} u \in\dom \(\bar{\partial} ^{*}\), \bar{\partial} ^{*} u \in \dom \(\bar{\partial}\)\rbk.
\end{align*}
It is the standard fact in the $\bar\partial$-Neumann theory that $\square _{q}$ is a self-adjoint and surjective, and it has a self-adjoint bounded inverse, called the $\bar{\partial}$-Neumann operator $N _{q}$. 

% Now, we will introduce these operators in special boundary charts. 
Recall the special boundary charts.
For each $0\leq \alpha \leq N$, let $U_\alpha$ be an open set in $\mathbb{C}^n$ such that $\lbk U_{\alpha} \rbk_{\alpha = 0} ^{N}$ is an open covering of $\overline{\Om}$ with $U_{0} \subset \subset  \Om$. %Let $\{\chi_\alpha\}_{\alpha = 0} ^{N}$ be a family of smooth cut-off functions subordinate to $\lbk U_{\alpha} \rbk_{\alpha = 0} ^{N}$ such that each $\chi_\alpha$ is compactly supported in some $U_\beta$ with $\chi_\alpha$ is identically equal to 1 on $V_\alpha \subset\subset U_\beta$ and $\sum_{\alpha=0}^N \chi_\alpha \equiv 1$ on $\overline{\Om}$. 
 Define $L _{n} = \sum _{k = 1} ^{n} \frac{1}{|\partial \rho|^{2}} \frac{\partial \rho}{\partial\bar{z} _{k}} \frac{\partial}{\partial z _{k}}$ to be the complex normal vector of $\(1, 0\)$-type  with unit Euclidean norm.
For any $U_{\alpha}$ with $1 \leq \alpha \leq N$, define vector fields $L _{1}, \cdots, L _{n - 1}$ of $\(1, 0\)$-type  in $U _{\alpha}$ which are orthonormal with respect to the Euclidean norm and span the holomorphic tangent space $T^{(1, 0)} _{z}  \partial \Om $ for any $z \in \partial\Omega \cap U_\alpha$. %where $\Om _{\varepsilon} = \{z \in \Om| \rho\(z\) < - \varepsilon \}$ 
% with $\delta$ sufficiently close to 0.
 It follows that $\{L _{1}, \cdots, L _{n - 1}, L _{n}\}$ forms an orthogonal basis of $T^{(1, 0)} _{z} \mathbb{C}^n$ for any $z\in U _{\alpha}$. Let $\omega _{1}, \omega _{2}, \cdots, \omega _{n}$ be the $\(0, 1\)$-forms dual to $L _{1}, L _{2}, \cdots, L _{n - 1}, L _{n}$ satisfying $\omega _{j} \(L _{k}\) = \delta _{jk}$. Then $\lbk \omega  _{j}\rbk _{j = 1} ^{n}$ forms an orthogonal basis of $\Lambda^{\(1, 0\)}_{z} \mathbb{C}^n$ for any $z\in U _{\alpha}$. %which are called a special boundary frame. 
These moving frames are referred to the special boundary chart as in \cite{FK1972}.
% For a $(0, q)$-form $u$, write $u= \sum' _{|J|=q} u _{J}^{\dagger} \bar{\omega} _{J} = \sum' _{|K|=q} u _{K} d \bar{z} _{K}$ with respect to the special boundary chart and Euclidean coordinates respectively.
For a $(0, q)$-form $u$, write $u= \sum' _{J} u _{J}^{\dagger} \bar{\omega} _{J} $ with respect to the special boundary chart.
Let $\phi$ be a smooth function on $\overline{\Om}$. Using special boundary charts, write 
$$\partial\bar\partial \phi =\sum _{i,j = 1} ^{n} \phi^{\dagger} _{ij} \omega_i \wedge \bar\omega_j.$$ 
For an $(0, q)$-form $u$, denote the tangential part of $u$ by
$u_{\text {Tan }}= \sum' _{n\notin J} u_{J}^{\dagger} \bar{\omega} _{J}$,  the normal part of  $u$ by  $u_{\text{Norm}} = \sum' _{n \in J} u _{J}^{\dagger} \bar{\omega} _{J} = \sum' _{|K|= q-1} u _{Kn}^{\dagger} \bar{\omega}_{K} \wedge \bar{\omega}_{n}$ and the normal component of $u$ by $u_{\text {norm }} = \sum' _{|K|= q-1}  u _{Kn}^{\dagger} \bar{\omega}_{K}$. It follows that $u_{\text{Norm}} = u_{\text {norm }} \wedge \bar{\omega}_{n}$ and $u_{\text {Tan }} \wedge \bar{\omega}_n = u \wedge \bar{\omega}_n$. It is easy to verify that $u \in \dom \(\bar{\partial} ^{*}\)$ if and only if $u^{\dagger} _{J} = 0$ on $\partial \Om$ when $n \in J$, which is also equivalent to $u_{\text{Norm}}=0$ on $\partial \Om$.

Let $u = \sum' _{J} u _{J}^{\dagger} \bar{\omega} _{J} =  \sum' _{K} u _{K} d\bar{z} _{K}$ be a smooth $(0, q)$-form with compact support in some $V_\alpha$ satisfying $\bar{\omega} _{J} = \sum'_K \epsilon _{J} ^{K} d\bar{z} _{K}$ for smooth functions $ \epsilon _{J} ^{K}$. For a vector field $X$, define $\nabla^{(b)}_{X} u = \sum' _{J}\(Xu_{J}^{\dagger}\) \bar{\omega} _{J}$  and $ \nabla_{X}u =X u = \sum' _{K}\(X u _{K}\)d\bar{z} _{K}$. 
By the straightforward calculation, 
  \begin{align*}
    \nabla^{(b)}_{X} u 
 %   &= \sideset{}{'}\sum_{J}\(Xu_{J}^{\dagger}\) \bar{\omega} _{J}
  %  = \sideset{}{'}\sum_{J}\(Xu_{J}^{\dagger}\) \varepsilon _{J} ^{K}d\bar{z} _{K}\\
  %  &= \sideset{}{'}\sum_{J} \(X\(u_{J}^{\dagger} \varepsilon _{J} ^{K}\)d\bar{z} _{K} - \(X \varepsilon _{J} ^{K}\) u_{J} ^{\dagger} \)d\bar{z} _{K}\\
    = \sideset{}{'}\sum_{J} \(X u_{K}\)d\bar{z} _{K} - \sideset{}{'}\sum_{J, K}\(X\epsilon _{J} ^{K}\) u_{J}^{\dagger}d\bar{z} _{K} =  \nabla_{X}u  - \sideset{}{'}\sum_{J, K}\(X\epsilon _{J} ^{K}\) u_{J}^{\dagger}d\bar{z} _{K}.
  \end{align*}
 In particular, we have

 %~~~~~~~~~~~~~~~~~~~~~~~~~~~~~~~~~~~~~~~~~~~~~~~~~~~~~~~~~~~~~~~~~~~~~~~~~~~~~~~~~~~~~~~~~~~~~~~~
 %\begin{comment}

  %It follows that 
 % \begin{align*}
 %   \left|   \nabla^{(b)}_{X} u\right| \lesssim\left| \nabla_{X}u \right| + \left| u \right|, ~~~~     \left|\nabla _{X} u\right| \lesssim \left|   \nabla^{(b)}_{X} u \right| + \left| u \right|.
%  \end{align*}
%  Moreover, for vector fields $X_{1}, \cdots, X_{k}$, we have
%  \begin{align*}%\label{change of coordinates}
 %  \left|   \nabla^{(b)}_{X_1} \cdots   \nabla^{(b)}_{X_{k}} u\right| 
 %   \lesssim\left|    \nabla_{X_1}   \cdots   \nabla_{X _{k}}u \right| +  \sum_{s=0}^{k-1} \left|\nabla^s u\right| ,\\ %\sum_{\substack{i_{1}+ \cdots + i_{2n} = s \\ |s|\leq k - 1}}\left| \frac{\partial ^{s}u}{\partial x _{1} ^{i_1} \cdots x _{2n} ^{i_{2n}}} \right|, \\
  %  \left|\nabla _{X _{1}} \cdots \nabla _{X _{k}}u\right| 
 %   \lesssim\left|    \nabla^{(b)}_{X_1}   \cdots   \nabla^{(b)}_{X_{k}} u \right| +   \sum_{s=0}^{k-1} \left|\nabla^s u\right| , %\sum_{\substack{i_{1}+ \cdots + i_{2n} = s \\ |s|\leq k - 1}}\left| \frac{\partial ^{s}u}{\partial x _{1} ^{i_1} \cdots x _{2n} ^{i_{2n}}} \right|.\nonumber
 % \end{align*}
%  where $\nabla^s u$ denotes all derivatives of $u$ of order $s$.

%  Therefore, the Sobolev norm of special boundary charts is comparable to the Sobolev norm of Euclidean coordinates with an error term of the lower-order derivatives.
 
% \end{comment} 
 %~~~~~~~~~~~~~~~~~~~~~~~~~~~~~~~~~~~~~~~~~~~~~~~~~~~~~~~~~~~~~~~~~~~~~~~~~~~~~~~~~~~~~~~~~~~~~~~~

  \begin{coro}\label{s.b.c comparable}
  If $u \in  \dom\(\bar{\partial} ^{*}\) \cap C^\infty_{(0, q)}(\overline{\Omega})$ and $T$ is a tangential vector field, then $\nabla^{(b)} _{T} u \in \dom\(\bar{\partial} ^{*}\)$ and $\nabla^{(b)} _{T} u = \nabla _{T} u +0$-th order term of $u$. 
  %  Let $T$ is tangential vector field, define $\nabla _{T} u = \sum' _{J} Tu_{J} d \bar{z} _{J}$, then we have
 % \begin{align}\label{preserve dom}
  %  \lnm \nabla _{T} u \rnm _{k} \leq \lnm Tu \rnm _{k} + \lnm u \rnm _{k},
  %\end{align} 
 % and $Tu \in \dom\(\bar{\partial} ^{*}\)$.
  \end{coro}

We now briefly recall the method of elliptic regularization due to Kohn-Nirenberg \cite{KN65} as follows. The interested reader may refer to \cite{KN65, T23, S08} for detailed discussions.  Let $u \in W _{\(0, q\)} ^{1} \(\Om\) \cap \dom \(\bar{\partial} ^{*}\)$.
For $\delta > 0$, consider the   quadratic form $Q _{\delta}$ defined by 
   \begin{align*}
     Q _{\delta} \(u, u\) = \lnm \bar{\partial} u \rnm ^{2} + \lnm \bar{\partial} ^{*} u \rnm ^{2} + \delta\lnm \nabla u \rnm ^{2},
   \end{align*} 
where $\nabla u$ denotes all first order derivatives of all coefficients of $u$. By the classical theory in functional analysis, there exists a unique self-adjoint operator $\square_{\delta, q}$  with 
\begin{align*}
  \left(\square_{\delta, q} u, v\right)=Q_\delta(u, v), \quad u \in \operatorname{dom}\left(\square_{\delta, q}\right), v \in W_{(0, q)}^1(\Omega) \cap \operatorname{dom}\left(\bar{\partial}^*\right) .
\end{align*}
It is actually an elliptic problem to invert $\square_{\delta, q}$ and the inverse operator $N_{\delta, q}$ gains two derivatives in Sobolev regularity.  As a consequence,  if $u \in C_{(0, q)}^{\infty}(\cl{\Omega})$, then $N_{\delta, q} u \in C_{(0, q)}^{\infty}(\cl{\Omega})$. It follows from the straightforward calculation that 
\begin{align*}
  \square_{\delta, q} u=-\left(\frac{1}{4}+\delta\right) \Delta u, \quad u \in \operatorname{dom}\left(\square_{\delta, q}\right) .
\end{align*}
Additionally, $u \in C_{(0, q)}^2(\cl{\Omega})$ belongs to $\operatorname{dom}\left(\square_{\delta, q}\right)$ if and only if $u \in \operatorname{dom}\left(\bar{\partial}^*\right)$ and
\begin{align*}
  (\bar{\partial} u)_{\text {norm }}+\delta \left( \frac{\partial u}{ \partial \nu} \right)_{\operatorname{Tan}}=0 \quad \text { on } \partial \Omega,
\end{align*}
where $\frac{\partial u}{ \partial \nu}$ denotes the outward normal componentwise differentiation on $u$  in the Euclidean coordinates. %This is equivalent to the wedge product $\(\bar{\partial} u\)_{\text {norm }} \wedge \bar{\omega}_{n}$, where $\(0, q\)$-form $(\bar{\partial} u)_{\text {norm }} = \sum' _{|K|= q} \(\bar{\partial} u\) _{Kn}^{\dagger} \bar{\omega}_{K}$ is the normal component of $\bar{\partial} u$. Thus, the above boundary condition for $\square_{\delta, q}$ gives the modified form 
It thus follows that 
$\bar{\partial} N _{\delta, q} u + \delta \frac{\partial N _{\delta, q} u}{ \partial \nu} \wedge \bar w_n \in \rm{dom}(\bar{\partial} ^{*})$ for any $u \in C^\infty_{(0, q)}(\overline{\Omega})$.
% \begin{align*}
%   u_{\text{norm}}:=\sum _{|K|=q-1}\(\sum _{j = 1} ^{n} \frac{\partial \rho}{\partial z_{j}}u_{jK}\) d\bar{z}_{K}.
% \end{align*}
% and $(\partial u / \partial \nu)_{\text {Tan }}$ denotes the tangential part of $(\partial u / \partial \nu)$.

\subsection{Unified $L^2$ estimates}

We first recall the following unified estimates obtained by Yue Zhang (see Corollary 3.3  in \cite{YZ2021}). 

\begin{thm}\label{basic estimate thm}
  Let $\Om$ be a smooth bounded domain in $\bbC ^{n}$ with a defining function $\rho$  and $U$ be a small open neighborhood of some boundary point. Suppose $g, \phi \in C ^{2}\(\cl{\Om}\)$ and $ g > 0$. 
  For $\gamma >0$ and $0 < \eta < \frac{1}{2}$, and for every ordered index set $I _s = \{j _k: 1 \leq k \leq s\} \subset \{1,2, \cdots, n-1\}$ and $J_s = \{1,2, \cdots, n\} \setminus I_s$ with $1\leq s\leq n-1$, there is a constant $C _{\eta, \gamma} > 0$ independent of $u$, $\phi$ and $g$, such that:
  \begin{align*}%\label{basic inequality}
    &\lnm \sqrt{g} \bar{\partial} u \rnm _{\phi}^2 + \(1 + \frac{1}{\gamma}\) \lnm \sqrt{g} \bar{\partial} _{\phi} ^{*} u \rnm _{\phi} ^{2} + C _{\eta, \gamma} \lnm \sqrt{g} u \rnm _{\phi} ^{2}\\\nonumber
    \geq & - \(\gamma + \frac{1}{\eta}\) \sideset{}{'}{\sum }_{\lmdl J \rmdl = q} \sum _{j \leq n} \lnm \frac{1}{\sqrt{g}} \(L _{j} g\) u_J^{\dagger} \rnm _{\phi} ^2 + \eta \sideset{}{'}{\sum}_{\lmdl J \rmdl = q} \( \sum_{j \in J_s} \lnm \sqrt{g} \cl{L}_j u _J^{\dagger} \rnm _{\phi}^2 + \sum_{j \in I_s} \lnm \sqrt{g} \delta_{w_j} u_J^{\dagger} \rnm _\phi^2\)\\\nonumber
    & + \sideset{}{'}{\sum} _{\lmdl J \rmdl = q} \sum_{j \in I_s} \(\(\(g^{\dagger}  _{jj} - g \phi ^{\dagger} _{jj}\) u_J^{\dagger}, u_J^{\dagger} \)_\phi - \int_{\partial \Om} g \rho^{\dagger} _{jj} u_J^{\dagger} \cl{u} _J^{\dagger} e^{-\phi} \,dS\)\\\nonumber
    & + \sideset{}{'}{\sum}_{\lmdl K \rmdl = q - 1}\sum_{i,j = 1} ^{n} \(\(\(-g^{\dagger} _{ij} + g \phi^{\dagger} _{ij}\) u _{iK}^{\dagger}, u _{jK}^{\dagger} \) _\phi + \int _{\partial \Om} g \rho^{\dagger} _{ij} u_{iK}^{\dagger} \cl{u }_{jK}^{\dagger} e^{-\phi} \,dS\)\nonumber
  \end{align*}
  holds  for all $ u = \sideset{}{'}{\sum} _J u _J^{\dagger} \bar{w}_J \in C _{\(0, q\)}^ \infty \(\cl{\Om} \) \cap \dom\(\bar{\partial}^*\)$
  with support in $\cl{\Om} \cap U$.
\end{thm}

For the boundary term, we can use Lemma 3.5 in \cite{YZ2021}.

\begin{lem}\label{boundary term}
  Let $\Om$ be a bounded pseudoconvex domain in $\bbC ^n$ with $ n > 2$ and $\rho$ be the smooth defining function and let $U$ be a small open neighborhood of some boundary point. For a fixed $s$ with $1 \leq s \leq n-2$, there exist $ q _0 $ with $s+1 \leq q _0 \leq n-1$ and an ordered index set $I _s = \{j_k: 1 \leq k \leq s \} \subset \{1,2,\cdots, n-1 \}$ such that   \begin{align}\label{boundary estimate}
    \sideset{}{'}{\sum} _{\lmdl K \rmdl = q - 1} \sum_{i,j = 1}^{n} \int_{\partial \Om} g \rho^{\dagger} _{ij} u _{iK}^{\dagger} \cl{u}_{jK}^{\dagger} e ^{- \phi} \, dS - \sideset{}{'}{\sum}_{\lmdl J \rmdl = q} \sum_{j \in I_s} \int _{\partial \Om} g \rho^{\dagger} _{jj} \lmdl u_J^{\dagger} \rmdl ^2  e ^{-\phi} \,dS \geq 0
  \end{align}
 holds for any $u \in C _{\(0, q\)} ^\infty \(\cl{\Om}\) \cap \dom \(\bar{\partial}^*\)$ supported in $\cl{\Om} \cap U$ with $n-1\geq q \geq q_0$.  
  In particular, if $s = n - 2$, then $q_0 = n-1$ and $I_s = \{1,2,\cdots, n-1 \} \setminus \{ t \}$, for any $1 \leq t \leq n-1$.
\end{lem}

\begin{proof}
Since  $u_{\rm{Norm}} = 0$ on $\partial\Om$, then 
 $$  \sideset{}{'}{\sum} _{\lmdl K \rmdl = q - 1} \sum_{i,j = 1}^{n}  \rho^{\dagger} _{ij} u _{iK}^{\dagger} \cl{u}_{jK}^{\dagger} \geq M \sideset{}{'}{\sum}_{\lmdl J \rmdl =q} \lmdl u_J^{\dagger} \rmdl ^2 $$ holds on $\partial\Om$  if and only if the sum of the smallest $q$ eigenvalues of the Hessian matrix $ \left( \rho^{\dagger} _{ij}\right)$  restricted to $T^{(1, 0)}\partial \Om$ is at least $M$ (cf. Lemma 4.7 in \cite{S10}). Because $\partial\bar\partial \rho$ is non-negative restricted to $T^{(1, 0)}\partial \Om$, for any $s$ with $1 \leq s \leq n-2$, $q_0 = n-1$ will enable (\ref{boundary estimate}) to hold.  
%(\ref{boundary estimate}) holds for every $u \in  C _{\(0, q+1\)}^ \infty \(\cl{\Om_\delta} \) \cap \dom\(\bar{\partial}^*\)$ by the same reason. Therefore,  $\Om$ possesses a family of transverse vector fields satisfying property ($\widetilde{P} _{q+1} ^{\#}$)
\end{proof}

By choosing appropriate $g, \phi$ in \autoref{basic estimate thm} and applying (\ref{boundary estimate}), we have the following unweighted $L^2$ estimates. 
 %Since the constants $\gamma$ and $\eta$ are constant and independent of our estimate, we may omit these constants in $\lesssim$.

\begin{prop}\label{modify basic estimate}
  Let $\Om$ be a smooth bounded pseudoconvex domain in $\bbC ^n$ with $n > 2$ and $U$ be a small open neighborhood of some boundary point. %and $\rho$ be the smooth defining function. 
  For a fixed $s$ with $1 \leq s \leq n-2$, let $q_0$ and the ordered index set $I_s$ be defined as in \autoref{boundary term}, such that the inequality (\ref{boundary estimate}) holds. Suppose $\lambda \in C^2 (\cl{\Omega})$. Then for any $\gamma >0$ and $0< \eta <\frac{1}{2}$, there exists $C\(\gamma,\eta,n\)>0$ such that:
  \begin{align*}%\label{duan}
    \lnm \bar{\partial} u \rnm^2 + \lnm \bar{\partial}^* u  \rnm^2 %\\\nonumber %+ \lnm u \rnm^2 
    \gtrsim 
    - C\(\gamma,\eta,n\) \sideset{}{'}{\sum} _{\lmdl J \rmdl = q}\sum_{j\leq n} \lnm \(L _j \lambda\) u_J^{\dagger} \rnm^2 
    + \sideset{}{'}{\sum} _{\lmdl J \rmdl = q} \sum_{j \in I_s} \(  - 2 \lambda^{\dagger} _{jj}  u_J^{\dagger},  u_J^{\dagger} \) + \sideset{}{'}{\sum}_{\lmdl K \rmdl = q - 1} \sum_{i,j = 1} ^{n} \(  2 \lambda^{\dagger} _{ij}  u_{iK}^{\dagger}, u_{jK}^{\dagger} \)
  \end{align*}
  holds for all $u = \sum'_J u_J^{\dagger} \bar{w}_J \in C_{\(0, q\)}^\infty \(\cl{\Om}\) \cap \dom (\bar{\partial}^*)$ with support in $\cl{\Om} \cap U$.
\end{prop}

\begin{proof}
  By (\ref{boundary estimate}), the boundary term in \autoref{basic estimate thm} is non-negative. %We also fix the constants $\gamma$ and $\eta$.
  Given  $\lambda \in C^2(\overline{\Om})$, let $g = e ^{-\lambda}$ and $\phi = \lambda$. Then for any $f \in C_{\(0, q\)} ^\infty \(\cl{\Om}\) \cap \dom (\bar{\partial}^*)$, it follows from \autoref{basic estimate thm} that
%
  %For any $u \in C_{\(0, q\)} ^\infty \(\cl{\Om}\) \cap \dom (\bar{\partial}^*)$, we denote $f = e ^{\lambda} u$. Thus, $f$ is also in $C_{\(0, q\)} ^\infty \(\cl{\Om}\) \cap \dom (\bar{\partial}^*)$, the following inequality holds.
%
  \begin{align*}
    &\lnm e^{- \frac{\lambda}{2}} \bar{\partial} f \rnm_\lambda^2 + \lnm e^{- \frac{\lambda}{2}} \bar{\partial}_\lambda^* f \rnm_\lambda^2 + \lnm e^{- \frac{\lambda}{2}} f \rnm_\lambda^2 \\
    \gtrsim & - C\(\gamma, \eta\) \sideset{}{'}{\sum}_{\lmdl J \rmdl = q} \sum_{j\leq n} \lnm e^{\frac{\lambda}{2}} e^{- \lambda} (L_j \lambda) f_J^{\dagger} \rnm_\lambda^2 \\
    & + \sideset{}{'}{\sum} _{\lmdl J \rmdl = q} \sum_{j \in I_s} \(\( e^{-\lambda} \lmdl L_j \lambda \rmdl^2 - e^{- \lambda} \lambda^{\dagger} _{jj} - e^{- \lambda} \lambda^{\dagger} _{jj} \) f_J^{\dagger}, f_J^{\dagger}\)_\lambda\\
    & + \sideset{}{'}{\sum }_{\lmdl K \rmdl = q - 1} \sum_{i,j = 1} ^{n} \(\( -e^{-\lambda}  (L_i \lambda )(\bar{L} _j \lambda)  + e^{- \lambda} \lambda^{\dagger} _{ij} + e^{- \lambda} \lambda^{\dagger} _{ij} \) f_{iK}^{\dagger},f_{jK}^{\dagger} \)_\lambda.
  \end{align*}
  Thus, we have
%
  % \begin{align*}
  %   &\lnm e^{- \frac{\lambda}{2}} \bar{\partial} f \rnm_\lambda^2 + \lnm e^{- \frac{\lambda}{2}} \bar{\partial}_\lambda^* f \rnm_\lambda^2 + \lnm e^{- \frac{\lambda}{2}} f \rnm_\lambda^2 \\
  %   \gtrsim & - C\(\gamma, \eta\) \sideset{}{'}{\sum} _{\lmdl J \rmdl = q}\sum_{j\leq n} \lnm e^{-\frac{\lambda}{2}} (L_j \lambda)画 f_J \rnm_\lambda^2 \\
  %   & + \sideset{}{'}{\sum} _{\lmdl J \rmdl = q} \sum_{j \in I_s} \(e^{-\lambda} \(\lmdl L_j \lambda \rmdl^2 - 2 \lambda_{jj}\) f_J, f_J\)_\lambda\\
  %   & + \sideset{}{'}{\sum} _{\lmdl K \rmdl = q - 1} \sum_{i,j = 1} ^{n} \( e^{-\lambda}\(- L_i \lambda \bar{L} _j \lambda + 2\lambda_{ij}\) f_{iK},f_{jK} \)_\lambda.
  % \end{align*}
%
  % Then we times $e ^{- \lambda}$ in both sides:
%
  \begin{align*}
   &\lnm e ^{- \lambda} \bar{\partial} f \rnm ^2 + \lnm e ^{- \lambda} \bar{\partial}_\lambda^* f \rnm ^2 + \lnm e ^{- \lambda} f \rnm ^2 \\
    \gtrsim & - C\(\gamma, \eta\) \sideset{}{'}{\sum} _{\lmdl J \rmdl = q} \sum_{j\leq n} \lnm e^{-\lambda} \(L_j \lambda\) f_J^{\dagger} \rnm ^2 
    + \sideset{}{'}{\sum} _{\lmdl J \rmdl = q} \sum_{j \in I_s} \(e^{-2\lambda} \(\lmdl L_j \lambda \rmdl^2 - 2 \lambda^{\dagger} _{jj}\) f_J^{\dagger}, f_J^{\dagger}\)\\
    & + \sideset{}{'}{\sum} _{\lmdl K \rmdl = q - 1} \sum_{i,j =1}^{n} \( e^{-2\lambda}\(- (L_i \lambda)( \bar{L} _j \lambda )+ 2\lambda^{\dagger} _{ij}\) f_{iK}^{\dagger},f_{jK}^{\dagger} \).
  \end{align*}
   Moreover, the above inequality holds for any $f \in L^2_{\(0, q\)}(\Om) \cap \dom (\bar{\partial}^*) \cap \dom (\bar{\partial})$ by the density lemma (cf. Proposition 2.3 in \cite{S10}).
   
  For any $u \in L ^{2} \(\cl{\Om}\) \cap \dom (\bar{\partial}^*)$, let $f = e^\lambda u$. %since $e^\lambda$ is a %fixed smooth function in $\cl{\Om}$ and $e^\lambda$ will multiply %each component of $u$, so that we still have that $e^\lambda u_J = 0$ %on $\partial \Om$, if $n \in J$ for all $J$. $f$ will have component %$e^\lambda u_J$.  
   Note that $\bar{\partial}_\lambda^* f = e^\lambda \bar{\partial}^* (e^{- \lambda} f)$, so $\bar{\partial}_\lambda^* ( e^\lambda u ) = e^\lambda \bar{\partial}^* u$. Then it follows %we have that
  \begin{align}\label{12222}
    &\lnm e ^{- \lambda} \bar{\partial} \(e^\lambda u\) \rnm ^2 + \lnm \bar{\partial}^{*} u \rnm ^2 + \lnm u \rnm ^2 \\\nonumber
    \gtrsim & - C\(\gamma, \eta\) \sideset{}{'}{\sum}_{\lmdl J \rmdl = q} \sum_{j\leq n} \lnm \(L_j \lambda\) u_J^{\dagger} \rnm ^2
    + \sideset{}{'}{\sum}_{\lmdl J \rmdl = q} \sum_{j \in I_s} \(\(\lmdl L_j \lambda \rmdl^2 - 2 \lambda^{\dagger} _{jj}\) u_J^{\dagger}, u_J^{\dagger}\)\\\nonumber
    & + \sideset{}{'}{\sum} _{\lmdl K \rmdl = q - 1} \sum_{i,j = 1} ^{n} \(\(- (L_i \lambda)( \bar{L} _j \lambda) + 2\lambda^{\dagger} _{ij}\) u_{iK}^{\dagger},u_{jK}^{\dagger} \).\nonumber
  \end{align}
%
  % So
%
  % \begin{align*}
  % \Vert & e^{-\lambda} \bar{\partial} ( e^\lambda u ) \Vert^2 + \Vert \bar{\partial}^* u  \Vert^2 + \Vert u \Vert^2 \\
  % & \gtrsim - C(\gamma,\eta) {\sum_{\vert J \vert = q}}^{'} \sum_{j\leq n} \Vert (L_j \lambda) u_J \Vert^2 \\
  % & + {\sum_{\vert J \vert = q}}^{'} \sum_{j \in I_s} (  ( \vert L_j \lambda \vert^2 - 2 \lambda_{jj} )  u_J,  u_J)\\
  % & + {\sum_{\vert K \vert = q - 1}}^{'} \sum_{i,j} ( ( - L_i \lambda \bar{L_j} \lambda  + 2 \lambda_{ij} )  u_{iK}, u_{jK} )
  % \end{align*}
Also, since
%  Now we caculate the first term in the above inequality,
  \begin{align}\label{2}
    \lnm e^{-\lambda} \bar{\partial} ( e^\lambda u ) \rnm^2 
    &= \lnm e^{-\lambda} \(e^\lambda \bar{\partial} u + e^\lambda \bar{\partial} \lambda \wedge u\) \rnm ^2 
     = \lnm \bar{\partial} u + \bar{\partial} \lambda \wedge u \rnm^2 \\
    &\leq 2 \lnm \bar{\partial} u \rnm^2 + 2 \lnm \bar{\partial} \lambda \wedge u \rnm^2 \nonumber \\
    &= 2 \lnm \bar{\partial} u \rnm^2 + 2 \lnm \sideset{}{'}{\sum}_{|J|=q} \sum_{j \leq n} (\bar{L}_j \lambda) u_J^{\dagger} \bar{w}_j \wedge \bar{w}_J \rnm^2 \nonumber \\
   & \lesssim C\(n\) \left(  \lnm \bar{\partial} u \rnm^2 + \sideset{}{'}{\sum} _{|J|=q} \sum_{j \leq n} \lnm \(L_j \lambda\) u_J ^{\dagger}\rnm^2 \right),\nonumber
  \end{align}
where  $C\(n\)$ is a positive constant depending on the special boundary chart,
%by combining inequalities (\ref{1}) and (\ref{2}), we have
(\ref{12222}) yields
  \begin{align*}
  \lnm \bar{\partial} u \rnm^2 + \lnm \bar{\partial}^{*} u \rnm ^2
     \gtrsim   &\lnm \bar{\partial} u \rnm^2 + \lnm \bar{\partial}^{*} u \rnm ^2 + \lnm u \rnm ^2 \\
    \gtrsim & - C\(\gamma, \eta, n\) \sideset{}{'}{\sum} _{\lmdl J \rmdl = q} \sum_{j\leq n} \lnm \(L_j \lambda\) u_J^{\dagger} \rnm ^2 
     + \sideset{}{'}{\sum} _{\lmdl J \rmdl = q} \sum_{j \in I_s} \(\(\lmdl L_j \lambda \rmdl^2 - 2 \lambda^{\dagger} _{jj}\) u_J^{\dagger}, u_J^{\dagger}\)\\
    & + \sideset{}{'}{\sum} _{\lmdl K \rmdl = q - 1} \sum_{i,j = 1} ^{n} \(\(- (L_i \lambda)( \bar L_j \lambda) + 2\lambda^{\dagger} _{ij}\) u_{iK}^{\dagger},u_{jK}^{\dagger} \).
  \end{align*}
Moreover, 
  \begin{align*}
    & \lmdl \sideset{}{'}{\sum}_{\lmdl K \rmdl = q - 1} \sum_{i,j = 1} ^{n} \(\(L_i \lambda\)\(\bar{L} _j \lambda\) u_{iK}^{\dagger}, u_{jK}^{\dagger}\) \rmdl 
    % =& \lmdl \int_\Om \sideset{}{'}{\sum} _{\lmdl K \rmdl = q - 1} \sum_{i,j} \(L_i \lambda\) u_{iK} \cl{(L_j \lambda) u_{jK} } \, dV \rmdl\\
    \lesssim  \sideset{}{'}{\sum}_{\lmdl K \rmdl = q - 1} \sum_{i,j = 1} ^{n} \( (L_i \lambda) u_{iK}^{\dagger} \bar{w}_K, (L_j \lambda) u_{jK}^{\dagger} \bar{w}_K \)\\
    \leq & \lnm \sideset{}{'}{\sum} _{\lmdl K \rmdl = q - 1} \sum_{i = 1} ^{n} (L_i \lambda) u_{iK} ^{\dagger}\bar{w}_K \rnm \lnm \sideset{}{'}{\sum} _{\lmdl K \rmdl = q - 1} \sum_{j = 1} ^{n} (L_j \lambda) u_{jK}^{\dagger} \bar{w}_K \rnm    \leq  \sideset{}{'}{\sum} _{\lmdl J \rmdl = q}\sum_{i \leq n} \lnm (L_i \lambda) u_{J}^{\dagger} \rnm^2
  \end{align*}
%  The constant depends on the frame function and $n$. By the above inequality, we have
% \begin{align*}
 %   \sideset{}{'}{\sum} _{\lmdl J \rmdl = q} \sum_{j \in I_s} \(\lmdl L_j \lambda \rmdl^2 u_J, u_J\) + \sideset{}{'}{\sum} _{\lmdl K \rmdl = q - 1} \sum_{i,j = 1} ^{n} \(- L_i \lambda \bar{L} _j \lambda u_{iK},u_{jK} \) \geq 0.
%  \end{align*}
finishs the proof. %(\ref{duan}).
%
%  \begin{align*}
 %   &\lnm \bar{\partial} u \rnm^2 + \lnm \bar{\partial}^* u  \rnm^2 + \lnm u \rnm^2 \\
 %   \gtrsim &  - C\(\gamma,\eta,n\) \sideset{}{'}{\sum} _{\lmdl J \rmdl = q} \sum_{j\leq n} \lnm \(L _j \lambda\) u_J^{\dagger} \rnm^2
  %  + \sideset{}{'}{\sum} _{\lmdl J \rmdl = q} \sum_{j \in I_s} \(- 2 \lambda_{jj}   u_J^{\dagger},  u_J^{\dagger} \) + \sideset{}{'}{\sum} _{\lmdl K \rmdl = q - 1} \sum_{i,j = 1} ^{n} \( 2 \lambda_{ij}  u_{iK}^{\dagger}, u_{jK}^{\dagger} \).
 % \end{align*}
\end{proof}

\subsection{Standard Sobolev estimates}

By the standard boundary regularity of the elliptic operator, we have the following Sobolev estimate. 

\begin{lem}
 Let $\Omega$ be a bounded smooth domain in $\mathbb{C}^n$. If $u \in \dom(\bar{\partial}^*) \cap W_{(0,q)}^k \(\Om\)$ for $k \geq 1$, then %the norm part $u _{\rm{Norm}}$ satisfies
  \begin{align*}
    \lnm u_{\rm{Norm}} \rnm_k \lesssim \lnm (\bar\partial \bar\partial^* + \vartheta \bar\partial  ) u \rnm_{k-2} + \lnm u \rnm_{k-1}.
  \end{align*}
\end{lem}

\begin{proof}
  Since $u \in \dom(\bar{\partial}^*)$, $u _{\rm{Norm}} |_{b\Om} = 0$. Write $u _{\rm{Norm}} = \sum' _{J}  u^{\dagger}_J \bar{w}_J = \sum' _{K}  u_K d \bar{z}_K $. Then by the standard elliptic estimates (cf. \cite{E2022, T23}), we have 
  \begin{align*}
    \lnm u_{\rm{Norm}}  \rnm_k 
     \lesssim \sum'_J \lnm \Delta u^{\dagger}_J \rnm_{k-2} +\sum'_J \lnm u^{\dagger}_J \rnm _{k - 1}
     \lesssim \sum'_K \lnm \Delta u_K  \rnm_{k-2} + \sum'_K \lnm u_K \rnm_{k-1} 
     \lesssim  \lnm (\bar\partial \bar\partial^* + \vartheta \bar\partial  ) u \rnm_{k-2} + \lnm u \rnm_{k-1}.
  \end{align*}
\end{proof} 

\begin{coro}\label{uN estimate}
 Let $\Omega$ be a bounded smooth domain in $\mathbb{C}^n$.   If $u \in \dom(\bar{\partial}^*) \cap W_{(0,q)}^k \(\Om\)$  for $k \geq 1$, then %the norm part $u _{N}$ satisfy the following inequality 
  \begin{align*}
    \lnm u_{\rm{Norm}} \rnm_k \lesssim \lnm \bar{\partial} u \rnm_{k-1} + \lnm \bar{\partial} ^{*} u \rnm_{k-1} + \lnm u \rnm_{k-1}.
  \end{align*}
\end{coro}

The next result is the benign estimates due to Boas and Straube \cite{BS91} and the well known interior elliptic regularity for $\bar\partial \oplus \vartheta$.
\begin{prop}\label{benign estimate}
  Let $\Om$ be a smooth bounded pseudoconvex domain in $\bbC^n$ with defining function $\rho$, $k \in \mathbb{N}$, $0 \leq q \leq n$ and $Y$ be a tangential vector field of $(1,0)$-type  with coefficients in $C^\infty \(\cl{\Om}\)$. %with $Y \rho = 0$ on $b\Om$. 
 % Then there is a constant $C$ such that 
  For any $u \in C_{(0,q)}^\infty \(\cl{\Om}\)\cap \dom \(\bar{\partial}^*\)$, we have %the estimates
  \begin{align*}
    \sideset{}{'}{\sum}_{J}\sum_{j} \lnm \frac{\partial u_J}{\partial \bar{z}_j} \rnm _{k - 1}^2 
    \lesssim  \lnm \bar{\partial} u \rnm _{k-1}^2 + \lnm \bar{\partial}^* u \rnm_{k-1}^2 + \lnm u \rnm _{k-1}^2 ;
  \end{align*}

  \begin{align*}
    \lnm Yu  \rnm _{k-1}^2 
    \lesssim \lnm \bar{\partial} u \rnm _{k-1}^2 + \lnm \bar{\partial}^* u \rnm_{k-1}^2 + \lnm u \rnm_{k-1} \lnm u \rnm_{k}.
  \end{align*}
Moreover, if $u$ is compactly supported in $\Om$, then 
  \begin{align}\label{interior estimate}
    \lnm u \rnm _{k} ^{2} \lesssim \lnm \bar{\partial} u \rnm _{k-1}^2 + \lnm \bar{\partial}^* u \rnm_{k-1}^2 .
  \end{align}
\end{prop}

%For the elliptic regularized $\bar{\partial}$-Neumann operator $N_{\delta, q}$, 
We also need the following results on $N_{\delta, q}$ due to Straube (cf. \cite{S08}).
\begin{prop}[\cite{S08}]\label{B regular}
 Let $\Omega$ be a bounded smooth domain in $\mathbb{C}^n$. For any $k \in \mathbb{N}$, there exists a positive constant $C_{k}$ independent of $\delta$, such that the following estimates hold for all $u \in C^\infty_{(0, q)}(\overline{\Omega})$.
\begin{itemize}
\item  %\begin{align*}
 If the Bergman projection $B_q: W^k_{(0, q)}(\overline\Omega) \to W^k_{(0, q)}(\overline\Omega)$ is bounded, %exactly regular on the $L^2$-Sobolev space, %$W^k_{(0, q)}(\overline\Omega)$, 
 then $$  \left\|N_{\delta, q} u\right\|_k \leqslant C_k \left(\left\|\bar{\partial} N_{\delta, q} u\right\|_k+\left\|\bar{\partial}^* N_{\delta, q} u\right\|_k\right);$$
%  \end{align*}  
 \item   %  \begin{align*}
 $   \left\|\bar{\partial} \bar{\partial}^* N_{\delta, q} u\right\|_k^2+\left\|\vartheta \bar{\partial} N_{\delta, q} u\right\|_k^2 
    \leq C _{k}\left(\left\|\bar{\partial} N_{\delta, q} u\right\|_k^2+\|u\|_k^2+\|u\|_1^2+\delta\left\|\nabla N_{\delta, q} u\right\|_k^2+\delta^2\|u\|_{k+1}^2\right).$
%  \end{align*}
  \end{itemize}
\end{prop}
% \begin{lem}[\cite{S08}]\label{delta}
%   Let $k \in \mathbb{N}$. There is a constant $C_k$ such that when $\delta>0$ and $u \in C_{(0, q)}^{\infty}(\cl{\Omega})$
%   \begin{align*}
%     \delta^2\left\|N_{\delta, q} u\right\|_{k+2}^2 \leqslant C _{k}\left(\|u\|_k^2+\delta\left\|\nabla N_{\delta, q} u\right\|_k^2+\delta^2\|u\|_{k+1}^2\right) .
%   \end{align*}
% \end{lem}
%\begin{lem}[\cite{S08}]\label{dbar dbar*}
%  Let $k \in \mathbb{N}$. Then we have the estimate
 % \begin{align*}
 %   &\left\|\bar{\partial} \bar{\partial}^* N_{\delta, q} u\right\|_k^2+\left\|\vartheta \bar{\partial} N_{\delta, q} u\right\|_k^2 \\
%    \leq &C _{k}\left(\left\|\bar{\partial} N_{\delta, q} u\right\|_k^2+\|u\|_k^2+\|u\|_1^2+\delta\left\|\nabla N_{\delta, q} u\right\|_k^2+\delta^2\|u\|_{k+1}^2\right)
%  \end{align*}
%  with a constant $C_{k}$ independent of $\delta$.
%\end{lem}

\section{Transverse vector fields}\label{vf}
\subsection{Definition of property ($\widetilde P^{\#}$)}
%The following 
%\begin{deft}\label{theta}
 Let $\Om \subset \bbC ^{n}$ be a bounded smooth domain with a smooth defining function $\rho$. Given a smooth, transverse, $\(1, 0\)$-vector field $v$ near $\partial\Om$ with $d\rho\(v\) \neq 0$, the $\(0,1\)$-form %the critical $\(0,1\)$-form near $\partial\Om$ associated to $v$ (with respect to $\rho$) to be 
 \begin{align*}
  \theta ^{v} = \frac{d\rho \(\left[\bar{\partial}, v\right]\)}{d\rho\(v\)} 
 \end{align*}
%\end{deft}
are introduced in \cite{BS93, H11}. 
Writing $v = \sum _{j = 1} ^{n} v_{j} \frac{\partial}{\partial z _{j}}$, then
\begin{align*}
  d\rho\(v\) = v \rho = \sum _{j = 1} ^{n} v_{j} \frac{\partial \rho}{\partial z_{j}}, ~~\text{         }~~~~~
  \theta ^{v} = \sum_{j=1}^{n} \theta _{j} ^{v} d \bar{z}_{j} = \(v\rho\) ^{-1} \sum _{j,k = 1} ^{n} \frac{\partial \rho}{\partial z _{k}} \frac{\partial v_{k}}{\partial \bar{z} _{j}} d \bar{z}_{j}.
  % \theta ^{v} = \sum_{j=1}^{n} \theta _{j} d \bar{z}_{j} = \(v\rho\) ^{-1} \sum _{j,k = 1} ^{n} \frac{\partial \rho}{\partial z _{k}} \frac{\partial v_{k}}{\partial \bar{z} _{j}} d \bar{z}_{j}.
\end{align*} 
%Let $\chi$ be a real smooth function on $\Omega$ with $\chi \equiv 1$ near $\partial\Om$. Then it easy to see that $$\theta^{\chi \cdot v}= v+ \chi^{-1} \bar\partial \chi,$$
%with $ \chi^{-1} \bar\partial \chi $ compact supported in $\Omega$. 
%Write $\theta ^{v}=\theta^{v} _{\rm Tan} + \theta^{v} _{\rm Norm}$ in special boundary charts,
%\begin{align*}
 % \theta ^{v} = \sum _{j = 1}^{n} \(\theta ^{v}_{j}\)^{\dagger} \bar{\omega}_{j} = \sum _{j = 1}^{n-1} \(\theta ^{v}\)^{\dagger}_{j} \bar{\omega}_{j} + \(\theta ^{v}\)^{\dagger}_{n} \bar{\omega}_{n},
%\end{align*}
%where $\theta^{v} _{\rm Tan}=\sum _{j = 1}^{n-1} \(\theta ^{v}\)^{\dagger}_{j} \bar{\omega}_{j}$ is the tangential part and $\theta^{v} _{\rm Norm} =\(\theta ^{v}\)^{\dagger}_{n} \bar{\omega}_{n}$ is the normal part of $\theta^{v}$.
% Write $\theta^{v} = \theta^{v} _{T} + \theta^{v} _{N}$, where $\theta^{v} _{T}$ is called the tangential component and $\theta^{v} _{N}$ is called the normal component.
For an ordered index set $I _s = \{j _k: 1 \leq k \leq s\} \subset \{1,2, \cdots, n-1\}$, %the sum of the smallest $q$ eigenvalues of $\(\lambda _{ij}\)$ minus $\sum _{j \in I _{s}} \lambda _{jj}$ is denoted by 
define the Hermitian matrix $\left(\mathcal{X}^{q}_{ij} \(\phi\) \right)_{1 \leq i, j \leq n}$ associated to $\phi$ acting on $(0, q)$-form $u=\sideset{}{'}{\sum} _{|J| = q } u _{J}^{\dagger} \bar{w} _{J}$ by:
$$\mathcal{X}^{q}_{ij} \(\phi\) \big(u\big)=\sideset{}{'}{\sum} _{|K| = q - 1}\sum _{i,j = 1} ^{n} \phi ^{\dagger} _{ij} u _{iK}^{\dagger} \bar{w} _{j} \wedge \bar{w} _{K} -\left( \sum _{j \in I _{s}} \phi^{\dagger}  _{jj} \right) u $$
and the corresponding Hermitian bilinear form by  
$$ \mathcal{X}^{q} \(\phi\) \big(u, u \big)= \left( \mathcal{X}^{q}_{ij} \(\phi\) \big(u\big), u \right)= \sideset{}{'}{\sum} _{|K| = q - 1}\sum _{i,j = 1} ^{n} \phi ^{\dagger} _{ij} u _{iK}^{\dagger} \bar{u} _{jK}^{\dagger} - \sum _{j \in I _{s}} \phi ^{\dagger} _{jj} \lmdl u \rmdl ^{2}. $$
%If we restrict $\partial\bar\partial \phi $ to $T^{(1,0)}(\partial\Omega)$ and denote the eigenvalues by $\mu' _{1} \leq \mu' _{2} \leq \cdots \leq \mu' _{n-1}$, then define $$\mathcal{X}^{\partial\Omega} _{q} \(\phi\) = \sum_{j=1}^q \mu'_j - \sum _{j \in I _{s}} \phi _{jj}.$$
Denote the eigenvalues of the matrix $\left( \phi^{\dagger}  _{ij} \right)_{1\leq i, j \leq n}$ by $\mu _{1} \leq \mu _{2} \leq \cdots \leq \mu _{n}$ and $$ \mathcal{Y}^{q} \(\phi\) =\sum _{i = 1} ^{q} \mu _{i} - \sum _{j \in I _{s}} \phi ^{\dagger} _{jj}. $$ 
  It follows from the standard argument (cf. Lemma 4.7 in \cite{S10}) that for any constant $A$ 
\begin{align*}
  \mathcal{X}^{q} \(\phi\) \big(u, u \big) 
  \geq  A  \lmdl u \rmdl ^{2} ~~~{\rm for ~any~}(0, q)-{\rm form~} u
%  \leq \sideset{}{'}{\sum} _{|K| = q - 1}\sum _{i,j = 1} ^{n} \lambda _{ij} u _{iK} \bar{u} _{jK} - \sum _{j \in I _{s}} \lambda _{jj} \lmdl u \rmdl ^{2}.
\end{align*}
if and only if $$ \mathcal{Y}^{q} \(\phi\)  \geq A   .$$
%Denote $$ \mathcal{Y}^{q} \(\phi\) =\sum _{i = 1} ^{q} \mu _{i} - \sum _{j \in I _{s}} \phi ^{\dagger} _{jj}. $$

Inspired by property ($\widetilde{P}^{\#}$) %and property (${P}^{\#}$) 
defined by Yue Zhang in \cite{YZ2021}, we introduce the following definitions.

\begin{deft}\label{property Pq}
  Let $\Om$ be a bounded pseudoconvex domain in $\mathbb{C}^n$ with $n \geq 3$ and $\rho$ be the smooth defining function.  We say that $\Om$ possesses a family of transverse vector fields satisfying property strong ($\widetilde{P} _{q} ^{\#}$) for $1< q \leq n-1$ if  there exist open sets  $\lbk V_{\alpha}\rbk _{\alpha = 1}^{N}$ in $\mathbb{C}^n$ with special boundary charts  satisfying $\partial\Omega \subset \bigcup_\alpha V_\alpha$, and,  for every $\varepsilon > 0$, 
  % an open covering $\{V'_{0, \varepsilon}, V'_{1, \varepsilon}, \cdots, V'_{N'', \varepsilon}\}$ of $\overline{\Omega}$ subordinate to $\{V_0, V_1, \cdots, V_{N'}\}$ and 
  there exist an index set $I _s \subset \{1,2, \cdots, n-1\}$ for each $V _\alpha$ with $1 \leq s \leq n-2$,
 %  Moreover, 
 %for every $\varepsilon > 0$, there exist 
  a family of   finite many smooth  vector fields $v _{\alpha, \varepsilon}$ of $(1,0)$-type with compact support  in $V_{\alpha}$ satisfying $\partial \Omega \subset \bigcup_\alpha {\rm supp}(v _{\alpha, \varepsilon}) $, $d \rho \(v _{\alpha, \varepsilon}\)\not=0$ on supp$(v _{\alpha, \varepsilon})\cap \partial\Om$ and $C ^{2}$-smooth functions $\lambda_{\alpha, \varepsilon}$ near $\partial \Om\cap V_{\alpha}$ such that \\
% an open covering $\{U_{\alpha, \varepsilon}\}$ of $\partial\Omega$ that is a refinement of $\{V_\alpha\}$, a smooth $(1,0)$ vector field $v _{\alpha, \varepsilon}$ defined in $U_{\alpha, \varepsilon}$ with $d \rho \(v _{\alpha, \varepsilon}\)\not=0$  and a function $\lambda_\varepsilon \in C ^{2}\(\partial \Om \cap U_{\alpha, \varepsilon}\)$ satisfying \\
%  (i)  $\mathcal{X}^{\partial\Omega} _{q} \(\rho\) \geq 0$ on $\partial \Om$;\\ 
  (i) $\lmdl \text{arg}  d \rho \(v_{\alpha, \varepsilon}\) \rmdl _{\text{C} ^{1}} < \varepsilon$   on $\partial \Om \cap V_{\alpha}$;\\ 
  (ii) $ \sum_{j \leq n} \lmdl L _{j} \lambda _{\alpha, \varepsilon} \rmdl ^{2} < \mathcal{Y}^{q} \(\lambda_{\alpha, \varepsilon}\)$ on $\partial \Om \cap V_{\alpha}$;\\
  (iii) $\left|  \theta ^{v_{\alpha, \varepsilon}} \right |^2 \cdot {\rm I}_{n \times n} %{1 \leq i, j \leq n} 
  < \varepsilon^2 \left( \mathcal{X}^{q}_{ij} \(\lambda_{\alpha, \varepsilon}\) \right)_{1 \leq i, j \leq n}$ on $\partial \Om \cap {\rm supp}(v _{\alpha, \varepsilon})$; \\
  and   \begin{align}\label{boundary estimate''}
    \sideset{}{'}{\sum} _{\lmdl K \rmdl = q - 1} \sum_{i,j = 1}^{n} \int_{\partial \Om } g \rho^{\dagger} _{ij} u _{iK}^{\dagger} \cl{u}_{jK}^{\dagger} e ^{- \phi} \, dS - \sideset{}{'}{\sum}_{\lmdl J \rmdl = q} \sum_{j \in I_s} \int _{\partial \Om} g \rho^{\dagger} _{jj} \lmdl u_J^{\dagger} \rmdl ^2  e ^{-\phi} \,dS \geq 0
  \end{align}
holds for all $g, \phi \in C^2(\overline{\Om})$, $u= \sideset{}{'}{\sum} _J u _J^{\dagger} \bar{w}_J \in  C _{\(0, q\)}^ \infty \(\cl{\Om} \) \cap \dom\(\bar{\partial}^*\)$ supported in $V_{\alpha}$.
\end{deft}

%For a smooth bounded pseudoconvex domain $\Om$ with a family of transverse vector fields satisfying \autoref{property Pq}. Then if there exists a finite cover $\{V_{j}\} _{j = 1} ^{N}$ of $\partial \Om$ with special boundary charts defined on each $V _{j}$ and the following also holds on each $V _{j}$, \\
%  (i) $\lmdl \text{arg}  v_\varepsilon \rho \rmdl < \varepsilon$ on $U \cap V _{j}$.\\ 
 % (ii) For any $\tau > 0$, $\sum' _{|K| = q - 1} \sum _{i,j = 1} ^{n} \lambda _{ij} - \sum' _{|J| = q} \sum _{j \in I _{s}} \lambda _{jj} \geq \tau \sum_{j \leq n} \lmdl L _{j} \lambda _\varepsilon \rmdl ^{2}$ on $U \cap V _{j}$.\\
%  (iii) $\lmdl \theta \wedge \bar{\theta} \rmdl < \varepsilon \mathcal{X}_{q} \(\lambda\)$ on $U \cap V _{j}$.
 %\end{comment}

\begin{rmk}\label{increasing}
By Schur’s majorization theorem, the sum of the smallest $q$ eigenvalues of an $n \times n$ Hermitian matrix is less than or equal to the sum of the smallest $q$ diagonal entries of the same matrix. If $\Om$ possesses a family of transverse vector fields satisfying  property strong ($\widetilde{P} _{q} ^{\#}$), then $\mathcal{X}^{q} \(\lambda_{\alpha, \varepsilon}\) \geq 0$ implies that $q \geq  s$ and the $(q+1)$-th smallest eigenvalue must be non-negative and thus $\mathcal{X}^{q+1} \(\lambda_{\alpha, \varepsilon}\) \geq \mathcal{X}^{q} \(\lambda_{\alpha, \varepsilon}\)$ and $\mathcal{Y}^{q+1} \(\lambda_{\alpha, \varepsilon}\) \geq \mathcal{Y}^{q} \(\lambda_{\alpha, \varepsilon}\). $ Since $\partial\bar\partial \rho$ is non-negative restricted to $T^{(1, 0)}\partial \Om$, (\ref{boundary estimate''}) holds for every $u \in  C _{\(0, q+1\)}^ \infty \(\cl{\Om} \) \cap \dom\(\bar{\partial}^*\)$ by the same reason. Therefore,  $\Om$ possesses a family of transverse vector fields satisfying  property strong ($\widetilde{P} _{q+1} ^{\#}$).
\end{rmk}

As in the proof of \autoref{boundary term}, (\ref{boundary estimate''}) holds automatically if $s=n-2, q=n-1$. This %inspires us to define
reduces to a weaker property $(\widetilde{P}^{\#}_{n-1})$.

\begin{deft}\label{property p}
  Let $\Om$ be a smooth bounded pseudoconvex domain in $\mathbb{C}^n$ with $n \geq 3$ and $\rho$ be the smooth defining function. We say that $\Om$ possesses a family of transverse vector fields satisfying property ($\widetilde{P} _{n-1} ^{\#}$) if  there exist open sets $\lbk V_{\alpha}\rbk _{\alpha = 1}^{N}$ in $\mathbb{C}^n$  with special boundary charts satisfying $\partial\Omega \subset \bigcup_\alpha V_\alpha$,  %an open covering $\{U_{\alpha, \varepsilon}\}$ of $\partial\Omega$ that is a refinement of $\{V_\alpha\}$, 
  % an open covering $\{V'_{0, \varepsilon}, V'_{1, \varepsilon}, \cdots, V'_{N'', \varepsilon}\}$ of $\overline{\Omega}$ subordinate to $\{V_0, V_1, \cdots, V_{N'}\}$ and 
and, for 
  every $\varepsilon > 0$, there exist 
 a family of  finite many smooth  vector fields $v _{\alpha, \varepsilon}$ of $(1,0)$-type with compact support  in $V_{\alpha}$ satisfying $\partial \Omega \subset \bigcup {\rm supp}(v _{\alpha, \varepsilon}) $, $d \rho \(v _{\alpha, \varepsilon}\)\not=0$ on supp$(v _{\alpha, \varepsilon})\cap \partial\Om$ and functions $\lambda_{\alpha, \varepsilon} \in C ^{2}\(\partial \Om\cap V_{\alpha}\)$, $1 \leq t_{\alpha, \varepsilon} \leq n-1$ such that \\
  (i)  $\lmdl \text{arg}  d \rho \(v_{\alpha, \varepsilon}\) \rmdl _{\text{C} ^{1}} < \varepsilon$ on $\partial \Om \cap V_{\alpha}$; \\ 
  (ii) $  \sum_{j \leq n} \lmdl L_j \lambda_{\alpha, \varepsilon} \rmdl ^{2}< \left( \lambda_{\alpha, \varepsilon}\right)^{\dagger} _{t_{\alpha, \varepsilon} t_{\alpha, \varepsilon}}$ 
  on $\partial \Om \cap V_{\alpha}$; \\
  (iii) %$\lmdl \theta^{v_{\alpha, \varepsilon}} _{\rm Tan}  \rmdl^2 < \varepsilon \left(\lambda_{\varepsilon}\right)_{t_{\alpha} t_{\alpha}}$ and $\lmdl \theta^{v_{\alpha, \varepsilon}}_{\rm Norm} \rmdl^2 < \varepsilon$ 
  $ \left|  \theta ^{v_{\alpha, \varepsilon}} \right |^2 < \varepsilon^2   \left( \lambda_{\alpha, \varepsilon}\right)^{\dagger} _{t_{\alpha, \varepsilon} t_{\alpha, \varepsilon}} $ %and $\lmdl \(\theta ^{v_{\alpha, \varepsilon}}\)^{\dagger}_n \rmdl^2 < \varepsilon $ 
  on $\partial \Om \cap {\rm supp}(v _{\alpha, \varepsilon})$.
\end{deft}

\begin{prop}
For a smooth bounded pseudoconvex domain $\Om$ in $\mathbb{C}^n$ with $n \geq 3$, if $\partial\Om$ satisfies property ($\widetilde{P} _{q} ^{\#}$) (resp. property ($\widetilde{P} _{n-1} ^{\#}$)) %\bigg(resp. property ($P _{n-1} ^{\#}$)\bigg) 
as defined in \cite{YZ2021}, then $\Om$ possesses a family of transverse vector fields satisfying property strong ($\widetilde{P} _{q} ^{\#}$) (resp. property ($\widetilde{P} _{n-1} ^{\#}$)).  %\bigg(resp. property ($P _{n-1} ^{\#}$) \bigg).
\end{prop}

\begin{proof}
We only consider the case of property strong ($\widetilde{P} _{q} ^{\#}$) and the argument for property ($\widetilde{P} _{n-1} ^{\#}$) is the same. 
Assume that $\partial\Om$ satisfies property ($\widetilde{P} _{q} ^{\#}$). Then (\ref{boundary estimate''}) and (ii) in \autoref{property Pq} follows trivially. For every $Q$ in $V_\alpha$, choose  %$V _{\alpha, \varepsilon}$, and $c \in \mathbb{C}$, choose 
a real smooth function $\chi$ near $Q$ with compact support in $V_\alpha$ satisfying $\chi(Q)=1$ and %$ L_{n} f_c (Q)=c$. 
let $v_{\alpha, \varepsilon} = \chi L_{n}$. 
Then $\lmdl \text{arg}  d \rho \(v_{\alpha, \varepsilon}\) \rmdl _{\text{C} ^{1}} \equiv 0$ on $\partial \Om \cap V_{\alpha}$.
%near $Q$, one can easily verify that $\theta^{v_{\alpha, \varepsilon}}_{\rm Norm} (Q)=0$ by choosing  appropriate $c$. 
Moreover, since  $ \left|  \theta ^{v_{\alpha, \varepsilon}} \right |$ is bounded on ${\rm supp}(v _{\alpha, \varepsilon}) \cap \partial{\Om}$ and the smallest eigenvalue of  $ \left( \mathcal{X}^{q}_{ij} \(\lambda_{\alpha, \varepsilon}\) \right)_{1 \leq i, j \leq n}$
%\left( \lambda_{\alpha, \varepsilon}\right)^{\dagger} _{t_{\alpha, \varepsilon} t_{\alpha, \varepsilon}} $ 
can be arbitrarily large on $\partial\Om \cap V_\alpha$, %$\sum _{j = 1}^{n-1} \lmdl \(\theta ^{v_{\alpha, \varepsilon}}\)^{\dagger}_j \rmdl^2 (Q)  < \varepsilon \left(\lambda_{\varepsilon}\right)_{t_{\alpha} t_{\alpha}}(Q)$. 
%Thus,  (i) and 
(iii) in \autoref{property Pq} holds. Namely, $\Om$ possesses a family of transverse vector fields satisfying property strong ($\widetilde{P} _{q} ^{\#}$). 
\end{proof}

We also observe the following fact.

\begin{lem}\label{scale}
Assume that $\lambda \in C^{2}(\overline{\Om})\cap V_\alpha$ satisfies   $  \mathcal{Y}_{q} \(\lambda\) >  \sum_{j \leq n} \lmdl L _{j} \lambda  \rmdl ^{2}  $   \bigg(resp. $ \lambda^{\dagger} _{tt} >  \sum_{j \leq n} \lmdl L_j \lambda \rmdl ^{2}$  \bigg) on $\partial \Om$. Then for any $A>0$, there exists $\lambda_A  \in C^{2}(\overline{\Om})\cap V_\alpha$ such that  $  \mathcal{Y}_{q} \(\lambda_A\) >  A   \sum_{j \leq n} \lmdl L _{j} \lambda_A  \rmdl ^{2} $ \bigg(resp.   $\left( \lambda_A \right)^{\dagger} _{tt} >  A \sum_{j \leq n} \lmdl L_j \lambda_A \rmdl ^{2}$ \bigg) on $\partial \Om$.
\end{lem}
\begin{proof}
It suffices to let $\lambda_A =\frac{1}{A} \lambda$.
\end{proof}

\begin{rmk}\label{nbhd}
By the continuity and \autoref{scale}, for every $\varepsilon>0$, there exists an open neighborhood $U_\varepsilon$ of $\partial\Om$ such that $\lmdl \text{arg}  d \rho \(v_{\alpha, \varepsilon}\) \rmdl _{\text{C} ^{1}} < 2\varepsilon$  holds on $V_\alpha \cap U_\varepsilon$; $ \sum_{j \leq n} \lmdl L _{j} \lambda _{\alpha, \varepsilon} \rmdl ^{2} < \mathcal{Y}^{q} \(\lambda_{\alpha, \varepsilon}\)$ holds on $V_\alpha \cap U_\varepsilon$; $\left|  \theta ^{v_{\alpha, \varepsilon}} \right |^2 \cdot {\rm I}_{n \times n} %{1 \leq i, j \leq n} 
  < 2 \varepsilon^2 \left( \mathcal{X}^{q}_{ij} \(\lambda_{\alpha, \varepsilon}\) \right)_{1 \leq i, j \leq n}$ holds on ${\rm supp}(v _{\alpha, \varepsilon}) \cap U_\varepsilon$ in  \autoref{property Pq}. The corresponding statement for  \autoref{property p} is also true.
\end{rmk}

\begin{rmk}
The particular form of the bound $\varepsilon$ is not important. In particular, one may replace $\varepsilon$ in (i) of \autoref{property Pq} and \autoref{property p} by $A_1(\varepsilon)$ and  $\varepsilon^2$ in (iii) of \autoref{property Pq} and \autoref{property p} by $A_2(\varepsilon)$ as long as $\lim_{ \varepsilon \rightarrow 0^+} A_j(\varepsilon)= 0$ for $j=1, 2$.
\end{rmk}

\subsection{Estimates regarding transverse vector fields}
%As stated before, we follow Harrington's strategy in \cite{H11} to prove the exact regularity of $\bar\partial$-Neumann operator. Now, we introduce the operators $D _{Z}$ and $D _{\bar{Z}}$ defined in \cite{H11}.
Let $u \in C^\infty_{(0, q)}(\overline{\Om})$ with $q \geq 1$, for any  vector field $Z= \sum _{j = 1} ^{n} Z _{j} \frac{\partial}{\partial z_{j}}$ of $\(1, 0\)$-type, the following operators are introduced by Herbig-McNeal \cite{HM06} and Harrington \cite{H11}.  %$D_Z$ by
 \begin{equation*}
 \begin{split}
  D _{Z} u 
 & = - \sum _{j = 1} ^{n} \(Z _{j} d \bar{z} _{j} \wedge \vartheta u + \vartheta\(Z _{j} d \bar{z} _{j} \wedge u\)\); \\
D_{\bar{Z}} u & = \nabla_{\bar{Z}} u-\sideset{}{'}{\sum}_{I} \sum_{j, k, l=1}^n \frac{\partial \rho}{\partial \bar{z}_{l}}\left(\frac{\partial}{\partial z_j} \bar{Z}_{l}\right) u_{j I} (Z \rho)^{-1} Z_{k} d \bar{z}_k  \wedge d \bar{z}_I, ~~{\rm if}~~Z\rho \not=0.
\end{split}
\end{equation*}
These operators play an essential role in our proof of the exact regularity for the $\bar\partial$-Neumann operator and we heavily rely on Harrington's new ideas  in \cite{H11} to treat the commutators $[\bar\partial^*, v_{\alpha, \varepsilon}], [\bar\partial, v_{\alpha, \varepsilon}]$.
Let us start with the review of some important properties. 
It follows from the definition that $\left[\vartheta, D_{Z}\right] \equiv 0$.
By the straightforward calculation (cf. Page 2537 and 2538 in \cite{H11}),
\begin{equation}\label{DZ}
\begin{split}
D_Z u &  =\nabla_Z u+\sum_{j=1}^n\left(\frac{\partial}{\partial z_j} Z_{j}\right) u-\sideset{}{'}{\sum}_{I} \sum_{j, k=1}^n\left(\frac{\partial}{\partial z_k} Z_{j}\right) u_{k I} d \bar{z}_j \wedge d \bar{z}_I, \\ %\\\label{DZ} %\nonumber
 D_{\bar{Z}} u &= \nabla_{\bar{Z}} u-\sideset{}{'}{\sum}_{I} \sum_{j, k=1}^n \overline{\theta_{j}^Z} u_{jI} \frac{\bar{Z} \rho}{Z \rho} Z_{k} d \bar{z}_k  \wedge d \bar{z}_I .
\end{split}
\end{equation}
% and
% \begin{align*}
%   D_Z u
%   % =\sum_{j=1}^n \frac{\partial}{\partial z_j}\left(Z^j u\right)-\sum_{I \in \mathcal{I}_{q-1}} \sum_{j, k=1}^n\left(\frac{\partial}{\partial z_k} Z^j\right) u_{k I} d \bar{z}_j \wedge d \bar{z}_I \\\nonumber
%   =Z u+\sum_{j=1}^n\left(\frac{\partial}{\partial z_j} Z_j\right) u,
% \end{align*}
% when $q = 0$. 
It thus follows that 
\begin{align}\label{DZ-1111}
  \(D _{Z} ^{k} - \nabla^k_Z\) u ~{\rm{and}}~    \(D _{\bar{Z}} ^{k} - \nabla^k_{\bar{Z}}\) u = {\rm the~ derivative~ of~} u {\rm{~with~order~at~most~ }}   k - 1.
\end{align} 
% Hence, the principal part of $D _{V}$ is $V$.  
% \begin{align*}
%   \left| \(D _{V}\) ^{k} u - V ^{k} u\right| \leq C \left| u \right| _{k - 1},
% \end{align*}
% where $C$ is a positive constant depending on $V$. 
%Moreover, the difficulty with $v_{\alpha, \varepsilon}$ is that it is non-tangential. This can be corrected by defining
For the transverse vector field $v_{\alpha, \varepsilon} = \sum _{j = 1} ^{n} \(v_{\alpha, \varepsilon}\) _{j} \frac{\partial}{\partial z _{j}}$ with $v_{\alpha, \varepsilon} \rho \not=0$ on ${\rm supp}(v _{\alpha, \varepsilon})$, define
\begin{align*}
  v_{\alpha, \varepsilon} ^{\#} = \frac{d \rho \(\bar{v}_{\alpha, \varepsilon}\)}{d \rho \(v_{\alpha, \varepsilon}\)} v_{\alpha, \varepsilon} %= \sum _{j = 1} ^{n} \(v_{\alpha, \varepsilon} ^{\#}\) _{j} \frac{\partial}{\partial z _{j}}
  , ~~~~~~ X_{\alpha, \varepsilon} = v_{\alpha, \varepsilon} ^{\#} - \bar{v}_{\alpha, \varepsilon}, ~~~~~~{\rm{and}}~~~~~~D_{X _{\alpha, \varepsilon}} = D_{v_{\alpha, \varepsilon} ^{\#} } - D_{\bar{v}_{\alpha, \varepsilon}}.
\end{align*}
Define $v_{\alpha, \varepsilon} ^{\#} = X_{\alpha, \varepsilon} =D_{X _{\alpha, \varepsilon}} = D_{v_{\alpha, \varepsilon} ^{\#} } = D_{\bar{v}_{\alpha, \varepsilon}}=0$ away from ${\rm supp}(v _{\alpha, \varepsilon})$. 
It is easy to verify that $X_{\alpha, \varepsilon}$ is a tangential vector field and Harrington discovered the  crucial property that 
$D_{X _{\alpha, \varepsilon}}  u \in \dom \(\bar{\partial} ^{*}\)$ if $u \in \dom \(\bar{\partial} ^{*}\)\cap C^\infty_{(0, q)}(\overline{\Omega})$ (cf. Page 2539 in \cite{H11}). 
Harrington also derived the following equality in \cite{H11} (cf. Page 2539): 
\begin{align}\label{alice1}
  \left[D_{X_{\alpha, \varepsilon}}, \vartheta\right] u= \left[\vartheta, D_{\bar{v}_{\alpha, \varepsilon}}\right] u
  =\left[\vartheta, \bar{v}_{\alpha, \varepsilon}\right] u + \sideset{}{'}{\sum}_{I} \sum_{j=1}^n \overline{ \theta^{v _{\alpha, \varepsilon}}_{j}} \left( v_{\alpha, \varepsilon}^{\sharp} u_{jI} \right) d \bar{z}_I+ 0{\rm th~ order~ term~in~}u,
\end{align}
The following inequalities are also derived  in \cite{H11} (cf. Page 2537 and 2538):
\begin{align}\label{AAA}
  \left| \( v_{\alpha, \varepsilon} ^{k} - \(v_{\alpha, \varepsilon} ^{\#}\) ^{k}\) u\right| 
  \leq \left|1 - \(\frac{d \rho \(\bar{v}_{\alpha, \varepsilon}\)}{d \rho \(v_{\alpha, \varepsilon}\)}\) ^{k} \right| \left| v_{\alpha, \varepsilon} ^{k} u\right| + {\rm the~ derivative~ of~} u {\rm{~of~order~at~most~ }}   k - 1, %C _{\varepsilon}\left|\nabla ^{k-1} u \right|,
  % + \left| \nabla_\varepsilon^{k-1} u \right|,
\end{align}
% where $\nabla_\varepsilon^{k-1}$ is a differential operator of order no more than $k-1$ with coefficient functions depending on $\varepsilon$, 
and
\begin{align}\label{BBB}
  \left|1 - \(\frac{d \rho \(\bar{v}_{\alpha, \varepsilon}\)}{d \rho \(v_{\alpha, \varepsilon}\)}\) ^{k} \right|
  = \left|1 - e ^{- 2ik \arg \(d \rho \(v_{\alpha, \varepsilon}\)\)} \right|
  \leq 2k \left|\arg \(d \rho \(v_{\alpha, \varepsilon}\)\)\right|.
\end{align}

The next result is standard and we record it here for completeness.
\begin{lem}
 For any vector field of $\bar{Z}$ of $\(0, 1\)$-type, $\left[\vartheta, \bar{Z} \right]$ is also an $\(0, 1\)$-type derivative. 
 \end{lem}

\begin{proof} 
 Let   $u = \sum' _{|J| = q} u _{J} d \bar{z} _{J}$  be an $\(0, q\)$-form and write $\(0, 1\)$-type vector field $\bar{Z} = \sum _{j = 1} ^{n} \bar{Z} _{j} \frac{\partial}{\partial \bar{z} _{j}}$. %we have 
 % \begin{align*}
  %  \left[\vartheta, \bar{Z} \right] f = \vartheta \bar{Z} f - \bar{Z} \vartheta f.
  %\end{align*}
 % Then, calculating the following terms, we have
It follows that  \begin{align*}
    \vartheta \bar{Z} u 
    &= \vartheta \(\sideset{}{'}{\sum} _{|J| = q}\sum _{j = 1} ^{n} \bar{Z} _{j} \frac{\partial u _{J}}{\partial \bar{z} _{j}}  d \bar{z} _{J}\) = - \sideset{}{'}{\sum} _{|K| = q - 1} \sum _{j,k = 1} ^{n} \frac{\partial}{\partial z _{k}} \(\bar{Z} _{j} \frac{\partial u _{k K}}{\partial \bar{z} _{j}} \) d\bar{z} _{K} \\
    &=  - \sideset{}{'}{\sum} _{|K| = q - 1} \sum _{j,k = 1} ^{n} \(\frac{\partial \bar{Z} _{j}}{\partial z _{k}} \) \(\frac{\partial u _{k K}}{\partial \bar{z} _{j}} \) d\bar{z} _{K} - \sideset{}{'}{\sum} _{|K| = q - 1} \sum _{j,k = 1} ^{n} \bar{Z} _{j} \frac{\partial^2 u _{kK} }{\partial z _{k} \partial \bar{z} _{j}} d\bar{z} _{K},
  \end{align*}
  and
  \begin{align*}
    \bar{Z} \vartheta u 
    = \bar{Z}\(- \sideset{}{'}{\sum} _{|K| = q - 1} \sum _{k = 1} ^{n} \frac{\partial u _{kK}}{\partial z _{k}}  d \bar{z} _{K}\)
    = - \sideset{}{'}{\sum} _{|K| = q - 1} \sum _{j, k = 1} ^{n} \bar{Z} _{j} \frac{\partial^2 u _{kK}}{\partial \bar{z} _{j} \partial z _{k}}  d \bar{z} _{K}.
  \end{align*}
  Thus,
%  \begin{align*}
$    \left[\vartheta, \bar{Z} \right] 
     = \vartheta \bar{Z} u - \bar{Z} \vartheta $
%     = -\sideset{}{'}{\sum} _{|K| = q - 1} \sum _{j,k = 1} ^{n} \(\frac{\partial}{\partial z _{k}} \bar{Z} _{j}\) \(\frac{\partial}{\partial \bar{z} _{j}} u _{kK}\) d\bar{z} _{K}.
 % \end{align*} 
%Thus $\left[\vartheta, \bar{v}\right]$
 is a $\(0, 1\)$-type derivative.
\end{proof}

For operators $A, B$, the following commutator identity  (cf. (3.54) in \cite{S10}) holds: 
\begin{align}\label{powers of lie bracket}
  [A, B^k] 
  % = \sum _{j=0} ^{k - 1} T^{j}[A,T]T^{k - j -1}
  = \sum_{j=1}^{k} \binom{k}{j} \underbrace{[\dots[[A,B],B]\dots,B]}_{j\text{-fold}} B^{k-j} = \sum_{j=1}^{k} \binom{k}{j} B^{k-j}\underbrace{[\dots[[A,B],B]\dots,B]}_{j\text{-fold}}.%\nonumber 
\end{align}
 The following equation follows from (\ref{alice1}) and (\ref{powers of lie bracket}) and  (cf. (4.11) in \cite{H11}):
\begin{align}\label{dbar star lie bracket}
\left|\left[D_{X_{\alpha, \varepsilon}} ^{k}, \bar{\partial}^*\right] u\right| 
& \leq C_{\varepsilon, k}\left(|\bar{\nabla} u| _{k - 1} +|u| _{k - 1} \right)+ k\left| \sideset{}{'}{\sum}_{I}\sum_{j=1}^n \overline{\theta^{v _{\alpha, \varepsilon}} _{j}} \left( v_{\alpha, \varepsilon}^{\#} \left( D_{X_{\alpha, \varepsilon}} ^{k - 1} u\right)_{jI} \right) d\bar z_I\right| \\\nonumber
& \leq C_{\varepsilon, k} \left(|\bar{\nabla} u| _{k - 1} + |u| _{k - 1}\right)+ k \sideset{}{'}{\sum}_{I} \sum_{j=1}^n \left| \overline{\theta ^{v _{\alpha, \varepsilon}} _{j} } \left( D_{X_{\alpha, \varepsilon}} ^{k}u \right)_{jI} \right| .\nonumber
\end{align}

\begin{lem}\label{D sharp}
Let   $\Xi_{\alpha, \varepsilon} =D_{X _{\alpha, \varepsilon}} - X _{\alpha, \varepsilon}$ be the part of the $0$-th order operator in $D_{X _{\alpha, \varepsilon}}$. Then $\Xi_{\alpha, \varepsilon}$ admits the Hilbert space adjoint $\Xi^{\#}_{\alpha, \varepsilon}$, which is also an $0$-th order linear operator with smooth coefficients up to the boundary. Moreover, for any $u, w \in C^\infty_{(0, q)}(\overline{\Omega})$, 
$$(D_{X _{\alpha, \varepsilon}} u, w)=(u, D_{X_{\alpha, \varepsilon}} ^{\#} w)$$ holds, 
 with $D_{X_{\alpha, \varepsilon}} ^{\#}=  v_{\alpha, \varepsilon} -\bar{v}_{\alpha, \varepsilon} ^{\#} +0$-th order linear operator with smooth coefficients up to the boundary.
\end{lem}

\begin{proof}
By the definition of  $D_{X _{\alpha, \varepsilon}}$,   
\begin{align*}
\Xi_{\alpha, \varepsilon}  ( u) = &  \sum_{j=1}^n\left(\frac{\partial}{\partial z_j} \(v_{\alpha, \varepsilon} ^{\#}\)_{j}\right) u
    -\sideset{}{'}{\sum}_{I} \sum_{j, k=1}^n\left(\frac{\partial}{\partial z_k} \(v_{\alpha, \varepsilon} ^{\#}\)_{j}\right) u_{kI} d \bar{z}_j \wedge d \bar{z}_I \\
    &-\sideset{}{'}{\sum}_{I} \sum_{j, k=1}^n \overline{{\theta}_{k}^{v_{\alpha, \varepsilon}}} u_{kI} \frac{\bar{v}_{\alpha, \varepsilon} \rho}{v_{\alpha, \varepsilon} \rho} \(v_{\alpha, \varepsilon}\)_{j} d \bar{z}_j \wedge d \bar{z}_I.
  \end{align*}
It follows that $$|(\Xi_{\alpha, \varepsilon}   u, w)| \leq C_{\varepsilon} \|u\| \|w\|$$ for any $u, w \in L^2_{(0, q)}(\Omega)$. 
Therefore, $\Xi^{\#}_{\alpha, \varepsilon}$ exists and is  an $0$-th order linear operator with smooth coefficients up to the boundary by straightforward calculation. 
For any $u, w \in C^\infty_{(0, q)}(\overline{\Omega})$, 
by integration by parts, 
$\(\(v_{\alpha, \varepsilon} ^{\#} -\bar{v}_{\alpha, \varepsilon}\)u, w\) = \(u, \(v_{\alpha, \varepsilon} -\bar{v}_{\alpha, \varepsilon} ^{\#} \) w\) + 0$-th order term in $u, w$. 
It then follows that $$(D_{X _{\alpha, \varepsilon}} u, w)=(u, D_{X_{\alpha, \varepsilon}} ^{\#} w).$$ 
\end{proof}

\begin{lem}\label{adjoint of D}
 Let $\Om \subset \bbC ^{n}$ with $n \geq 3$ be a smooth bounded pseudoconvex domain and $v _{\alpha, \varepsilon}$ be a smooth vector field with $d \rho \(v _{\alpha, \varepsilon}\)  \neq 0$ and $\lmdl \text{arg}  d \rho \(v _{\alpha, \varepsilon}\) \rmdl < \varepsilon$ on $\partial\Om \cap {\rm supp}(v _{\alpha, \varepsilon})$. Then 
   \begin{align*}
     \lnm D_{X _{\alpha, \varepsilon}} ^{k} u -\(D_{X_{\alpha, \varepsilon}} ^{\#}\)^{k} u\rnm ^{2}
     \lesssim k^{2}\varepsilon^{2}\lnm v_{\alpha, \varepsilon} ^{k} u \rnm ^{2}+ C _{\varepsilon}\(\lnm \bar{\partial} u \rnm _{k - 1} ^{2}+ \lnm \bar{\partial} ^{*} u \rnm _{k - 1}^{2} + \lnm u \rnm _{k - 1}^{2} \)
    % \lnm D_{X _{\alpha, \varepsilon} ^{\#}} ^{k} u -\(D_{X_{\alpha, \varepsilon} ^{\#}} ^{\#}\)^{k} u\rnm ^{2}
    %  \leq 4k\varepsilon ^{2}\lnm v_{\alpha, \varepsilon} ^{k} u \rnm ^{2}+ C\(\lnm \bar{\partial} u \rnm _{k - 1} ^{2}+ \lnm \bar{\partial} ^{*} u \rnm _{k - 1} ^{2}+ \lnm u \rnm _{k - 1}^{2} \),
   \end{align*}
  holds for any  $u \in \dom\(\bar{\partial} ^{*}\) \cap C^\infty_{(0, q)}(\overline{\Om})$.
 \end{lem}
 \begin{proof}
   By the definition of $D _{X_{\alpha, \varepsilon}}$ and \autoref{D sharp}, write $D _{X_{\alpha, \varepsilon}} = v _{\alpha, \varepsilon} ^{\#} - \bar{v}_{\alpha, \varepsilon}+ A$ and $D _{X_{\alpha, \varepsilon}}^{\#} = v _{\alpha, \varepsilon} - \bar{v} _{\alpha, \varepsilon} ^{\#}+ B$, where $A, B$ are the $0$-order operators. Thus, we have
   \begin{align*}
     D_{X _{\alpha, \varepsilon}} ^{k} -\(D_{X_{\alpha, \varepsilon}} ^{\#}\)^{k}
     &= \(v _{\alpha, \varepsilon} ^{\#} - \bar{v}_{\alpha, \varepsilon}+ A\) ^{k} - \(v _{\alpha, \varepsilon} - \bar{v} _{\alpha, \varepsilon} ^{\#}+ B\)^{k}\\
     &= \(v _{\alpha, \varepsilon} ^{\#}\)^{k} - v _{\alpha, \varepsilon}^{k} + \nabla ^{k - 1} \bar{v}_{\alpha, \varepsilon} + \nabla ^{k - 1} \bar{v}^{\#}_{\alpha, \varepsilon} + \nabla ^{k - 1} + {\rm lower ~order~ terms}.
   \end{align*}
%   where the suppressed terms are all of lower-order. 
   Since $\bar{v}_{\alpha, \varepsilon}$ is a $\(0, 1\)$-derivative, it follows from (\ref{AAA}), (\ref{BBB})  and the benign estimate in \autoref{benign estimate} that
   \begin{align*}
    \lnm D_{X _{\alpha, \varepsilon}} ^{k} u -\(D_{X_{\alpha, \varepsilon}} ^{\#}\)^{k} u\rnm ^{2}
   & \lesssim \lnm \(v _{\alpha, \varepsilon} ^{\#}\)^{k} u -  v _{\alpha, \varepsilon}^{k} u \rnm ^{2}+ C _{\varepsilon}\lnm \bar{v}_{\alpha, \varepsilon} u \rnm _{k - 1} ^{2}+C _{\varepsilon}\lnm \bar{v}^{\#}_{\alpha, \varepsilon} u \rnm _{k - 1} ^{2}+ C _{\varepsilon}\lnm u \rnm _{k - 1}^{2} \\
     &   \lesssim  k ^{2}\lnm \arg \(d \rho \(v_{\alpha, \varepsilon}\)\) v_{\alpha, \varepsilon} ^{k} u \rnm ^{2}+ C _{\varepsilon}\(\lnm \bar{\partial} u \rnm _{k - 1}^{2} + \lnm \bar{\partial} ^{*} u \rnm _{k - 1}^{2} + \lnm u \rnm _{k - 1} ^{2}\).\\
   \end{align*}
  %  By the same argument, we have
  %  \begin{align*}
  %   \lnm D_{X _{\alpha, \varepsilon} ^{\#}} ^{k} u -\(D_{X_{\alpha, \varepsilon} ^{\#}} ^{\#}\)^{k} u\rnm ^{2}
  %   \leq \lnm \( v _{\alpha, \varepsilon}^{k} u  -  \(v _{\alpha, \varepsilon} ^{\#}\)^{k}\) u\rnm ^{2}+ \lnm \bar{v}_{\alpha, \varepsilon} u \rnm _{k - 1} ^{2}+ \lnm u \rnm _{k - 1}^{2}
  % \end{align*}
By \autoref{nbhd}, $\lmdl \text{arg}  d \rho \(v _{\alpha, \varepsilon}\) \rmdl < 2\varepsilon$ on $U_\varepsilon \cap {\rm supp}(v _{\alpha, \varepsilon})$. Choosing a smooth function $\chi_ \varepsilon$ in $\Omega$ such that $0 \leq \chi_ \varepsilon \leq 1$, ${\rm supp} \chi_ \varepsilon \subset U_\varepsilon \cap \Omega$ and $\chi_ \varepsilon \equiv 1$ near $\partial\Omega$, it follows that 
   \begin{align*}
    &~~~ \lnm D_{X _{\alpha, \varepsilon}} ^{k} u -\(D_{X_{\alpha, \varepsilon}} ^{\#}\)^{k} u\rnm^{2} \\
     &\lesssim k ^{2}\lnm \chi_ \varepsilon \arg \(d \rho \(v_{\alpha, \varepsilon}\)\) v_{\alpha, \varepsilon} ^{k} u \rnm ^{2}+ k ^{2}\lnm (1-\chi_ \varepsilon) \arg \(d \rho \(v_{\alpha, \varepsilon}\)\) v_{\alpha, \varepsilon} ^{k} u \rnm ^{2}+ C _{\varepsilon}\(\lnm \bar{\partial} u \rnm _{k - 1}^{2} + \lnm \bar{\partial} ^{*} u \rnm _{k - 1}^{2} + \lnm u \rnm _{k - 1} ^{2}\) \\
   &   \lesssim k ^{2}\varepsilon ^{2}\lnm v_{\alpha, \varepsilon} ^{k} u \rnm ^{2}+ C_{\varepsilon, k}\( \| (1-\chi_ \varepsilon)u  \|_k^2 +\lnm \bar{\partial} u \rnm _{k - 1} ^{2}+ \lnm \bar{\partial} ^{*} u \rnm _{k - 1} ^{2}+ \lnm u \rnm _{k - 1}^{2} \) \\
   &   \lesssim k ^{2}\varepsilon ^{2}\lnm v_{\alpha, \varepsilon} ^{k} u \rnm ^{2}+ C_{\varepsilon, k}\( \|  \bar{\partial} \left((1-\chi_ \varepsilon)u\right)  \|_{k-1}^2 +\|  \bar{\partial}^* \left((1-\chi_ \varepsilon)u\right)  \|_{k-1}^2+\lnm \bar{\partial} u \rnm _{k - 1} ^{2}+ \lnm \bar{\partial} ^{*} u \rnm _{k - 1} ^{2}+ \lnm u \rnm _{k - 1}^{2} \)\\
    &   \lesssim k ^{2}\varepsilon ^{2}\lnm v_{\alpha, \varepsilon} ^{k} u \rnm ^{2}+ C_{\varepsilon, k}\( \lnm \bar{\partial} u \rnm _{k - 1} ^{2}+ \lnm \bar{\partial} ^{*} u \rnm _{k - 1} ^{2}+ \lnm u \rnm _{k - 1}^{2} \),
   \end{align*}
   where the third inequality   follows from (\ref{interior estimate}). 
   % and
   % \begin{align*}
   %  \lnm D_{X _{\alpha, \varepsilon} ^{\#}} ^{k} u -\(D_{X_{\alpha, \varepsilon} ^{\#}} ^{\#}\)^{k} u\rnm ^{2}
   %   \leq 4k\varepsilon ^{2}\lnm v_{\alpha, \varepsilon} ^{k} u \rnm ^{2}+ C\(\lnm \bar{\partial} u \rnm _{k - 1} ^{2}+ \lnm \bar{\partial} ^{*} u \rnm _{k - 1} ^{2}+ \lnm u \rnm _{k - 1}^{2} \),
   % \end{align*}
 \end{proof}
%  Also,  define $T_{\alpha, \varepsilon} ^{\#} = v_{\alpha, \varepsilon} - \bar{v}_{\alpha, \varepsilon} ^{\#}$ is a tangential opetator, preserves $\dom \(\bar{\partial} ^{*}\)$, and
%  \begin{align*}
%   \left| \( \bar{v}_{\alpha, \varepsilon} ^{k} - \(\bar{v}_{\alpha, \varepsilon} ^{\#}\) ^{k}\) u\right| 
%   \leq \left|1 - \(\frac{d \rho \(v_{\alpha, \varepsilon}\)}{d \rho \(\bar{v}_{\alpha, \varepsilon}\)}\) ^{k} \right| \left| \bar{v}_{\alpha, \varepsilon} ^{k} u\right|.
%  \end{align*}
% Thus, by condition (i) in \autoref{property p},

% \begin{align}\label{T1-T}
%   \left|T_{\alpha, \varepsilon} ^{\#} - T_{\alpha, \varepsilon}\right| \leq \varepsilon \left|v\right| + \varepsilon  \left|\bar{v}\right|
% \end{align}.

\section{Exact regularity for $(0, n - 1)$-forms}

In this section, we are going to prove the exact regularity of $\bar{\partial}$-Neumann operator $N _{n -1}$.

\subsection{More estimates for $(0, n - 1)$-forms}

%In this section, we define property ($\widetilde{P} _{n - 1} ^{\#}$) and derive some basic estimates. We can split 
The following $L^2$ estimate follows from \autoref{modify basic estimate} in \autoref{prel}. %section 2. %in the above lemma.

\begin{prop}\label{condition 2 basic estimate} 
   Let $\Om$ be a bounded smooth pseudoconvex domain in $\mathbb{C}^n$ with $n \geq 3$ %that possesses a family of transverse vector fields satisfying property ($\widetilde{P} _{n-1} ^{\#}$) or property ($P _{n-1} ^{\#}$) 
   and $U$ be a small open neighborhood of some boundary point. Assume that $\lambda \in C^{2} (\cl{\Om})$
  satisfies $\lambda^{\dagger} _{tt} \geq \left( C(\gamma, \eta, n)+10\right) \sum_{j \leq n} \lmdl L_j \lambda \rmdl ^{2}$ on $\partial\Om \cap U$, 
 for some $1 \leq t \leq n-1$. Let $u = \sum'_J u_J ^{\dagger} \bar{w}_J \in L ^{2}_{\(0, n - 1\)}\(\Om\) \cap \dom\(\bar{\partial}\) \cap \dom\(\bar{\partial} ^{*}\)$ with support in $\cl{\Om} \cap U$. Then there exists $C>0$, such that $u _{\text{Tan}} = u _{1, 2, \cdots, n-1}^{\dagger} \bar{w} _{1} \wedge \bar{w} _{2} \wedge \cdots \wedge \bar{w} _{n -1}$ satisfies
  \begin{align*}%\label{MKH of uT}
  %+ \lnm u_T \rnm^2 
    % \gtrsim &  \( \(\lambda_{tt} \)  u _{1, 2, \cdot, n-1}, u _{1, 2, \cdot, n-1}\)
 \int_\Om \lambda^{\dagger} _{tt} \left| u _{1, 2, \cdots, n-1}^{\dagger} \right|^2 \, dV 
    \lesssim  \lnm \bar{\partial} u \rnm ^2 + \lnm \bar{\partial}^* u  \rnm^2 
  \end{align*}
\end{prop}

\begin{proof}
Note that $u_{\text{Tan}}  \in \dom (\bar{\partial}^*)$. %$u$ is a $\(0, n-1\)$-form, then by \autoref{modify basic estimate}, we have $J = \bar{w} _{1} \wedge \bar{w} _{2} \wedge \cdots \wedge \bar{w} _{n -1}$, $iK = \bar{w} _{1} \wedge \bar{w} _{2} \wedge \cdots \wedge \bar{w} _{n -1}$ and satisfies the following inequality
If $\lambda$ satisfies $\lambda^{\dagger} _{tt} \geq \left( C(\gamma, \eta, n)+10\right) \sum_{j \leq n} \lmdl L_j \lambda \rmdl ^{2}$ 
  on $\partial\Om$, then it follows from applying \autoref{modify basic estimate} to $u_{\text{Tan}} $ with $I_s=\{1, 2, \cdots, n-1\} \setminus \{t\}$ that
%  \begin{align*}
  %  & \lnm \bar{\partial} (u_T) \rnm ^2 + \lnm \bar{\partial}^* (u_T)  \rnm^2 + \lnm u_T \rnm^2 \\
  %  \gtrsim & - C\(\gamma, \eta, n\) \sum_{j\leq n} \lnm (L_j \lambda) u_{1,2,\cdots,n-1} \rnm ^2
 %   +  \sum_{j \in I_s} \(\( - 2 \lambda_{jj} \) u_{1,2,\cdots,n-1},  u_{1,2,\cdots,n-1}\)\\
 %   & +  \sum_{i = 1}^{n-1} \(\( 2 \lambda_{ii}\)  u_{1,2,\cdots,n-1}, u_{1,2,\cdots,n-1}\).
 % \end{align*}
%
 % Thus, we have
%
  \begin{align*}
    \lnm \bar{\partial} u  \rnm ^2 + \lnm \bar{\partial}^* u  \rnm^2 
    \geq - C\(\gamma, \eta, n\) \sum_{j\leq n} \lnm (L_j \lambda) u _{1, 2, \cdots, n-1}^{\dagger} \rnm ^2 + 2 \int_\Om \lambda^{\dagger} _{tt}  \left| u _{1, 2, \cdots, n-1}^{\dagger} \right|^2 \, dV
    % \( 2 \lambda^{\dagger} _{tt}  u _{1, 2, \cdots, n-1}^{\dagger},  u _{1, 2, \cdots, n-1}^{\dagger}\) 
    \geq  \int_\Om \lambda^{\dagger} _{tt}  \left| u _{1, 2, \cdots, n-1}^{\dagger} \right|^2 \, dV.
  \end{align*}
%Thus  (\ref{MKH of uT}) follows from the standard Morrey-Kohn-H\"ormander estimates.
%
 %  By condition, $\lambda_{tt} \geq M \sum_{j \leq n} \lmdl L_j \lambda \rmdl ^{2}$, we can let $M$ large enough such that
%
%   \begin{align*}
 %   - C\(\gamma, \eta, n\) \sum_{j\leq n} \lnm (L_j \lambda) u_{(1,2,\cdots,n-1)} \rnm ^2 + \(\(\lambda_{tt} \) u_{1,2,\cdots,n-1},  u_{1,2,\cdots,n-1}\) \geq 0.
%   \end{align*}
%
%If $0 \leq \lambda \leq 1$ on $\Om$, by applying (\ref{wentian}) to $u_T$ with $I_s=\{1, 2, \cdots, n-1\} \setminus \{t\}$, we also get
 % \begin{align*}
 %   \lnm \bar{\partial} (u_T) \rnm ^2 + \lnm \bar{\partial}^* (u_T)  \rnm^2 + \lnm u_T \rnm^2 
 %   \geq C & \int_\Om \lambda_{tt}  \left| u _{1, 2, \cdots, n-1} \right|^2 \, dV, 
 % \end{align*}
%  since the weighted $L^2$ norm is comparable to the  unweighted $L ^{2}$-norm.
\end{proof}

The next two lemmas are inspired by the idea in \cite{H11}.

\begin{lem}\label{lie bracket basic estimate}
  Let $\Om \subset \bbC ^{n}$ with $n \geq 3$ be a smooth bounded pseudoconvex domain that possesses a family of transverse vector fields satisfying property ($\widetilde{P} _{n-1} ^{\#}$). Then 
  \begin{align*}%\label{alice22}
    \lnm \left[D_{X_{\alpha, \varepsilon}} ^{k}, \bar{\partial}^*\right] u\rnm^{2}
    \lesssim \varepsilon^2 k^{2}\(\| \bar{\partial} D_{X_{\alpha, \varepsilon}} ^{k}u \|^2 + \| \bar{\partial}^* D_{X_{\alpha, \varepsilon}}^{k} u \|^2  \) + C _{\varepsilon, k} \(\lnm \bar{\partial} u \rnm _{k -1} ^2 + \lnm \bar{\partial}^*u \rnm _{k - 1}^2 + \lnm u \rnm _{k -1} ^{2}\)
  \end{align*}
  holds for any  $u \in \dom\(\bar{\partial} ^{*}\)\cap C^\infty_{(0, n-1)}(\overline{\Om})$, $k \in \mathbb{N}$.
\end{lem}

\begin{proof}
By applying (\ref{dbar star lie bracket}) and the benign estimates in the first inequality,
 we have 
%  By (\ref{dbar star lie bracket}) and the benign estimates, we have
  \begin{align*}
    \lnm \left[D_{X_{\alpha, \varepsilon}} ^{k}, \bar{\partial}^*\right] u\rnm^2
   & \lesssim k^2  \sideset{}{'}{\sum} _{I} \sum_{j=1}^n \lnm \overline{\theta ^{v _{\alpha, \varepsilon}} _{j}} \left(D_{X_{\alpha, \varepsilon}} ^{k} u\right)_{jI} \rnm^2 + C _{\varepsilon, k}\(\lnm \bar{\partial} u \rnm _{k -1} ^{2}+\lnm \bar{\partial} ^{*} u \rnm_{k -1} ^{2} + \lnm u \rnm_{k -1}^2 \) \\
%  \end{align*}
 % Thus, we only need to estimate $\lnm \sum’ _{I} \sum_{j=1}^n \bar{\theta} ^{v _{\alpha, \varepsilon}} _{j} D_{X_{\alpha, \varepsilon}} ^{k} u_{jI} \rnm ^{2}$. By changing of coordinates and (\ref{change of coordinates}), we have
%  \begin{align*}
  %  \lnm \sideset{}{'}{\sum} _{I} \sum_{j=1}^n \bar{\theta} ^{v _{\alpha, \varepsilon}} _{j} D_{X_{\alpha, \varepsilon}} ^{k} u_{jI} \rnm ^{2}
%   & \leq\lnm \sideset{}{'}{\sum} _{I} \sum_{j=1}^n \(\bar{\theta} ^{v _{\alpha, \varepsilon}} _{j}\)^{\dagger}D_{X_{\alpha, \varepsilon}} ^{k} u_{j I}^{\dagger} \rnm ^{2} + C _{\varepsilon} \lnm u\rnm_{k-1} \\
%  \end{align*}
%  where $u = \sum'_{I} \sum_{j=1}^{n} u _{jI}^{\dagger} \bar{\omega} _{j} \wedge \bar{\omega} _{I} = u _{T} + u _{N}$, where $u _{T}$ is the tangential component of $u$ and $u_{N}$ is the normal component of $u$. Then we have
%  \begin{align*}
%    \lnm \sideset{}{'}{\sum} _{I} \sum_{j=1}^n \(\bar{\theta} ^{v _{\alpha, \varepsilon}} _{j}\)^{\dagger}D_{X_{\alpha, \varepsilon}} ^{k} u_{j I}^{\dagger} \rnm ^{2}
    &\lesssim   k^2 \lnm \left| \theta ^{v _{\alpha, \varepsilon}}\right| \left( D_{X_{\alpha, \varepsilon}} ^{k} u \right)_{\rm Tan}\rnm ^{2}+ C _{\varepsilon, k}\( \lnm \left( D_{X_{\alpha, \varepsilon}} ^{k} u \right) _{\rm Norm}\rnm  ^{2} + \lnm \bar{\partial} u \rnm _{k -1} ^{2}+\lnm \bar{\partial} ^{*} u \rnm_{k -1} ^{2} + \lnm u \rnm_{k -1}^2 \). 
    \end{align*}  
By \autoref{scale} and \autoref{nbhd},  $ \left( C(\gamma, \eta, n)+10\right) \sum_{j \leq n} \lmdl L_j \lambda_{\alpha, \varepsilon} \rmdl ^{2}< \left( \lambda_{\alpha, \varepsilon}\right)^{\dagger} _{t_{\alpha, \varepsilon} t_{\alpha, \varepsilon}}$ 
  on $U_\varepsilon \cap V_{\alpha}$ and  $ \left|  \theta ^{v_{\alpha, \varepsilon}} \right |^2 \lesssim \varepsilon^2   \left( \lambda_{\alpha, \varepsilon}\right)^{\dagger} _{t_{\alpha, \varepsilon} t_{\alpha, \varepsilon}} $ %and $\lmdl \(\theta ^{v_{\alpha, \varepsilon}}\)^{\dagger}_n \rmdl^2 < \varepsilon $ 
  on $U_\varepsilon \cap {\rm supp}(v _{\alpha, \varepsilon})$. Choosing a smooth function $\chi_ \varepsilon$ in $\Omega$ such that $0 \leq \chi_ \varepsilon \leq 1$, ${\rm supp} \chi_ \varepsilon \subset U_\varepsilon \cap \Omega$ and $\chi_ \varepsilon \equiv 1$ near $\partial\Omega$, it follows that 
      \begin{align*}
   \lnm \left[D_{X_{\alpha, \varepsilon}} ^{k}, \bar{\partial}^*\right] u\rnm^2 
    &\lesssim   k^2 \lnm \chi_ \varepsilon \left| \theta ^{v _{\alpha, \varepsilon}}\right| \left( D_{X_{\alpha, \varepsilon}} ^{k} u \right)_{\rm Tan}\rnm ^{2}+ k^2 \lnm (1-\chi_ \varepsilon )\left| \theta ^{v _{\alpha, \varepsilon}}\right| \left( D_{X_{\alpha, \varepsilon}} ^{k} u \right)_{\rm Tan}\rnm ^{2} \\
     &~~~ +     C _{\varepsilon, k} \left(  \lnm \bar{\partial} u \rnm _{k -1} ^{2}+\lnm \bar{\partial} ^{*} u \rnm_{k -1} ^{2} + \lnm u _{\rm Norm}\rnm _{k} ^{2}+  \lnm u \rnm_{k-1} ^{2} \right)\\
      &\lesssim   k^2 \lnm \chi_ \varepsilon \left| \theta ^{v _{\alpha, \varepsilon}}\right| \left( D_{X_{\alpha, \varepsilon}} ^{k} u \right)_{\rm Tan}\rnm ^{2}+ k^2 \lnm (1-\chi_ \varepsilon )\left| \theta ^{v _{\alpha, \varepsilon}}\right|  D_{X_{\alpha, \varepsilon}} ^{k} u \rnm ^{2} \\
     &~~~ +     C _{\varepsilon, k} \left( \lnm \left( D_{X_{\alpha, \varepsilon}} ^{k} u \right) _{\rm Norm}\rnm  ^{2}+ \lnm \bar{\partial} u \rnm _{k -1} ^{2}+\lnm \bar{\partial} ^{*} u \rnm_{k -1} ^{2} +  \lnm u \rnm_{k-1} ^{2} \right)\\
      &\lesssim   k^2 \lnm \chi_ \varepsilon \left| \theta ^{v _{\alpha, \varepsilon}}\right| \left( D_{X_{\alpha, \varepsilon}} ^{k} u \right)_{\rm Tan}\rnm ^{2} \\ 
      &~~~~~  +    C _{\varepsilon, k} \left(  \lnm (1-\chi_ \varepsilon ) u \rnm_k ^{2} + \lnm \bar{\partial} u \rnm _{k -1} ^{2}+\lnm \bar{\partial} ^{*} u \rnm_{k -1} ^{2} +  \lnm u \rnm_{k-1} ^{2} \right) \\
       &\lesssim k^2 \lnm \chi_ \varepsilon \left| \theta ^{v _{\alpha, \varepsilon}}\right| \left( D_{X_{\alpha, \varepsilon}} ^{k} u \right)_{\rm Tan}\rnm ^{2}  + C _{\varepsilon, k} \left(  \lnm \bar{\partial} u \rnm _{k -1} ^{2}+\lnm \bar{\partial} ^{*} u \rnm_{k -1} ^{2} +  \lnm u \rnm_{k-1} ^{2} \right),
       \end{align*}
   where  the second inequality follows from  Corollary \ref{uN estimate} and   the fourth inequality from (\ref{interior estimate}).
   Moreover, 
it follows from
 \autoref{condition 2 basic estimate} that
        \begin{align*}
        \lnm \chi_ \varepsilon \left| \theta ^{v _{\alpha, \varepsilon}}\right| \left( D_{X_{\alpha, \varepsilon}} ^{k} u \right)_{\rm Tan}\rnm ^{2} 
        &\lesssim \varepsilon^2  \int_\Omega  \left( \lambda_{\alpha, \varepsilon}\right)^{\dagger} _{t_{\alpha, \varepsilon} t_{\alpha, \varepsilon}}  \left|  \chi_ \varepsilon \left(D_{X_{\alpha, \varepsilon}} ^{k} u \right)_{\rm Tan}\right| ^{2}\, dV  \\
%     &\lesssim \varepsilon^2  \int_\Omega  \left( \lambda_{\alpha, \varepsilon}\right)^{\dagger} _{t_{\alpha, \varepsilon} t_{\alpha, \varepsilon}}  \left| \left(D_{X_{\alpha, \varepsilon}} ^{k} u \right)_{\rm Tan}\right| ^{2}\, dV  \\
       &\lesssim \varepsilon^2 \(\lnm \bar{\partial} \left(  \chi_ \varepsilon D_{X_{\alpha, \varepsilon}} ^{k}u \right) \rnm^2 + \lnm \bar{\partial}^* \left(  \chi_ \varepsilon D_{X_{\alpha, \varepsilon}}^{k} u \right) \rnm^2 \) \\
       &\lesssim \varepsilon^2 \(\lnm \bar{\partial}  D_{X_{\alpha, \varepsilon}} ^{k}u \rnm^2 + \lnm \bar{\partial}^* D_{X_{\alpha, \varepsilon}}^{k} u  \rnm^2 \) + C_{\varepsilon, k} \| \chi'_ \varepsilon D_{X_{\alpha, \varepsilon}}^{k} u\|^2
  \end{align*}
  for some smooth function $\chi'_ \varepsilon$ with compact support in $\Omega$. By the compactness, for any $\tau >0$, there exists $C_\tau$ such that 
  $$\| \chi'_ \varepsilon D_{X_{\alpha, \varepsilon}}^{k} u\|^2 \leq \tau \| \chi'_ \varepsilon D_{X_{\alpha, \varepsilon}}^{k} u\|_1^2 + C_\tau \| \chi'_ \varepsilon D_{X_{\alpha, \varepsilon}}^{k} u\|^2_{-1}.$$
  It follows from (\ref{interior estimate}) that  
  \begin{align*}
  \| \chi'_ \varepsilon D_{X_{\alpha, \varepsilon}}^{k} u\|^2 &\leq \tau \left( \| \bar\partial \left(\chi'_ \varepsilon D_{X_{\alpha, \varepsilon}}^{k} u\right)\|^2 +\| \bar\partial^* \left( \chi'_ \varepsilon D_{X_{\alpha, \varepsilon}}^{k} u\right)\|^2 \right) + C_\tau C_{\varepsilon, k} \|   u\|^2_{k-1}\\
 & \leq \tau C_{\varepsilon} \left(\| \bar{\partial} D_{X_{\alpha, \varepsilon}} ^{k}u \|^2 + \| \bar{\partial}^* D_{X_{\alpha, \varepsilon}}^{k} u \|^2  \right) + C_\tau C_{\varepsilon, k} \|   u\|^2_{k-1}.
\end{align*}
Therefore, the lemma follows by choosing $\tau C_{\varepsilon} \leq \varepsilon^2$.
\end{proof}

\begin{lem}\label{dbar lie bracket}
  Let $\Om \subset \bbC ^{n}$ with $n \geq 3$ be a smooth bounded pseudoconvex domain that possesses a family of transverse vector fields satisfying property ($\widetilde{P} _{n-1} ^{\#}$). Then 
  \begin{align*}%\label{alice999}
    \lnm \left[\bar{\partial}, v_{\alpha, \varepsilon}\right] D_{X_{\alpha, \varepsilon}} ^{k-1} u\rnm^{2} %\\\nonumber
    \lesssim \varepsilon^{2}\(\lnm \bar{\partial} D_{X_{\alpha, \varepsilon}} ^{k}u \rnm^2 + \lnm \bar{\partial}^* D_{X_{\alpha, \varepsilon}}^{k} u \rnm^2  \) + C _{\varepsilon, k} \(\lnm \bar{\partial} u \rnm _{k -1} ^2 + \lnm \bar{\partial}^*u \rnm _{k - 1}^2 + \lnm u \rnm _{k -1} ^{2}+ {\rm s.c.}  \lnm  u \rnm ^2_{k}\)
  \end{align*}
  holds for any  $u \in \dom\(\bar{\partial} ^{*}\)\cap C^\infty_{(0, n-1)}(\overline{\Om})$, $k \in \mathbb{N}$.
\end{lem}

\begin{proof}  
It follows from (\ref{AAA}), (\ref{BBB}) and property ($\widetilde{P} _{n-1} ^{\#}$) that
  \begin{align}\label{duan123}
   \left | v_{\alpha, \varepsilon}D _{X_{\alpha, \varepsilon}} ^{k - 1}   u \right|
    &\leq  \left | v_{\alpha, \varepsilon} ^{\#} D _{X_{\alpha, \varepsilon}} ^{k - 1}   u\right| + \left| \(v_{\alpha, \varepsilon} - v_{\alpha, \varepsilon} ^{\#}\)D _{X_{\alpha, \varepsilon}} ^{k - 1}   u\right| \\\nonumber
    & \leq \left| D _{v_{\alpha, \varepsilon} ^{\#}} D _{X_{\alpha, \varepsilon}} ^{k - 1}   u \right| +  \left| \(v_{\alpha, \varepsilon} ^{\#}-D _{v_{\alpha, \varepsilon} ^{\#}}\) D _{X_{\alpha, \varepsilon}} ^{k - 1}   u\right| + 2\varepsilon \left| v_{\alpha, \varepsilon} D _{X_{\alpha, \varepsilon}} ^{k - 1}  N _{\delta} u\right| + C _{\varepsilon}\left|  u\right| _{k - 1}   \\\nonumber
    & \leq \left| D _{X_{\alpha, \varepsilon}} ^{k}   u\right| + \left| D _{\bar{v}_{\alpha, \varepsilon}} D _{X_{\alpha, \varepsilon}} ^{k - 1}   u\right| + C _{\varepsilon}\left|   u\right|_{k-1}  + 2\varepsilon \left| v_{\alpha, \varepsilon} D _{X_{\alpha, \varepsilon}} ^{k - 1}   u\right|.
  \end{align}
  Moreover, 
  noting that $\left[\bar{\partial}, v_{\alpha, \varepsilon}\right] - \theta ^{v_{\alpha, \varepsilon}} \wedge v_{\alpha, \varepsilon}$ is a $\(1, 0\)$-tangential derivative, it follows from the benign estimate, the small-large constant inequality and  (\ref{duan123}) that
  \begin{align*}%\label{wt6}
    &\lnm \left[\bar{\partial}, v_{\alpha, \varepsilon}\right] D _{X_{\alpha, \varepsilon}} ^{k-1} u \rnm ^{2}\\\nonumber
    \lesssim &\lnm \theta ^{v_{\alpha, \varepsilon}} \wedge v_{\alpha, \varepsilon} D _{X_{\alpha, \varepsilon}} ^{k-1} u \rnm ^{2} + \lnm \(\left[\bar{\partial}, v_{\alpha, \varepsilon}\right] -\theta ^{v_{\alpha, \varepsilon}} \wedge v_{\alpha, \varepsilon}\)D _{X_{\alpha, \varepsilon}} ^{k-1} u \rnm ^{2}\\\nonumber
    \leq &\lnm \theta ^{v_{\alpha, \varepsilon}} \wedge v_{\alpha, \varepsilon} D _{X_{\alpha, \varepsilon}} ^{k - 1}  u \rnm ^{2} + C_{\varepsilon} \bigg(\lnm \bar{\partial} D _{X_{\alpha, \varepsilon}} ^{k-1} u \rnm ^{2} + \lnm \bar{\partial} ^{*}D _{X_{\alpha, \varepsilon}} ^{k-1}  u \rnm ^{2} + \lnm D _{X_{\alpha, \varepsilon}} ^{k-1} u \rnm \lnm D _{X_{\alpha, \varepsilon}} ^{k-1} u \rnm _{1}\bigg),\\\nonumber
     \lesssim & \lnm \theta ^{v_{\alpha, \varepsilon}} \wedge v_{\alpha, \varepsilon} D _{X_{\alpha, \varepsilon}} ^{k - 1}  u \rnm ^{2} + C _{\varepsilon, k}\left(\lnm \bar{\partial}  u \rnm_{k-1} ^{2} + \lnm \bar{\partial} ^{*}  u \rnm_{k-1} ^{2}  + \lnm   u \rnm_{k-1}^2+ {\rm s.c.} \lnm  u \rnm _{k}^2  \right)\\\nonumber  
      \lesssim &\lnm \left| \theta  ^{v_{\alpha, \varepsilon}} \right| \left| D _{X_{\alpha, \varepsilon}} ^{k}  u  \right| \rnm ^{2} +  C _{\varepsilon, k} \left(\lnm \bar{\partial}  u \rnm_{k-1} ^{2} + \lnm \bar{\partial} ^{*}  u \rnm_{k-1} ^{2}  + \| D _{\bar{v}_{\alpha, \varepsilon}} D _{X_{\alpha, \varepsilon}} ^{k - 1} u \|^2 + \lnm  u\rnm _{k -1} ^{2}+ {\rm s.c.} \lnm  u \rnm _{k}^2 \right) \\\nonumber
       \lesssim &\lnm \left| \theta  ^{v_{\alpha, \varepsilon}} \right| \( D _{X_{\alpha, \varepsilon}} ^{k}  u  \)_{\rm Tan} \rnm ^{2} \\
       &~~~+  C _{\varepsilon, k} \left(\lnm \(   D _{X_{\alpha, \varepsilon}}^{k} u \) _{\rm Norm}  \rnm^{2} +\lnm \bar{\partial}  u \rnm_{k-1} ^{2} + \lnm \bar{\partial} ^{*}  u \rnm_{k-1} ^{2}  +\lnm \bar{\partial} D _{X_{\alpha, \varepsilon}} ^{k-1}  u \rnm ^{2} +\lnm \bar{\partial} ^{*} D _{X_{\alpha, \varepsilon}} ^{k-1}  u   \rnm ^{2} + \lnm  u\rnm _{k -1} ^{2}+ {\rm s.c.} \lnm  u \rnm _{k}^2 \right) \\\nonumber
         \lesssim &\lnm \left| \theta  ^{v_{\alpha, \varepsilon}} \right| \( D _{X_{\alpha, \varepsilon}} ^{k}  u  \)_{\rm Tan} \rnm ^{2} +C _{\varepsilon, k} \left(\lnm \bar{\partial}  u \rnm_{k-1} ^{2} + \lnm \bar{\partial} ^{*}  u \rnm_{k-1} ^{2}   + \lnm  u\rnm _{k -1} ^{2}+ {\rm s.c.} \lnm  u \rnm _{k}^2 \right) .\\\nonumber
     \end{align*}
% , small-large constant inequality and the induction in (\ref{induction}).
  % $\bar{\partial} D _{X_{\alpha, \varepsilon}} ^{k-1} = D _{X_{\alpha, \varepsilon}} ^{k-1}\bar{\partial} + \left[\bar{\partial}, D _{X_{\alpha, \varepsilon}} ^{k-1} \right]$
%  to the third inequality. And using small-large estimate and indction estimate (\ref{induction}) to the last inequality. 
%  Since domain satisfying property ($\widetilde{P} _{n-1} ^{\#}$), 
%By condition (iii) in Definition \ref{property p}, \autoref{uN estimate},  (\ref{induction}) and (\ref{duan123}), we have %$ \lmdl \theta^{v} _{T} \wedge \overline{\theta^{v}_{T}} \rmdl < \varepsilon \left(\lambda_{\varepsilon}\right)_{t_{\alpha} t_{\alpha}}$ and $\lmdl \theta^{v}_{N} \wedge \overline{\theta^{v}_{N}} \rmdl < \varepsilon$, the first term in the above inequality as follows,
Then the lemma follows from applying the same argument to $\lnm \left| \theta  ^{v_{\alpha, \varepsilon}} \right| \( D _{X_{\alpha, \varepsilon}} ^{k}  u  \)_{\rm Tan} \rnm ^{2}$ as in the proof of \autoref{lie bracket basic estimate}.
\end{proof}

\subsection{A priori estimates of $N _{n - 1}$}
%This section will prove that \autoref{property p} implies the priori estimate of $N _{n - 1}$. Firstly, we , 

Now we are ready to prove the following a priori estimates using the method of elliptic regularization. The approach follows \cite{KN65, S08}.

\begin{prop}\label{priori estiamte}
  Let $\Om \subset \bbC ^{n}$ with $n \geq 3$ be a smooth bounded pseudoconvex domain that possesses a family of transverse vector fields satisfying property ($\widetilde{P} _{n-1} ^{\#}$). Then, for every nonnegative integer $k$, there exist $\delta _{0} > 0$, and a positive constant $C$, such that 
   %$\bar{\partial}$-Neumann operator $N _{\delta, n - 1}$ satisfies the Sobolev estimate
  \begin{align*}
   \lnm N_{\delta, n - 1} u \rnm _{k} \leq C \lnm u \rnm _{k}
  \end{align*}
  holds for any $u \in C _{\(0, n - 1\)} ^{\infty} \(\cl{\Om}\)$, $ 0 < \delta \leq \delta _{0}$.
\end{prop} 

\begin{proof}
We simply denote $N_{\delta, n - 1}$ by $N_{\delta}$. 
  Since the exact regularity of $N_{q-1}$ is equivalent to the exact regularity of the Bergman projections $B _{q - 2}, B_{q - 1}, B _{q}$ (cf. \cite{BS1990}) and since $N _{n}$ is exactly  regularity, it follows that $B _{n - 1}$ is also exactly regular. Thus \autoref{B regular} yields
  \begin{align}\label{Nu}
    \left\|N_{\delta} u\right\|_k 
    \leq C _{k} \left(\left\|\bar{\partial} N_{\delta} u\right\|_k+\left\|\bar{\partial}^* N_{\delta} u\right\|_k\right) .
  \end{align}
  We claim that, for every nonnegative integer $\beta$, there exist $\delta _{0} > 0$, and a constant $C$, such that 
  \begin{align}\label{induction}
    \lnm \bar{\partial} N _{\delta} u \rnm _{\beta} ^{2} + \lnm \bar{\partial} ^{*} N _{\delta}u \rnm _{\beta}^{2} + \delta\lnm \nabla N _{\delta} u \rnm _{\beta}^{2}
    \leq C \lnm u \rnm _{\beta} ^{2}
  \end{align}
  holds for any $u \in C _{\(0, n - 1\)} ^{\infty} \(\cl{\Om}\), 0 < \delta \leq \delta _{0}$.
     Prove the claim using the induction on $\beta$. 
  When $\beta = 0$,
  \begin{align*}
    \lnm \bar{\partial} N _{\delta} u \rnm ^{2} + \lnm \bar{\partial} ^{*} N _{\delta}u \rnm ^{2} + \delta\lnm \nabla N _{\delta} u \rnm ^{2}
    = \(\Box _{\delta} N _{\delta} u, N _{\delta} u\)
    = \(u, N _{\delta} u\)
    \leq C \lnm u \rnm ^{2}.
  \end{align*} 
  Now assume that (\ref{induction}) holds for all $0 \leq \beta \leq k - 1$. We aim to prove (\ref{induction})  for $\beta = k$.

 First we deal with  the term $\delta\lnm \nabla N _{\delta} u \rnm ^{2}$. Since the normal derivative $\frac{\partial}{\partial \nu} = L_{n} + \bar{L} _{n} = (L_{n} - \bar{L} _{n}) + 2\bar{L} _{n}$ with $L_{n} - \bar{L} _{n}$ being tangential,  
   $\nabla$ can be written as the sum of tangential derivatives $ \nabla _{T}$ and the $\(0, 1\)$-derivative $\bar{L} _{n}$, where $T$ denotes an arbitrary tangential vector field. 
  % + \frac{\partial}{\partial \nu}$, 
%  where $\frac{\partial}{\partial \nu}$ can be expressed as $\frac{\partial}{\partial \nu} = L_{n} + \bar{L} _{n} = L_{n} - \bar{L} _{n} + 2\bar{L} _{n}$. It follows that $\nabla N _{\delta} u = \nabla _{T} N _{\delta} u$ plus $\(0, 1\)$-derivative of $N _{\delta} u$. 
By applying the benign estimates in \autoref{benign estimate} and (\ref{Nu}), we obtain
  \begin{align}\label{wt1}
    \delta\lnm \nabla N _{\delta} u \rnm _{k} ^{2}
    \lesssim & \delta \lnm \nabla _{T} N _{\delta} u \rnm _{k} ^{2} + \delta \lnm \bar{L} _{n}  N _{\delta} u \rnm _{k} ^{2}\\\nonumber
    \leq&\delta \lnm \nabla _{T} N _{\delta} u \rnm _{k} ^{2} + \delta C \(\lnm \bar{\partial} N _{\delta} u\rnm _{k} ^{2} + \lnm \bar{\partial} ^{*} N _{\delta} u\rnm _{k} ^{2} + \lnm N _{\delta} u\rnm _{k} ^{2}\) \\\nonumber
    \leq&\delta \lnm \nabla _{T} N _{\delta} u \rnm _{k} ^{2} + \delta C_{k} \(\lnm \bar{\partial} N _{\delta} u\rnm _{k} ^{2} + \lnm \bar{\partial} ^{*} N _{\delta} u\rnm _{k} ^{2}\). \nonumber
  \end{align}
%Since, by choosing $\delta$ small enough,  the last two terms can be observed by the left hand side in (\ref{induction}), 
%Note that by \autoref{nbhd}, we only need to do the estimate . 
By \autoref{nbhd}, we may assume that (i),(ii),(iii) in \autoref{property p} hold in $U_\varepsilon \cap V_\alpha$ for all $\alpha$. Then it suffices to estimate the first term: %as follows. 
  \begin{align*}
     \delta \lnm \nabla _{T} N _{\delta} u \rnm _{k} ^{2} 
    \leq \delta C_{\varepsilon} \( \sum_{\alpha} \lnm v_{\alpha, \varepsilon} ^{k} \nabla _{T} N _{\delta} u \rnm ^{2} + \lnm Y \nabla _{T} N _{\delta} u \rnm _{k - 1} ^{2} + \sum _{j = 1} ^{n} \lnm \frac{\partial}{\partial \bar{z} _{j}} \nabla _{T} N _{\delta} u \rnm _{k - 1} ^{2} +      \|  N _{\delta} u \|^2_{k} \),
  \end{align*}
  where $Y$ is a tangential vector field of $(1,0)$-type.
  % , and apply the benign estimate
  % \delta = \frac{1}{2C}
 By \autoref{s.b.c comparable}, %and benign estimate to the second and third terms since $TN _{\delta} u \in \dom\(\bar{\partial} ^{*}\)$, 
  \begin{align*}
    \lnm Y \nabla _{T} N _{\delta} u \rnm _{k - 1} ^{2} %\\\nonumber 
    \lesssim  & \lnm Y \nabla^{(b)}_{T} N _{\delta} u \rnm _{k - 1} ^{2} + \|N_{\delta} u\|^2_k \\
  %  \lesssim &\lnm Y T N _{\delta} u \rnm _{k - 1} ^{2} + \lnm \nabla _{T} N _{\delta} u \rnm _{k - 1} ^{2}\\
    \lesssim & \lnm \bar{\partial} \nabla^{(b)}_{T} N _{\delta} u\rnm _{k - 1} ^{2} + \lnm \bar{\partial} ^{*} \nabla^{(b)}_{T} N _{\delta} u\rnm _{k - 1} ^{2} + \lnm \nabla^{(b)}_{T} N _{\delta} u\rnm _{k - 1} \lnm \nabla^{(b)}_{T} N _{\delta} u\rnm _{k} + \lnm N _{\delta} u \rnm _{k } ^{2} \\
    \lesssim & \lnm  \nabla^{(b)}_{T} \bar{\partial} N _{\delta} u\rnm _{k - 1} ^{2} + \lnm \nabla^{(b)}_{T}\bar{\partial} ^{*} N _{\delta} u\rnm _{k - 1} ^{2} +\lnm [\bar\partial, \nabla^{(b)}_{T}] N _{\delta} u\rnm _{k - 1}^{2} + \lnm [\bar\partial ^{*}, \nabla^{(b)}_{T}] N _{\delta} u\rnm _{k - 1} ^{2} \\
    &+ \lnm \nabla^{(b)}_{T} N _{\delta} u\rnm _{k - 1} \lnm \nabla^{(b)}_{T} N _{\delta} u\rnm _{k} + \lnm  N _{\delta} u \rnm _{k } ^{2} \\
    \leq &C_{\varepsilon, k} \bigg(\lnm \bar{\partial} N _{\delta} u\rnm _{k} ^{2} + \lnm \bar{\partial} ^{*} N _{\delta} u\rnm _{k} ^{2} \bigg) + \text{l.c.}\lnm N _{\delta} u\rnm _{k}^{2} + \text{s.c.} \lnm \nabla^{(b)}_{T}  N _{\delta} u\rnm^2 _{k} \\
    \leq & C_{\varepsilon, k}\bigg(\lnm \bar{\partial} N _{\delta} u\rnm _{k} ^{2} + \lnm \bar{\partial} ^{*} N _{\delta} u\rnm _{k} ^{2} \bigg) + \text{s.c.} \(\lnm \nabla_{T} N _{\delta} u\rnm _{k} ^{2} + \lnm N _{\delta} u\rnm _{k} ^{2}\)\\
    \leq& C_{\varepsilon, k}\bigg(\lnm \bar{\partial} N _{\delta} u\rnm _{k} ^{2} + \lnm \bar{\partial} ^{*} N _{\delta} u\rnm _{k} ^{2} \bigg) + \text{s.c.} \lnm \nabla_{T} N _{\delta} u\rnm _{k} ^{2},
  \end{align*}
  where we apply the benign estimates in the second inequality, and  the small-large constant inequality in the fourth inequality. By the similar argument, we also have
  \begin{align*}
    \sum _{j = 1} ^{n} \lnm \frac{\partial}{\partial \bar{z} _{j}} \nabla _{T} N _{\delta} u \rnm _{k - 1} ^{2}
    \lesssim C_{\varepsilon, k}\bigg(\lnm \bar{\partial} N _{\delta} u\rnm _{k} ^{2} + \lnm \bar{\partial} ^{*} N _{\delta} u\rnm _{k} ^{2} \bigg).
  \end{align*}
  Note that the only difference here is that  there is no $\lnm \nabla^{(b)}_{T} N _{\delta} u\rnm _{k} ^{2}$, thus no $\lnm \nabla_{T} N _{\delta} u\rnm _{k} ^{2}$ terms on the right hand side, when applying the benign estimate. 
Therefore, it follows that
  \begin{align}\label{wentian26}
    &\delta \lnm \nabla _{T} N _{\delta} u \rnm _{k} ^{2} 
    \leq \delta C_{\varepsilon, k} \left( \sum_\alpha \lnm v_{\alpha, \varepsilon} ^{k} \nabla _{T} N _{\delta} u \rnm ^{2}+ \lnm \bar{\partial} N _{\delta} u\rnm _{k} ^{2} + \lnm \bar{\partial} ^{*} N _{\delta} u\rnm _{k} ^{2}  +\text{s.c.} \lnm \nabla_{T} N _{\delta} u\rnm _{k} ^{2} \right).   \end{align}
  Combining (\ref{wt1}), (\ref{wentian26}) and  choosing s.c. small enough such that  $C_{\varepsilon, k} \cdot \text{s.c.} \ll 1$, we have 
  \begin{align}\label{wt2}
    \delta \lnm \nabla  N _{\delta} u \rnm _{k} ^{2} 
    \lesssim \delta C_{\varepsilon,k} \bigg( \sum_\alpha \lnm v_{\alpha, \varepsilon} ^{k} \nabla _{T} N _{\delta} u \rnm ^{2}+ \lnm \bar{\partial} N _{\delta} u\rnm _{k} ^{2} + \lnm \bar{\partial} ^{*} N _{\delta} u\rnm _{k} ^{2} \bigg).
  \end{align}
 %  We now estimate the second inequality using the benign estimate, and the third term by applying a Lie bracket identity. Then, applying the small-large estimate, along with (\ref{Nu}) and induction estimate (\ref{induction}), we have
 %  \begin{align*}
 %    &\delta \lnm \nabla _{T} N _{\delta} u \rnm _{k} ^{2} \\\nonumber
 %    \leq& \delta C_{\varepsilon} \lnm v_{\alpha, \varepsilon} ^{k} \nabla _{T} N _{\delta} u \rnm ^{2}+ \delta C_{\varepsilon}\bigg(\lnm \bar{\partial} N _{\delta} u\rnm _{k} ^{2} + \lnm \bar{\partial} ^{*} N _{\delta} u\rnm _{k} ^{2} \bigg) + \text{s.c.} \delta C_{\varepsilon} \lnm \nabla _{T} N _{\delta} u\rnm _{k}^{2}+ \text{l.c.} \delta C_{\varepsilon}\lnm \nabla _{T} N _{\delta} u\rnm _{k - 1} ^{2}\\\nonumber
 %    \leq& \delta C_{\varepsilon} \lnm v_{\alpha, \varepsilon} ^{k} \nabla _{T} N _{\delta} u \rnm ^{2}+ \delta C_{\varepsilon}\bigg(\lnm \bar{\partial} N _{\delta} u\rnm _{k} ^{2} + \lnm \bar{\partial} ^{*} N _{\delta} u\rnm _{k} ^{2} \bigg)+ \text{s.c.} \delta C_{\varepsilon}\lnm \nabla _{T} N _{\delta} u\rnm _{k}^{2}+ \text{l.c.} \delta C_{\varepsilon}\lnm u\rnm _{k - 1}^{2},\nonumber
 %  \end{align*}
 %  %s.c.= \frac{1}{2} l.c.=4
 % where the second and the third term can be observed to the left-side hand by choosing $\delta$ and small constant small enough. 
 Thus, it remains to estimate  %$ \delta \lnm v_{\alpha, \varepsilon} ^{k} \nabla _{T} N _{\delta} u \rnm ^{2}$,
 \begin{align}\label{gradient1}
    & \delta \lnm  v_{\alpha, \varepsilon} ^{k} \nabla _{T} N _{\delta} u \rnm  ^{2} \\\nonumber
    \lesssim& \delta\lnm \(v_{\alpha, \varepsilon} ^{k} - \(v_{\alpha, \varepsilon}^{\#}\) ^{k}\)\nabla _{T} N _{\delta} u \rnm ^{2} + \delta\lnm \(v_{\alpha, \varepsilon}^{\#}\) ^{k} \nabla _{T} N _{\delta} u \rnm ^{2}\\\nonumber
    \lesssim & \delta\lnm \(v_{\alpha, \varepsilon}^{\#}\) ^{k} \nabla _{T} N _{\delta} u \rnm ^{2} + 4k ^{2}\varepsilon^{2} \delta\lnm  v_{\alpha, \varepsilon} ^{k} \nabla _{T} N _{\delta} u \rnm ^{2}+\delta C_{\varepsilon} \lnm \nabla _{T} N _{\delta} u \rnm _{k - 1}^{2}\\\nonumber
    \lesssim &\delta\lnm D _{v_{\alpha, \varepsilon}^{\#}} ^{k} \nabla _{T} N _{\delta} u \rnm ^{2} + \delta\lnm \(\(v_{\alpha, \varepsilon}^{\#}\)^{k} - D _{v_{\alpha, \varepsilon}^{\#}} ^{k}\) \nabla _{T} N _{\delta} u \rnm ^{2} + 4k ^{2}\varepsilon^{2} \delta \lnm v_{\alpha, \varepsilon} ^{k}\nabla _{T} N _{\delta} u \rnm ^{2} +\delta C_{\varepsilon} \lnm \nabla _{T} N _{\delta} u \rnm _{k - 1}^{2}  \\\nonumber
    \lesssim &\delta\lnm D _{v_{\alpha, \varepsilon}^{\#}} ^{k} \nabla _{T} N _{\delta} u \rnm ^{2}  + 4k ^{2}\varepsilon^{2} \delta\lnm v_{\alpha, \varepsilon} ^{k} \nabla _{T} N _{\delta} u \rnm ^{2}+\delta C_{\varepsilon} \lnm \nabla _{T} N _{\delta} u \rnm _{k - 1}^{2} \\\nonumber
    \leq& \delta\lnm D _{X_{\alpha, \varepsilon}} ^{k} \nabla _{T} N _{\delta} u \rnm ^{2} + \delta C_{\varepsilon}\lnm D_{\bar{v}_{\alpha, \varepsilon}} \nabla _{T} N _{\delta} u \rnm_{k - 1}^{2} + \delta C_{\varepsilon}\lnm  N _{\delta} u \rnm _{k}^{2} + 4k ^{2}\varepsilon^{2} \delta \lnm v_{\alpha, \varepsilon} ^{k}\nabla _{T} N _{\delta} u \rnm ^{2},\nonumber
  \end{align}
where  the second inequality follows from (\ref{BBB}), condition (i) in \autoref{property p} and the fourth inequality is obtained by applying (\ref{DZ-1111}). By the definition of $D_{\bar{v}_{\alpha, \varepsilon}}$, the highest order term is a $\(0, 1\)$-derivative. It then follows from \autoref{s.b.c comparable} and similar argument that %as before, and then still apply Lie-bracket identity and the benign estimate to $T N _{\delta} u \in \dom\(\partial ^{*}\) $,
  \begin{align*}
    &\delta C _{\varepsilon} \lnm D_{\bar{v}_{\alpha, \varepsilon}} \nabla _{T} N _{\delta} u \rnm_{k - 1}^{2}  \lesssim \delta C _{\varepsilon} \(\lnm \bar{v}_{\alpha, \varepsilon} \nabla_T^{(b)} N _{\delta} u \rnm_{k - 1}^{2} + \lnm  N _{\delta} u \rnm_{k }^{2} \) 
  %  \lesssim &\delta C _{\varepsilon} \(\lnm D_{\bar{v}_{\alpha, \varepsilon}} T N _{\delta} u \rnm_{k - 1}^{2} + \lnm  N _{\delta} u \rnm_{k }^{2} \)\\
  %  \leq & \delta C _{\varepsilon}\(\lnm \bar{\partial} T N _{\delta} u \rnm_{k - 1}^{2} + \lnm \bar{\partial} ^{*} T N _{\delta} u \rnm_{k - 1}^{2} + \lnm T N _{\delta} u \rnm_{k - 1}^{2} + \lnm  N _{\delta} u \rnm_{k }^{2}\)\\
 %   \leq & \delta C _{\varepsilon}\(\lnm\bar{\partial}  N _{\delta} u \rnm_{k}^{2} +  \lnm \bar{\partial} ^{*}  N _{\delta} u \rnm_{k}^{2} +\lnm \left[\bar{\partial}, T\right] N _{\delta} %u \rnm_{k - 1}^{2} + \lnm \left[\bar{\partial} ^{*}, T\right] N _{\delta} u \rnm_{k - 1}^{2} + \lnm N _{\delta} u \rnm_{k}^{2}\)\\
   \leq \delta C _{\varepsilon, k}\(\lnm \bar{\partial}  N _{\delta} u \rnm_{k}^{2} + \lnm \bar{\partial} ^{*} N _{\delta} u \rnm_{k }^{2} \).
  \end{align*}
  %\delta= \frac{1}{4C _{\varepsilon}}
  % where we apply (\ref{Nu}) to the last inequality. 
  By choosing $\varepsilon$ small enough such that $4 k^2 \varepsilon^2 \ll 1$, we get 
%  Apply the induction estimate to the lower-order terms, then
  \begin{align}\label{nabla}
    \delta \lnm  v_{\alpha, \varepsilon} ^{k} \nabla _{T} N _{\delta} u \rnm  ^{2}
     \lesssim \delta\lnm D _{X_{\alpha, \varepsilon}} ^{k} \nabla _{T} N _{\delta} u \rnm ^{2} + \delta C _{\varepsilon, k}\(\lnm \bar{\partial}  N _{\delta} u \rnm_{k}^{2} + \lnm \bar{\partial} ^{*} N _{\delta} u \rnm_{k }^{2}\).
  \end{align}
By the similar argument, we get %applies to the term $\lnm \bar{\partial} ^{*} N _{\delta} u \rnm _{k} ^{2}$.
  \begin{align*}
   \lnm \bar{\partial} ^{*} N _{\delta} u \rnm _{k} ^{2}   & \leq C_{\varepsilon}\(\sum_\alpha \lnm v_{\alpha, \varepsilon} ^{k} \bar{\partial} ^{*} N _{\delta} u \rnm  ^{2} + \lnm Y \bar{\partial} ^{*} N _{\delta} u \rnm _{k - 1} ^{2} + \sum _{j = 1} ^{n} \lnm \frac{\partial}{\partial \bar{z} _{j}}\bar{\partial} ^{*} N _{\delta} u \rnm _{k - 1} ^{2} +  \lnm \bar{\partial} ^{*} N _{\delta} u \rnm _{k-1} ^{2} \)\\\nonumber
   &  \lesssim C_{\varepsilon} \(\sum_\alpha \lnm v_{\alpha, \varepsilon} ^{k} \bar{\partial} ^{*} N _{\delta} u \rnm  ^{2} + \lnm \bar{\partial} \bar{\partial} ^{*} N _{\delta} u \rnm _{k - 1} ^{2} +  \lnm\bar{\partial} ^{*} N _{\delta} u \rnm _{k - 1} ^{2} + \lnm\bar{\partial} ^{*} N _{\delta} u \rnm _{k - 1}\lnm\bar{\partial} ^{*} N _{\delta} u \rnm _{k}\)\\\nonumber
   &  \leq C_{\varepsilon, k}\(\sum_\alpha \lnm v_{\alpha, \varepsilon} ^{k} \bar{\partial} ^{*} N _{\delta} u \rnm  ^{2} + %\lnm \bar{\partial} \bar{\partial} ^{*} N _{\delta} u \rnm _{k } ^{2} + 
   \lnm u \rnm _{k } ^{2}\) + \text{s.c.}C_{\varepsilon} \lnm\bar{\partial} ^{*} N _{\delta} u \rnm _{k} ^{2},\nonumber
  \end{align*}
 where the second inequality follows from the benign estimates and the third inequality from the small-large constant inequality, the induction estimates and the following equality: %We now apply the \autoref{dbar dbar*} to the second term to obtain
  \begin{align}\label{dbar dbar star2}
    \lnm \vartheta \bar{\partial}  N _{\delta} u \rnm _{k - 1} ^{2} + \lnm \bar{\partial} \bar{\partial} ^{*} N _{\delta} u \rnm _{k - 1} ^{2}
    \leq C_{k} \(\lnm \bar{\partial}  N _{\delta} u \rnm _{k - 1} ^{2} + \lnm u \rnm _{k - 1} ^{2} + \lnm u \rnm _{1} ^{2} + \delta \lnm \nabla N _{\delta} u\rnm _{k - 1} ^{2} + \delta ^{2} \lnm u \rnm _{k} ^{2}\) 
    \leq C _{k} \lnm u \rnm _{k} ^{2} 
  \end{align}
  by  \autoref{B regular}.
  Again choosing s.c. small enough with s.c.$C_{\varepsilon} \ll 1$,  we obtain
 \begin{align}\label{wt3}
    \lnm \bar{\partial} ^{*} N _{\delta} u \rnm _{k} ^{2}
    \lesssim C_{\varepsilon, k} \( \sum_\alpha\lnm v_{\alpha, \varepsilon} ^{k} \bar{\partial} ^{*} N _{\delta} u \rnm  ^{2} + \lnm u \rnm _{k} ^{2} \) .
  \end{align}
By the similar argument, we also have %used for estimating $ \delta \lnm v_{\alpha, \varepsilon} ^{k} \nabla _{T} N _{\delta} u \rnm ^{2}$, which yields
   \begin{align}\label{dbar star}
    \lnm v_{\alpha, \varepsilon} ^{k} \bar{\partial} ^{*} N _{\delta} u \rnm ^{2}
    \lesssim \lnm D_{X_{\alpha, \varepsilon}}^{k}\bar{\partial} ^{*} N _{\delta} u \rnm ^{2} + C _{\varepsilon, k} \lnm u \rnm _{k} ^{2} .
  \end{align}
  % Applying the above inequality to (\ref{1}), we have
  % \begin{align}\label{dbar star}
  %   \lnm \bar{\partial} ^{*} N _{\delta} u \rnm _{k} ^{2}
  %   \leq& \lnm \(D _{X}\) ^{k} \bar{\partial} ^{*} N _{\delta} u \rnm ^{2} + C _{\varepsilon, k} \(\lnm \bar{\partial} \bar{\partial} ^{*} N _{\delta} u \rnm _{k - 1} ^{2} + \lnm u \rnm _{k - 1} ^{2}\).
  % \end{align}
Now, for the term $\lnm \bar{\partial} N _{\delta} u\rnm _{k} ^{2}$, %since $\bar{\partial} N _{\delta} u$ is not in $\dom\(\bar{\partial} ^{*}\)$, we can not apply the benign estimate in \autoref{benign estimate}. Instead, 
we apply the similar argument  to $\bar{\partial} N _{\delta} u + \delta \frac{\partial}{\partial \nu} N _{\delta} u \wedge \bar{\omega} _{n} \in \dom \(\bar{\partial} ^{*}\)$ and get 
  \begin{align}\label{2}
    & \lnm \bar{\partial} N _{\delta} u \rnm _{k} ^{2} \\\nonumber
     \leq & C _{\varepsilon}\bigg(\sum_\alpha\lnm v_{\alpha, \varepsilon} ^{k} \bar{\partial}N _{\delta} u\rnm ^{2} + \lnm Y \bar{\partial}N _{\delta} u \rnm _{k - 1} ^{2} + \sum _{j = 1} ^{n} \lnm \frac{\partial}{\partial \bar{z} _{j}}\bar{\partial}N _{\delta} u \rnm _{k - 1} ^{2}+  \lnm \bar{\partial} N _{\delta} u \rnm _{k-1} ^{2}\bigg)\\\nonumber
    \lesssim & C _{\varepsilon}\bigg(\sum_\alpha \lnm v_{\alpha, \varepsilon} ^{k} \bar{\partial}N _{\delta} u\rnm ^{2} + \lnm Y \(\bar{\partial}N _{\delta} u + \delta \frac{\partial}{\partial \nu} N _{\delta} u\wedge \bar{\omega} _{n}\) \rnm _{k - 1} ^{2} + \sum _{j = 1} ^{n} \lnm \frac{\partial}{\partial \bar{z} _{j}}\(\bar{\partial}N _{\delta} u + \delta \frac{\partial}{\partial \nu} N _{\delta} u\wedge \bar{\omega} _{n}\)\rnm _{k - 1} ^{2}\\\nonumber
    &+  \lnm Y \( \delta \frac{\partial}{\partial \nu} N _{\delta} u\wedge \bar{\omega} _{n}\) \rnm _{k - 1} ^{2} + \sum _{j = 1} ^{n} \lnm \frac{\partial}{\partial \bar{z} _{j}}\(\delta \frac{\partial}{\partial \nu} N _{\delta} u\wedge \bar{\omega} _{n}\)\rnm _{k - 1} ^{2} + \|u\|^2_{k-1}\bigg)\\\nonumber
    \leq & C _{\varepsilon}\bigg(\sum_\alpha\lnm v_{\alpha, \varepsilon} ^{k} \bar{\partial}N _{\delta} u\rnm ^{2} + \lnm Y \(\bar{\partial}N _{\delta} u + \delta \frac{\partial}{\partial \nu} N _{\delta} u\wedge \bar{\omega} _{n}\) \rnm _{k - 1} ^{2} + \sum _{j = 1} ^{n} \lnm \frac{\partial}{\partial \bar{z} _{j}}\(\bar{\partial}N _{\delta} u + \delta \frac{\partial}{\partial \nu} N _{\delta} u\wedge \bar{\omega} _{n}\)\rnm _{k - 1} ^{2}  \\\nonumber
    &+\delta ^{2}\lnm \nabla N _{\delta} u \rnm _{k} ^{2} + \|u\|^2_{k-1} \bigg)\\\nonumber 
    \lesssim & C _{\varepsilon}\bigg(\sum_\alpha\lnm v_{\alpha, \varepsilon} ^{k} \bar{\partial}N _{\delta} u\rnm ^{2} + \text{l.c.} \|u\|^2_k + \text{l.c.} \delta ^{2}\lnm \nabla N _{\delta} u \rnm _{k} ^{2} + \text{s.c.}  \lnm \bar{\partial} N _{\delta} u \rnm _{k} ^{2} \bigg), \nonumber
  \end{align}
where,  by the similar argument as above, the last inequality follows from
  \begin{align*}
    &\lnm Y \(\bar{\partial}N _{\delta} u + \delta \frac{\partial}{\partial \nu} N _{\delta} u\wedge \bar{\omega} _{n}\) \rnm _{k - 1} ^{2} + \sum _{j = 1} ^{n} \lnm \frac{\partial}{\partial \bar{z} _{j}}\(\bar{\partial}N _{\delta} u + \delta \frac{\partial}{\partial \nu} N _{\delta} u\wedge \bar{\omega} _{n}\)\rnm _{k - 1} ^{2}\\
  %  \leq & C _{\varepsilon} \bigg( \lnm \vartheta \bar{\partial} N _{\delta} u \rnm _{k - 1} ^{2} + \lnm \bar{\partial} N _{\delta} u\rnm _{k - 1} ^{2} + \lnm \bar{\partial} N _{\delta} u\rnm _{k}\lnm \bar{\partial} N _{\delta} u\rnm _{k- 1} + \delta ^{2}\lnm \bar{\partial} \frac{\partial}{\partial \nu} N _{\delta} u \wedge \bar{\omega} _{n} \rnm _{k - 1} ^{2} \\
%    &+\delta ^{2}\lnm \bar{\partial} ^{*} \frac{\partial}{\partial \nu} N _{\delta} u \wedge \bar{\omega} _{n} \rnm _{k - 1} ^{2} + \delta ^{2}\lnm \frac{\partial}{\partial \nu} N _{\delta} u \wedge \bar{\omega} _{n} \rnm _{k-1} \lnm \frac{\partial}{\partial \nu} N _{\delta} u \wedge \bar{\omega} _{n} \rnm _{k} + \delta ^{2}\lnm \frac{\partial}{\partial \nu} N _{\delta} u \wedge \bar{\omega} _{n} \rnm _{k-1} ^{2}\\
 %   & + \delta \lnm \bar{\partial} N _{\delta} u\rnm _{k} \lnm \frac{\partial}{\partial \nu} N _{\delta} u \wedge \bar{\omega} _{n} \rnm _{k - 1} + \delta \lnm \bar{\partial} N _{\delta} u\rnm _{k - 1} \lnm \frac{\partial}{\partial \nu} N _{\delta} u \wedge \bar{\omega} _{n} \rnm _{k} \bigg)\\
  %  \lesssim & C _{\varepsilon} \bigg(\lnm \vartheta \bar{\partial} N _{\delta} u \rnm _{k - 1} ^{2} + \lnm \bar{\partial} N _{\delta} u\rnm _{k - 1} ^{2} + \lnm \bar{\partial} N _{\delta} u\rnm _{k}\lnm \bar{\partial} N _{\delta} u\rnm _{k- 1}+ \delta ^{2}\lnm \nabla N _{\delta} u \rnm _{k} ^{2} \\
 %   &+ \delta \lnm \bar{\partial} N _{\delta} u\rnm _{k} \lnm \nabla N _{\delta} u \rnm _{k - 1} + \delta \lnm \bar{\partial} N _{\delta} u\rnm _{k - 1} \lnm \nabla N _{\delta} u \rnm _{k} \bigg)\\
    \lesssim &  \lnm \vartheta \bar{\partial} N _{\delta} u \rnm _{k - 1} ^{2} + \text{l.c.} \lnm \bar{\partial} N _{\delta}u\rnm _{k - 1} ^{2} + \text{l.c.} \delta ^{2}\lnm \nabla N _{\delta} u \rnm _{k} ^{2} + \text{s.c.}  \lnm \bar{\partial} N _{\delta} u \rnm _{k} ^{2} \\\nonumber
    \lesssim & \text{l.c.} \|u\|^2_k + \text{l.c.} \delta ^{2}\lnm \nabla N _{\delta} u \rnm _{k} ^{2} + \text{s.c.}  \lnm \bar{\partial} N _{\delta} u \rnm _{k} ^{2} .\\\nonumber
  \end{align*} 
 By choosing s.c. sufficiently small such that $\text{s.c.} \cdot C _{\varepsilon} \ll 1$, then we obtain 
%  Applying (\ref{dbar dbar star2}) to the second term, and the last two terms can be observed to  the left-hand side by choosing $\delta$ and small constant small enough yields,
   \begin{align}\label{wt4}
    \lnm \bar{\partial} N _{\delta} u\rnm _{k} ^{2}
    % \leq &\lnm v_{\alpha, \varepsilon} ^{k} \(\bar{\partial}N _{\delta} u + \delta \frac{\partial}{\partial \nu} N _{\delta} u\wedge \bar{\omega} _{n}\)\rnm  ^{2} + C _{\varepsilon} \bigg(\lnm Y \(\bar{\partial}N _{\delta} u + \delta \frac{\partial}{\partial \nu} N _{\delta} u\wedge \bar{\omega} _{n}\) \rnm _{k - 1} ^{2} \\\nonumber
    \lesssim C _{\varepsilon} \sum_\alpha  \lnm v_{\alpha, \varepsilon} ^{k} \bar{\partial}N _{\delta} u \rnm ^{2} + C _{\varepsilon, k}   \lnm u\rnm _{k} ^{2} + \delta ^{2}C _{\varepsilon}\lnm \nabla N _{\delta} u \rnm _{k} ^{2}.
  \end{align}
Moreover, we have 
  \begin{align*}
  \lnm v_{\alpha, \varepsilon} ^{k} \bar{\partial}N _{\delta} u\rnm  ^{2}
%   \leq &\lnm \(v_{\alpha, \varepsilon} ^{k} - \(v_{\alpha, \varepsilon}^{\#}\) ^{k}\) \bar{\partial}N _{\delta} u  \rnm ^{2} + \lnm \(v_{\alpha, \varepsilon}^{\#}\) ^{k} \bar{\partial}N _{\delta} u  \rnm ^{2}\\
 %   \leq &\lnm D _{v_{\alpha, \varepsilon}^{\#}} ^{k} \bar{\partial}N _{\delta} u\rnm  ^{2} + \lnm \(\(v_{\alpha, \varepsilon}^{\#}\) ^{k}- D _{v_{\alpha, \varepsilon}^{\#}} ^{k}\) \bar{\partial}N _{\delta} u \rnm ^{2} + 4k^{2}\varepsilon ^{2}\lnm v_{\alpha, \varepsilon} ^{k} \bar{\partial}N _{\delta} u\rnm  ^{2} +C _{\varepsilon} \lnm \bar{\partial}N _{\delta} u \rnm _{k - 1} ^{2} \\
    \lesssim &\lnm D _{X _{\alpha, \varepsilon}}^{k} \bar{\partial}N _{\delta} u\rnm  ^{2} + C _{\varepsilon}\lnm \bar{v}_{\alpha, \varepsilon} \bar{\partial}N _{\delta} u\rnm _{k - 1} ^{2} + C _{\varepsilon}\lnm \bar{\partial}N _{\delta} u \rnm _{k - 1} ^{2} + 4k^{2}\varepsilon ^{2}\lnm v_{\alpha, \varepsilon} ^{k} \bar{\partial}N _{\delta} u\rnm  ^{2},
  \end{align*}
and 
  \begin{align*}
    \lnm \bar{v}_{\alpha, \varepsilon} \bar{\partial}N _{\delta} u\rnm _{k - 1} ^{2}
    &\lesssim \lnm \bar{v}_{\alpha, \varepsilon} \(\bar{\partial}N _{\delta} u + \delta \frac{\partial}{\partial \nu} N _{\delta} u \wedge \bar{\omega} _{n}\)\rnm _{k - 1} ^{2} + \delta ^{2} C _{\varepsilon}\lnm \frac{\partial}{\partial \nu} N _{\delta} u \wedge \bar{\omega} _{n}\rnm _{k} ^{2}\\
    &\lesssim C _{\varepsilon}\(\lnm \vartheta \bar{\partial}N _{\delta} u \rnm _{k-1}^{2}+ \delta ^{2}\lnm \nabla N _{\delta} u \rnm _{k} ^{2} + \lnm  \bar{\partial}N _{\delta} u \rnm _{k-1}^{2}\)\\
    % &\leq  C _{\varepsilon}\(\lnm \bar{\partial} ^{*} \bar{\partial}N _{\delta} u \rnm _{k-1}^{2} + \lnm u \rnm _{k-1}^{2}\) + \delta ^{2} C _{\varepsilon}\lnm \nabla N _{\delta} u \rnm _{k} ^{2}\\
    &\leq  C _{\varepsilon, k} \lnm u \rnm _{k}^{2}+ \delta ^{2} C _{\varepsilon}\lnm \nabla N _{\delta} u \rnm _{k} ^{2}.
  \end{align*}
% where the second inequality follows from the estimate $\delta ^{2}\lnm \frac{\partial}{\partial \nu} N _{\delta} u \wedge \bar{\omega} _{n} \rnm _{k} ^{2} \leq \delta ^{2}\lnm \nabla N _{\delta} u \rnm _{k} ^{2}$, while the third is obtained via induction estimate (\ref{induction}) and (\ref{dbar dbar star2}). 
It thus follows that
  \begin{align}\label{dbar}
   \lnm v_{\alpha, \varepsilon} ^{k} \bar{\partial}N _{\delta} u \rnm  ^{2}
   \lesssim &\lnm D _{X _{\alpha, \varepsilon}} ^{k}\bar{\partial}N _{\delta} u \rnm  ^{2} + C _{\varepsilon, k} \lnm u \rnm _{k}^{2}+ \delta ^{2} C _{\varepsilon}\lnm \nabla N _{\delta} u \rnm _{k} ^{2}.
  \end{align}

   For the first terms in (\ref{nabla}), (\ref{dbar star}) and (\ref{dbar}), the argument unfolds as follows. %Since $\lnm  D _{X_{\alpha, \varepsilon}} ^{k} \nabla _{T} N _{\delta} u \rnm ^{2} \leq \lnm  D _{X_{\alpha, \varepsilon}} ^{k} \nabla N _{\delta} u \rnm ^{2}$, we are going to estimate as follows, 
  \begin{align}\label{999}
    &\delta \lnm  D _{X_{\alpha, \varepsilon}} ^{k} \nabla N _{\delta} u \rnm ^{2}\\\nonumber
    =& \delta\( \(D _{X_{\alpha, \varepsilon}} ^{\#}\)^{k} \nabla N _{\delta} u,  D _{X_{\alpha, \varepsilon}} ^{k} \nabla N _{\delta} u\) + \delta\( \(D _{X_{\alpha, \varepsilon}} ^{k} - \(D _{X_{\alpha, \varepsilon}} ^{\#}\)^{k}\) \nabla N _{\delta} u,  D _{X_{\alpha, \varepsilon}} ^{k} \nabla N _{\delta} u\)\\\nonumber
%    \leq &\delta\( \(D _{X_{\alpha, \varepsilon}} ^{\#}\)^{k} \nabla N _{\delta} u,  D _{X_{\alpha, \varepsilon}} ^{k} \nabla N _{\delta} u\) + \delta\lnm \(D _{X_{\alpha, \varepsilon}} ^{k} - \(D _{X_{\alpha, \varepsilon}} ^{\#}\)^{k}\) \nabla N _{\delta} u\rnm \lnm  D _{X_{\alpha, \varepsilon}} ^{k} \nabla N _{\delta} u\rnm\\\nonumber
    \leq &\delta\( \(D _{X_{\alpha, \varepsilon}} ^{\#}\)^{k} \nabla N _{\delta} u,  D _{X_{\alpha, \varepsilon}} ^{k} \nabla N _{\delta} u\) + 4 \delta \lnm \(D _{X_{\alpha, \varepsilon}} ^{k} - \(D _{X_{\alpha, \varepsilon}} ^{\#}\)^{k}\) \nabla N _{\delta} u\rnm ^{2} + \frac{\delta}{4}  \lnm  D _{X_{\alpha, \varepsilon}} ^{k} \nabla N _{\delta} u\rnm ^{2}\\\nonumber
    \leq &\delta\( \(D _{X_{\alpha, \varepsilon}} ^{\#}\)^{k} \nabla N _{\delta} u,  D _{X_{\alpha, \varepsilon}} ^{k} \nabla N _{\delta} u\) + 4 \delta \lnm \(D _{X_{\alpha, \varepsilon}} ^{k} - \(D _{X_{\alpha, \varepsilon}} ^{\#}\)^{k}\) \nabla _{T} N _{\delta} u\rnm ^{2} + 4 \delta C _{\varepsilon}\lnm \bar{L} _{n} N _{\delta} u\rnm _{k}^{2} 
    %\\\nonumber
   +\frac{\delta}{4} \lnm  D _{X_{\alpha, \varepsilon}} ^{k} \nabla N _{\delta} u\rnm ^{2},\nonumber
  \end{align}
  % where $D _{X_{\alpha, \varepsilon}} ^{\#}$ is the adjoint of $D _{X_{\alpha, \varepsilon}}$. 
where
  \begin{align*}
    \delta \lnm \bar{L} _{n} N _{\delta} u\rnm _{k}^{2}
    \leq \delta C_k \(\lnm \bar{\partial}  N _{\delta} u \rnm _{k}^{2}+ \lnm \bar{\partial} ^{*} N _{\delta} u \rnm _{k} ^{2}+ \lnm N _{\delta} u \rnm _{k}\) \lesssim \delta C_k \(\lnm \bar{\partial}  N _{\delta} u \rnm _{k}^{2}+ \lnm \bar{\partial} ^{*} N _{\delta} u \rnm _{k} ^{2} \),
  \end{align*}
  and
 % Letting $\delta$ small enough, these terms can be observed by the left-side hand.
%   % Applying \autoref{s.b.c comparable} and \autoref{adjoint of D} to the second term, we have
 %  For the second term in (\ref{999}), we have
  \begin{align*}
    &\delta\lnm \(D _{X_{\alpha, \varepsilon}} ^{k} - \(D _{X_{\alpha, \varepsilon}} ^{\#}\)^{k}\) \nabla _{T} N _{\delta} u\rnm ^{2}\\
    \lesssim &\delta\lnm \(D _{X_{\alpha, \varepsilon}} ^{k} - \(D _{X_{\alpha, \varepsilon}} ^{\#}\)^{k}\) \nabla^{(b)} _{T} N _{\delta} u\rnm ^{2} + \delta C _{\varepsilon}\lnm  N _{\delta} u\rnm _{k}^{2}\\
    \lesssim & 4k^{2} \varepsilon^{2} \delta\lnm v_{\alpha, \varepsilon} ^{k} \nabla^{(b)} _{T} N _{\delta} u \rnm ^{2}+ \delta C _{\varepsilon}\(\lnm \bar{\partial} \nabla^{(b)} _{T} N _{\delta} u \rnm _{k - 1} ^{2}+ \lnm \bar{\partial} ^{*} \nabla^{(b)} _{T} N _{\delta} u \rnm _{k - 1}^{2} + \lnm N _{\delta} u \rnm _{k}^{2}\) \\
        \leq& 4k^{2}\varepsilon^{2} \delta\lnm v_{\alpha, \varepsilon} ^{k} \nabla_{T} N _{\delta} u \rnm^{2} + \delta C _{\varepsilon}\(\lnm \bar{\partial}  N _{\delta} u \rnm _{k}^{2}+ \lnm \bar{\partial} ^{*} N _{\delta} u \rnm _{k} ^{2}+ \lnm N _{\delta} u \rnm^2 _{k}  \)\\
    \leq& 4k^{2}\varepsilon^{2} \delta\lnm v_{\alpha, \varepsilon} ^{k} \nabla_{T} N _{\delta} u \rnm^{2} + \delta C _{\varepsilon, k}\(\lnm \bar{\partial}  N _{\delta} u \rnm _{k}^{2}+ \lnm \bar{\partial} ^{*} N _{\delta} u \rnm _{k} ^{2}\) 
  \end{align*}
  % \begin{align*}
  %   &\delta\lnm \(D _{X_{\alpha, \varepsilon}} ^{k} - \(D _{X_{\alpha, \varepsilon}} ^{\#}\)^{k}\) \nabla _{T} N _{\delta} u\rnm ^{2}\\
  %   \leq& 4k \varepsilon^{2} \delta\lnm v_{\alpha, \varepsilon} ^{k} \nabla_{T} N _{\delta} u \rnm ^{2}+ \delta C _{\varepsilon}\(\lnm \bar{\partial} \nabla_{T} N _{\delta} u \rnm _{k - 1} ^{2}+ \lnm \bar{\partial} ^{*} \nabla_{T} N _{\delta} u \rnm _{k - 1}^{2} + \lnm \nabla _{T} N _{\delta} u \rnm _{k - 1}^{2} \)\\
  %   \leq& 4k\varepsilon^{2} \delta\lnm v_{\alpha, \varepsilon} ^{k} \nabla_{T} N _{\delta} u \rnm^{2} + \delta C _{\varepsilon}\(\lnm \bar{\partial}  N _{\delta} u \rnm _{k}^{2}+ \lnm \bar{\partial} ^{*} N _{\delta} u \rnm _{k} ^{2}+ \lnm N _{\delta} u \rnm _{k} +\lnm \nabla N _{\delta} u \rnm _{k - 1}^{2} \)\\
  %   \leq& 4k\varepsilon^{2} \delta\lnm v_{\alpha, \varepsilon} ^{k} \nabla_{T} N _{\delta} u \rnm^{2} + \delta C _{\varepsilon}\(\lnm \bar{\partial}  N _{\delta} u \rnm _{k}^{2}+ \lnm \bar{\partial} ^{*} N _{\delta} u \rnm _{k} ^{2}\) + \delta C _{\varepsilon} \lnm  u \rnm _{k - 1}^{2},
  % \end{align*}
 by \autoref{adjoint of D}. %since $T N _{\delta} u \in \dom\(\bar{\partial}^{*}\)$, and the forth inequality is obtained by using the Lie bracket identity, while the last inequality follows from (\ref{Nu}) and induction estimate (\ref{induction}). 
% Thus, the first term can be observed by $\delta\lnm v_{\alpha, \varepsilon} ^{k} \nabla N _{\delta} u \rnm^{2}$, and the second term can be observed by $\lnm \bar{\partial}  N _{\delta} u \rnm _{k}^{2}+ \lnm \bar{\partial} ^{*} N _{\delta} u \rnm _{k} ^{2}$ to the left side hand by letting $\delta \to 0$. For the first term in (\ref{999}),
Moreover, using 
  \begin{align*}
    [\nabla,  D _{X_{\alpha, \varepsilon}} ^{2k}]
     &= \nabla D _{X_{\alpha, \varepsilon}} ^{k}D _{X_{\alpha, \varepsilon}} ^{k} - D _{X_{\alpha, \varepsilon}} ^{k}\nabla D _{X_{\alpha, \varepsilon}} ^{k} + D _{X_{\alpha, \varepsilon}} ^{k}\nabla D _{X_{\alpha, \varepsilon}} ^{k} -  D _{X_{\alpha, \varepsilon}} ^{k} D _{X_{\alpha, \varepsilon}} ^{k} \nabla\\
     &= [\nabla, D _{X_{\alpha, \varepsilon}} ^{k} ]D _{X_{\alpha, \varepsilon}} ^{k} + D _{X_{\alpha, \varepsilon}} ^{k}[\nabla, D _{X_{\alpha, \varepsilon}} ^{k}] \\
     &= \left[[\nabla, D _{X_{\alpha, \varepsilon}} ^{k} ],D _{X_{\alpha, \varepsilon}} ^{k}\right] + 2D _{X_{\alpha, \varepsilon}} ^{k}[\nabla, D _{X_{\alpha, \varepsilon}} ^{k}]\\
     & = 2D _{X_{\alpha, \varepsilon}} ^{k}[\nabla, D _{X_{\alpha, \varepsilon}} ^{k}] + \sum_{j=1}^{k} \binom{k}{j} D _{X_{\alpha, \varepsilon}}^{k-j}\underbrace{[\dots[[\nabla,D _{X_{\alpha, \varepsilon}} ^{k}],D _{X_{\alpha, \varepsilon}}]\dots,D _{X_{\alpha, \varepsilon}}]}_{j+1\text{-fold}},
  \end{align*}
it follows that
    \begin{align}\label{nabla2}
    & \delta\( \(D _{X_{\alpha, \varepsilon}} ^{\#}\)^{k} \nabla N _{\delta} u,  D _{X_{\alpha, \varepsilon}} ^{k} \nabla N _{\delta} u\) \\\nonumber
    = &\delta \(\nabla N _{\delta} u,  D _{X_{\alpha, \varepsilon}} ^{2k} \nabla N _{\delta} u\) \\ \nonumber
    =  & \delta \(\nabla N _{\delta} u, \nabla  D _{X_{\alpha, \varepsilon}} ^{2k}  N _{\delta} u\) - \delta \(\nabla N _{\delta} u, [\nabla,  D _{X_{\alpha, \varepsilon}} ^{2k}]  N _{\delta} u\) \\\nonumber
    \leq  & \delta \(\nabla N _{\delta} u, \nabla  D _{X_{\alpha, \varepsilon}} ^{2k}  N _{\delta} u\) -  2\delta \(\nabla N _{\delta} u, D _{X_{\alpha, \varepsilon}} ^{k}[\nabla,  D _{X_{\alpha, \varepsilon}} ^{k}] N _{\delta} u\) + \delta C _{\varepsilon, k} \lnm \nabla N _{\delta} u \rnm _{k - 1}\lnm N _{\delta} u \rnm _{k}  \\\nonumber
    \leq & \delta \(\nabla N _{\delta} u, \nabla  D _{X_{\alpha, \varepsilon}} ^{2k}  N _{\delta} u\) + 2\delta \lnm \(D _{X_{\alpha, \varepsilon}} ^{\#}\)^{k}\nabla N _{\delta} u\rnm \lnm [\nabla,  D _{X_{\alpha, \varepsilon}} ^{k}] N _{\delta} u\rnm + \delta C _{\varepsilon, k} \lnm \nabla N _{\delta} u \rnm _{k - 1}\lnm N _{\delta} u \rnm _{k} \\\nonumber
     \leq & \delta \(\nabla N _{\delta} u, \nabla  D _{X_{\alpha, \varepsilon}} ^{2k}  N _{\delta} u\) + \text {s.c.} \delta C _{\varepsilon}  \lnm \nabla N _{\delta} u\rnm _{k} ^{2} +  \text{l.c.} \delta C _{\varepsilon, k} \(\lnm \bar{\partial}N _{\delta} u\rnm _{k} ^{2} +\lnm \bar{\partial} ^{*}N _{\delta} u\rnm _{k} ^{2}\),
  \end{align}
where the first equality follows from   \autoref{D sharp}    and the last inequality follows from 
  \begin{align*}
     &2\delta \lnm \(D _{X_{\alpha, \varepsilon}} ^{\#}\)^{k}\nabla N _{\delta} u\rnm \lnm [\nabla,  D _{X_{\alpha, \varepsilon}} ^{k}] N _{\delta} u\rnm + \delta C _{\varepsilon, k} \lnm \nabla N _{\delta} u \rnm _{k - 1}\lnm N _{\delta} u \rnm _{k}\\
 %    \leq &2\delta C _{\varepsilon} \lnm \nabla N _{\delta} u\rnm _{k} \lnm [\nabla,  D _{X_{\alpha, \varepsilon}} ^{k}] N _{\delta} u\rnm\\
  %   \lesssim &\delta C _{\varepsilon} s.c. \lnm \nabla N _{\delta} u\rnm _{k} ^{2} + \delta C _{\varepsilon} \text{l.c.} \lnm [\nabla,  D _{X_{\alpha, \varepsilon}} ^{k}] N _{\delta} u\rnm ^{2}\\
     \leq & \text{s.c.} \delta C _{\varepsilon, k}  \lnm \nabla N _{\delta} u\rnm _{k} ^{2} + \text{l.c.}\delta C _{\varepsilon, k} \lnm N _{\delta} u\rnm _{k} ^{2}\\
     \leq & \text{s.c.} \delta C _{\varepsilon, k} \lnm \nabla N _{\delta} u\rnm _{k} ^{2} + \text{l.c.} \delta C _{\varepsilon, k}  \(\lnm \bar{\partial}N _{\delta} u\rnm _{k} ^{2} +\lnm \bar{\partial} ^{*}N _{\delta} u\rnm _{k} ^{2}\).
  \end{align*}
 % Apply the above inequality to (\ref{nabla2}) and also apply the small-large constant estimate to the last term in (\ref{nabla2}), we have
%  \begin{align}\label{nabla3}
 %   &\delta\( \(D _{X_{\alpha, \varepsilon}} ^{\#}\)^{k} \nabla N _{\delta} u,  D _{X_{\alpha, \varepsilon}} ^{k} \nabla N _{\delta} u\)\\
 %   \lesssim &\delta \(\nabla N _{\delta} u, \nabla  D _{X_{\alpha, \varepsilon}} ^{2k}  N _{\delta} u\) + \delta C _{\varepsilon} s.c. \lnm \nabla N _{\delta} u\rnm _{k} ^{2} + \delta C _{\varepsilon, k} \text{l.c.} \(\lnm \bar{\partial}N _{\delta} u\rnm _{k} ^{2} +\lnm \bar{\partial} ^{*}N _{\delta} u\rnm _{k} ^{2}\) \nonumber
%  \end{align}
It follows from  (\ref{999}) and (\ref{nabla2}) that 
  \begin{align}\label{nabla4}
    &\delta \lnm  D _{X_{\alpha, \varepsilon}} ^{k} \nabla N _{\delta} u \rnm ^{2}\\\nonumber
    \lesssim& \delta \(\nabla N _{\delta} u, \nabla  D _{X_{\alpha, \varepsilon}} ^{2k}  N _{\delta} u\) +4k ^{2}\varepsilon^{2} \delta\lnm v_{\alpha, \varepsilon} ^{k} \nabla_{T} N _{\delta} u \rnm^{2}+ \delta C _{\varepsilon, k}\(\lnm \bar{\partial}N _{\delta} u\rnm _{k} ^{2} +\lnm \bar{\partial} ^{*}N _{\delta} u\rnm _{k} ^{2}\)  + \text{s.c.} \delta C _{\varepsilon}  \lnm \nabla N _{\delta} u\rnm _{k} ^{2} .\nonumber
  \end{align}
It follows from the similar argument that %to $\lnm D_{X_{\alpha, \varepsilon}} ^{k} \bar{\partial} ^{*} N _{\delta} u \rnm ^{2}$ yields,
   \begin{align*}
    \lnm D_{X_{\alpha, \varepsilon}} ^{k} \bar{\partial} ^{*} N _{\delta} u \rnm ^{2}
    \leq &\( \(D _{X_{\alpha, \varepsilon}} ^{\#}\)^{k} \bar{\partial} ^{*} N _{\delta} u,  D _{X_{\alpha, \varepsilon}} ^{k} \bar{\partial} ^{*} N _{\delta} u\) + 4 \lnm \(D _{X_{\alpha, \varepsilon}} ^{k} - \(D _{X_{\alpha, \varepsilon}} ^{\#}\)^{k}\) \bar{\partial} ^{*} N _{\delta} u\rnm ^{2} + \frac{1}{4} \lnm  D _{X_{\alpha, \varepsilon}} ^{k} \bar{\partial} ^{*} N _{\delta} u\rnm ^{2}\\
  \leq & \( \(D _{X_{\alpha, \varepsilon}} ^{\#}\)^{k} \bar{\partial} ^{*} N _{\delta} u,  D _{X_{\alpha, \varepsilon}} ^{k} \bar{\partial} ^{*} N _{\delta} u\) + 32k^{2} \varepsilon^{2} \lnm v_{\alpha, \varepsilon} ^{k} \bar{\partial} ^{*} N _{\delta} u \rnm ^{2}+ C _{\varepsilon, k}\lnm u \rnm _{k} ^{2}  + \frac{1}{4} \lnm  D _{X_{\alpha, \varepsilon}} ^{k} \bar{\partial} ^{*} N _{\delta} u\rnm ^{2},
  \end{align*}
  where the second inequality follows from  \autoref{adjoint of D} and (\ref{dbar dbar star2}).
%  \begin{align*}
 %   \lnm \(D _{X_{\alpha, \varepsilon}} ^{k} - \(D _{X_{\alpha, \varepsilon}} ^{\#}\)^{k}\) \bar{\partial} ^{*} N _{\delta} u\rnm ^{2}
 %   \leq 4k^{2} \varepsilon^{2} \lnm v_{\alpha, \varepsilon} ^{k} \bar{\partial} ^{*} N _{\delta} u \rnm ^{2}+ C _{\varepsilon}\(\lnm \bar{\partial} \bar{\partial} ^{*} N _{\delta} u \rnm _{k - 1} ^{2} + \lnm \bar{\partial} ^{*} N _{\delta} u \rnm _{k - 1}^{2} \).
 % \end{align*}
%  Applying (\ref{dbar dbar star2}) to the second term and induction eatimate (\ref{induction}) to the last term, we have
 %  \begin{align*}
 %   \lnm \(D _{X_{\alpha, \varepsilon}} ^{k} - \(D _{X_{\alpha, \varepsilon}} ^{\#}\)^{k}\) \bar{\partial} ^{*} N _{\delta} u\rnm ^{2}
%    \leq 4k^{2} \varepsilon^{2} \lnm v_{\alpha, \varepsilon} ^{k} \bar{\partial} ^{*} N _{\delta} u \rnm ^{2}+ C _{\varepsilon, k}\lnm u \rnm _{k} ^{2}.%
%  \end{align*}
Also, by the similar argument, we have
  \begin{align*}
    &\( \(D _{X_{\alpha, \varepsilon}} ^{\#}\)^{k} \bar{\partial} ^{*} N _{\delta} u,  D _{X_{\alpha, \varepsilon}} ^{k} \bar{\partial} ^{*} N _{\delta} u\)\\
    = &\(\bar{\partial} ^{*} N _{\delta} u,  D _{X_{\alpha, \varepsilon}} ^{2k} \bar{\partial} ^{*}N _{\delta} u\) \\ \nonumber
    =  & \(\bar{\partial} ^{*} N _{\delta} u, \bar{\partial} ^{*}  D _{X_{\alpha, \varepsilon}} ^{2k}  N _{\delta} u\) - \(\bar{\partial} ^{*} N _{\delta} u, [\bar{\partial} ^{*},  D _{X_{\alpha, \varepsilon}} ^{2k}]  N _{\delta} u\) \\\nonumber
    \leq  & \(\bar{\partial} ^{*} N _{\delta} u, \bar{\partial} ^{*}  D _{X_{\alpha, \varepsilon}} ^{2k}  N _{\delta} u\) -  2 \(\bar{\partial} ^{*} N _{\delta} u, D _{X_{\alpha, \varepsilon}} ^{k}[\bar{\partial} ^{*},  D _{X_{\alpha, \varepsilon}} ^{k}] N _{\delta} u\) + C _{\varepsilon, k} \lnm \bar{\partial} ^{*} N _{\delta} u \rnm _{k - 1}\lnm N _{\delta} u \rnm _{k}  \\\nonumber
%    \leq & \(\bar{\partial} ^{*} N _{\delta} u, \bar{\partial} ^{*}  D _{X_{\alpha, \varepsilon}} ^{2k}  N _{\delta} u\) + 2\lnm \(D _{X_{\alpha, \varepsilon}} ^{\#}\)^{k}\bar{\partial} ^{*} N _{\delta} u\rnm \lnm [\bar{\partial} ^{*},  D _{X_{\alpha, \varepsilon}} ^{k}] N _{\delta} u\rnm + C _{\varepsilon, k} \lnm \bar{\partial} ^{*} N _{\delta} u \rnm _{k - 1}\lnm N _{\delta} u \rnm _{k}\\
    \leq & \(\bar{\partial} ^{*} N _{\delta} u, \bar{\partial} ^{*}  D _{X_{\alpha, \varepsilon}} ^{2k}  N _{\delta} u\) + 2\lnm \(D _{X_{\alpha, \varepsilon}} ^{\#}\)^{k}\bar{\partial} ^{*} N _{\delta} u\rnm \lnm [\bar{\partial} ^{*},  D _{X_{\alpha, \varepsilon}} ^{k}] N _{\delta} u\rnm + \text{l.c.} C _{\varepsilon, k} \lnm \bar{\partial} ^{*} N _{\delta} u \rnm _{k - 1} ^{2} + \text{s.c.} C _{\varepsilon, k}\lnm N _{\delta} u \rnm _{k} ^{2}.
  \end{align*}
  For the second term, by the similar argument, we have
  \begin{align*}
    &\lnm \(D _{X_{\alpha, \varepsilon}} ^{\#}\)^{k}\bar{\partial} ^{*} N _{\delta} u\rnm \lnm [\bar{\partial} ^{*},  D _{X_{\alpha, \varepsilon}} ^{k}] N _{\delta} u\rnm\\
    \leq & \frac{1}{4} \lnm \(D _{X_{\alpha, \varepsilon}} ^{\#}\)^{k}\bar{\partial} ^{*} N _{\delta} u\rnm ^{2} + 4 \lnm [\bar{\partial} ^{*},  D _{X_{\alpha, \varepsilon}} ^{k}] N _{\delta} u\rnm ^{2}\\
     \leq &\frac{1}{2} \lnm D _{X_{\alpha, \varepsilon}}^{k}\bar{\partial} ^{*} N _{\delta} u\rnm ^{2} + \frac{1}{2} \lnm \(\(D _{X_{\alpha, \varepsilon}} ^{\#}\)^{k} - D _{X_{\alpha, \varepsilon}}^{k}\)\bar{\partial} ^{*} N _{\delta} u\rnm ^{2} + 4 \lnm [\bar{\partial} ^{*},  D _{X_{\alpha, \varepsilon}} ^{k}] N _{\delta} u\rnm ^{2}\\
     \leq & \frac{1}{2} \lnm D _{X_{\alpha, \varepsilon}}^{k}\bar{\partial} ^{*} N _{\delta} u\rnm ^{2} + 4k^{2}\varepsilon^{2} \lnm v_{\alpha, \varepsilon}^{k}\bar{\partial} ^{*} N _{\delta} u\rnm ^{2} + C _{\varepsilon} \(\lnm\bar{\partial} \bar{\partial} ^{*} N _{\delta} u\rnm _{k -1} ^{2} + \lnm \bar{\partial} ^{*} N _{\delta} u\rnm _{k - 1} ^{2}\) 
     + 4 \lnm [\bar{\partial} ^{*},  D _{X_{\alpha, \varepsilon}} ^{k}] N _{\delta} u\rnm ^{2}.
  \end{align*}
 % where the last inequality in obtained by \autoref{adjoint of D}. 
 Combining  the above inequalities and applying the induction, we have
  \begin{align}\label{dbar star4}
    \lnm D_{X_{\alpha, \varepsilon}} ^{k} \bar{\partial} ^{*} N _{\delta} u \rnm ^{2}  &\lesssim \(\bar{\partial} ^{*} N _{\delta} u, \bar{\partial} ^{*} D_{X_{\alpha, \varepsilon}} ^{2k}  N _{\delta} u\)  + \text{s.c.} C _{\varepsilon, k} \(\lnm\bar{\partial} N _{\delta} u\rnm _{k } ^{2} + \lnm \bar{\partial} ^{*} N _{\delta} u\rnm _{k } ^{2}\)  \\\nonumber
    &~~~ +  \lnm [\bar{\partial} ^{*},  D _{X_{\alpha, \varepsilon}} ^{k}] N _{\delta} u\rnm ^{2} +4k^{2} \varepsilon^{2} \lnm v_{\alpha, \varepsilon} ^{k} \bar{\partial} ^{*} N _{\delta} u \rnm ^{2} +  C_{\varepsilon, k}\lnm u \rnm _{k} ^{2} .\nonumber
  \end{align}
  % \textcolor{red}{
  % Combine (\ref{dbar star}) and (\ref{dbar star4}), we have
  % \begin{align}\label{dbar star2}
  %   \lnm v_{\alpha, \varepsilon} ^{k} \bar{\partial} ^{*} N _{\delta} u \rnm ^{2}
  %   \lesssim \(\bar{\partial} ^{*} N _{\delta} u, \bar{\partial} ^{*} D_{X_{\alpha, \varepsilon}} ^{2k}  N _{\delta} u\) + l.c. \lnm [\bar{\partial} ^{*},  D _{X_{\alpha, \varepsilon}} ^{k}] N _{\delta} u\rnm ^{2}  +C_{\varepsilon, k}\lnm u \rnm _{k} ^{2} + s.c. C _{\varepsilon, k}\lnm N _{\delta} u \rnm _{k} ^{2},
  % \end{align}}
  %where the last term can be observed to the left-hand side by choosing small constant small enough. 
%  For the term $\lnm D _{X_{\alpha, \varepsilon}} ^{k}\bar{\partial} N _{\delta} u \rnm ^{2}$, we apply the same argument,
  Similarly, we have 
  \begin{align*}
  &~~~~~~ \lnm D_{X_{\alpha, \varepsilon}} ^{k} \bar{\partial} N _{\delta} u \rnm ^{2} \\
  &  \leq \( \(D _{X_{\alpha, \varepsilon}} ^{\#}\)^{k} \bar{\partial} N _{\delta} u,  D _{X_{\alpha, \varepsilon}} ^{k} \bar{\partial} N _{\delta} u\) + 4 \lnm \(D _{X_{\alpha, \varepsilon}} ^{k} - \(D _{X_{\alpha, \varepsilon}} ^{\#}\)^{k}\) \bar{\partial} N _{\delta} u\rnm ^{2} + \frac{1}{4} \lnm  D _{X_{\alpha, \varepsilon}} ^{k} \bar{\partial} N _{\delta} u\rnm ^{2} \\
& \leq \( \(D _{X_{\alpha, \varepsilon}} ^{\#}\)^{k} \bar{\partial} N _{\delta} u,  D _{X_{\alpha, \varepsilon}} ^{k} \bar{\partial} N _{\delta} u\) + 8 \lnm \(D _{X_{\alpha, \varepsilon}} ^{k} - \(D _{X_{\alpha, \varepsilon}} ^{\#}\) ^{k}\) \(\bar{\partial} N _{\delta} u + \delta \frac{\partial}{\partial \nu} N _{\delta} u \wedge \bar{\omega} _{n}\)\rnm ^{2}  \\
&~~~ + 8 \lnm \(D _{X_{\alpha, \varepsilon}} ^{k} - \(D _{X_{\alpha, \varepsilon}}^{\#}\) ^{k}\) \(\delta \frac{\partial}{\partial \nu} N _{\delta} u \wedge \bar{\omega} _{n}\)\rnm ^{2} +  \frac{1}{4} \lnm  D _{X_{\alpha, \varepsilon}} ^{k} \bar{\partial} N _{\delta} u\rnm ^{2} \\
&\leq  \( \(D _{X_{\alpha, \varepsilon}} ^{\#}\)^{k} \bar{\partial} N _{\delta} u,  D _{X_{\alpha, \varepsilon}} ^{k} \bar{\partial} N _{\delta} u\)  + 120 k^{2} \varepsilon^{2} \lnm v_{\alpha, \varepsilon} ^{k} \bar{\partial} N _{\delta} u \rnm ^{2}+ C_{\varepsilon, k} \lnm u \rnm _{k} ^{2} + \delta ^{2}C _{\varepsilon}  \lnm \nabla N _{\delta} u \rnm _{k}^{2}+ \frac{1}{4} \lnm  D _{X_{\alpha, \varepsilon}} ^{k} \bar{\partial} N _{\delta} u\rnm ^{2},
  \end{align*}
where the last inequality follows from  \autoref{adjoint of D}, (\ref{dbar dbar star2}) and the induction. 
It follows from the similar argument that
  \begin{align*}
    &\( \(D _{X_{\alpha, \varepsilon}} ^{\#}\)^{k} \bar{\partial} N _{\delta} u,  D _{X_{\alpha, \varepsilon}} ^{k} \bar{\partial} N _{\delta} u\)\\
%    = &\(\bar{\partial} ^{*} N _{\delta} u,  D _{X_{\alpha, \varepsilon}} ^{2k} \bar{\partial} ^{*}N _{\delta} u\) \\ \nonumber
  %  =  & \(\bar{\partial} N _{\delta} u, \bar{\partial}  D _{X_{\alpha, \varepsilon}} ^{2k}  N _{\delta} u\) + \(\bar{\partial} N _{\delta} u, [\bar{\partial},  D _{X_{\alpha, \varepsilon}} ^{2k}]  N _{\delta} u\) \\\nonumber
 %   \leq  & \(\bar{\partial} N _{\delta} u, \bar{\partial}  D _{X_{\alpha, \varepsilon}} ^{2k}  N _{\delta} u\) +  2 \(\bar{\partial} N _{\delta} u, D _{X_{\alpha, \varepsilon}} ^{k}[\bar{\partial},  D _{X_{\alpha, \varepsilon}} ^{k}] N _{\delta} u\) + C _{\varepsilon, k} \lnm \bar{\partial} N _{\delta} u \rnm _{k - 1}\lnm N _{\delta} u \rnm _{k}  \\\nonumber
 %   \leq & \(\bar{\partial} N _{\delta} u, \bar{\partial}  D _{X_{\alpha, \varepsilon}} ^{2k}  N _{\delta} u\) + 2\lnm \(D _{X_{\alpha, \varepsilon}} ^{\#}\)^{k}\bar{\partial} N _{\delta} u\rnm \lnm [\bar{\partial},  D _{X_{\alpha, \varepsilon}} ^{k}] N _{\delta} u\rnm + C _{\varepsilon, k} \lnm \bar{\partial} N _{\delta} u \rnm _{k - 1}\lnm N _{\delta} u \ram _{k}\\
    \leq & \(\bar{\partial} N _{\delta} u, \bar{\partial}  D _{X_{\alpha, \varepsilon}} ^{2k}  N _{\delta} u\) + 2\lnm \(D _{X_{\alpha, \varepsilon}} ^{\#}\)^{k}\bar{\partial} N _{\delta} u\rnm \lnm [\bar{\partial},  D _{X_{\alpha, \varepsilon}} ^{k}] N _{\delta} u\rnm + \text{l.c.} C _{\varepsilon, k} \lnm \bar{\partial} N _{\delta} u \rnm _{k - 1} ^{2} + \text{s.c.} C _{\varepsilon, k}\lnm N _{\delta} u \rnm _{k} ^{2}\\
   \leq & \(\bar{\partial} N _{\delta} u, \bar{\partial}  D _{X_{\alpha, \varepsilon}} ^{2k}  N _{\delta} u\) + \frac{1}{3} \lnm D _{X_{\alpha, \varepsilon}}^{k}\bar{\partial} N _{\delta} u\rnm ^{2} + 10 \lnm \(\(D _{X_{\alpha, \varepsilon}} ^{\#}\)^{k} - D _{X_{\alpha, \varepsilon}}^{k} \)\bar{\partial} N _{\delta} u\rnm ^{2} + 32 \lnm [\bar{\partial},  D _{X_{\alpha, \varepsilon}} ^{k}] N _{\delta} u\rnm ^{2} \\
  & +C _{\varepsilon, k} \lnm  u \rnm _{k} ^{2} + \text{s.c.} C _{\varepsilon, k}\lnm N _{\delta} u \rnm _{k} ^{2}\\
    \leq & \(\bar{\partial} N _{\delta} u, \bar{\partial}  D _{X_{\alpha, \varepsilon}} ^{2k}  N _{\delta} u\) + \frac{1}{3} \lnm D _{X_{\alpha, \varepsilon}}^{k}\bar{\partial} N _{\delta} u\rnm ^{2} + 100k^{2}\varepsilon^{2} \lnm v_{\alpha, \varepsilon}^{k}\bar{\partial} N _{\delta} u\rnm ^{2} + \delta ^{2} C _{\varepsilon}\lnm \nabla N _{\delta} u\rnm_k ^{2} +C _{\varepsilon, k} \lnm  u \rnm _{k} ^{2}\\    
   &  + 32 \lnm [\bar{\partial},  D _{X_{\alpha, \varepsilon}} ^{k}] N _{\delta} u\rnm ^{2}  + \text{s.c.} C _{\varepsilon, k} \(\lnm\bar{\partial} N _{\delta} u\rnm _{k } ^{2} + \lnm \bar{\partial} ^{*} N _{\delta} u\rnm _{k } ^{2}\)   ,
  \end{align*}
 % The second term, we apply the small-large constant estimate,
 % \begin{align*}
  %  &\lnm \(D _{X_{\alpha, \varepsilon}} ^{\#}\)^{k}\bar{\partial} N _{\delta} u\rnm \lnm [\bar{\partial},  D _{X_{\alpha, \varepsilon}} ^{k}] N _{\delta} u\rnm\\
  %  \leq &\text{s.c.} \lnm \(D _{X_{\alpha, \varepsilon}} ^{\#}\)^{k}\bar{\partial} N _{\delta} u\rnm ^{2} + \text{l.c.} \lnm [\bar{\partial},  D _{X_{\alpha, \varepsilon}} ^{k}] N _{\delta} u\rnm ^{2}\\
   %  \leq &\text{s.c.} \lnm D _{X_{\alpha, \varepsilon}}^{k}\bar{\partial} N _{\delta} u\rnm ^{2} + \text{s.c.} \lnm \(\(D _{X_{\alpha, \varepsilon}} ^{\#}\)^{k} - D _{X_{\alpha, \varepsilon}}^{k} \)\bar{\partial} N _{\delta} u\rnm ^{2} + \text{l.c.} \lnm [\bar{\partial},  D _{X_{\alpha, \varepsilon}} ^{k}] N _{\delta} u\rnm ^{2}.
  %  \end{align*}
%   Since $\bar{\partial} N _{\delta} u + \delta \frac{\partial}{\partial \nu} N _{\delta} u \wedge \bar{\omega} _{n} \in \dom \(\bar{\partial} ^{*}\)$, we using \autoref{adjoint of D} to the second term,
   where the last inequality follows from 
    \begin{align*}
    &\lnm \(\(D _{X_{\alpha, \varepsilon}} ^{\#}\)^{k} - D _{X_{\alpha, \varepsilon}}^{k} \)\bar{\partial} N _{\delta} u\rnm ^{2}\\
     \leq &\frac{6}{5} \lnm \(\(D _{X_{\alpha, \varepsilon}} ^{\#}\)^{k} - D _{X_{\alpha, \varepsilon}}^{k} \)\(\bar{\partial} N _{\delta} u + \delta \frac{\partial}{\partial \nu} N _{\delta} u \wedge \bar{\omega} _{n}\)\rnm ^{2} +10 \lnm \(\(D _{X_{\alpha, \varepsilon}} ^{\#}\)^{k} - D _{X_{\alpha, \varepsilon}}^{k} \)\(\delta \frac{\partial}{\partial \nu} N _{\delta} u \wedge \bar{\omega} _{n}\)\rnm ^{2}\\
%     \leq &\lnm \(\(D _{X_{\alpha, \varepsilon}} ^{\#}\)^{k} - D _{X_{\alpha, \varepsilon}}^{k} \)\(\bar{\partial} N _{\delta} u + \delta \frac{\partial}{\partial \nu} N _{\delta} u \wedge \bar{\omega} _{n}\)\rnm ^{2} +\delta ^{2} C _{\varepsilon}\lnm \nabla N _{\delta} u\rnm _{k}^{2} \\
     \leq & 10k^{2}\varepsilon^{2} \lnm v_{\alpha, \varepsilon}^{k}\bar{\partial} N _{\delta} u\rnm ^{2} + C _{\varepsilon, k} \lnm u \rnm_k^{2}  + \delta ^{2} C _{\varepsilon}\lnm \nabla N _{\delta} u\rnm_k ^{2}.
  \end{align*}
  Thus, combining the above inequalities, %and applying (\ref{dbar dbar star2}) and induction estimate to the lower-order term, 
  we have
  \begin{align}\label{dbar4}
    \lnm D_{X_{\alpha, \varepsilon}} ^{k} \bar{\partial} N _{\delta} u \rnm ^{2}
    \lesssim &\(\bar{\partial} N _{\delta} u, \bar{\partial} D_{X_{\alpha, \varepsilon}} ^{2k}  N _{\delta} u\) +  \lnm [\bar{\partial},  D _{X_{\alpha, \varepsilon}} ^{k}] N _{\delta} u\rnm ^{2} + k^{2} \varepsilon^{2} \lnm v_{\alpha, \varepsilon} ^{k} \bar{\partial} N _{\delta} u \rnm ^{2} +C_{\varepsilon, k}\lnm u \rnm _{k} ^{2} \\
    &+ \text{s.c.} C _{\varepsilon, k}\(\lnm\bar{\partial} N _{\delta} u\rnm _{k } ^{2} + \lnm \bar{\partial} ^{*} N _{\delta} u\rnm _{k } ^{2}\)  + \delta ^{2} C _{\varepsilon}\lnm \nabla N _{\delta} u\rnm ^{2}_k.\nonumber
  \end{align}
 Since the first terms in (\ref{nabla4}), (\ref{dbar star4}), and (\ref{dbar4}) yield
  \begin{align}\label{Q delta}
    &\(\bar{\partial} N _{\delta} u, \bar{\partial} D _{X_{\alpha, \varepsilon}} ^{2k}  N _{\delta} u\) + \(\bar{\partial} ^{*} N _{\delta} u, \bar{\partial} ^{*} D _{X_{\alpha, \varepsilon}} ^{2k}  N _{\delta} u\) + \delta \(\nabla N _{\delta} u,  \nabla D _{X_{\alpha, \varepsilon}}^{2k} N _{\delta} u\)\\\nonumber
    =& Q _{\delta} \(N _{\delta} u, D _{X_{\alpha, \varepsilon}}^{2k}  N _{\delta} u\)\nonumber
    = \(\Box _{\delta}N _{\delta} u,D _{X_{\alpha, \varepsilon}} ^{2k}  N _{\delta} u\)\nonumber
    =\(u,D _{X_{\alpha, \varepsilon}} ^{2k}  N _{\delta} u\)  \\\nonumber
    \leq & \(\(D _{X_{\alpha, \varepsilon}}^{\#}\)^{k} u,D _{X_{\alpha, \varepsilon}} ^{k}  N _{\delta} u\) \leq C _{\varepsilon} \lnm u \rnm _{k}\lnm N _{\delta}u \rnm _{k}  \\\nonumber
    \leq &      \text{s.c.}  C _{\varepsilon} \lnm N _{\delta}u \rnm _{k} ^{2} + \text{l.c.} C _{\varepsilon} \lnm u \rnm _{k} ^{2} \leq   \text{s.c.}  C _{\varepsilon} \(\lnm\bar{\partial}  N _{\delta} u\rnm _{k } ^{2} + \lnm \bar{\partial} ^{*} N _{\delta} u\rnm _{k } ^{2}\) + \text{l.c.} C _{\varepsilon} \lnm u \rnm _{k} ^{2}, \nonumber
  \end{align}
it suffices  to estimate $\lnm[\bar{\partial} ^{*}, D _{X_{\alpha, \varepsilon}} ^{k}]N _{\delta} u\rnm ^{2}$ and $\lnm[\bar{\partial}, D _{X_{\alpha, \varepsilon}}^{k}]   N _{\delta} u\rnm ^{2}$. 
  % For the first of these $\delta\lnm \left[\nabla, D _{X_{\alpha, \varepsilon}} ^{k}\right]N _{\delta} u\rnm ^{2}$, we find that
  % \begin{align}\label{delta 3}
  %   % \delta\lnm [\nabla, D _{X}^{k}] N _{\delta} u\rnm ^{2}
  %   \delta\lnm \left[\nabla, D _{X_{\alpha, \varepsilon}} ^{k}\right]N _{\delta} u\rnm ^{2}
  %   \leq \delta C _{\varepsilon}\lnm N _{\delta} u\rnm _{k} ^{2}
  %   \leq \delta C _{\varepsilon}\(\lnm \bar{\partial} N _{\delta} u\rnm _{k} ^{2} + \lnm \bar{\partial}^{*} N _{\delta} u\rnm _{k} ^{2}\),
  % \end{align}
  % we apply (\ref{Nu}) to the second inequlaity, thus, these terms can be observed by the left-side hand by letting $\delta \to 0$. 
 By \autoref{lie bracket basic estimate} and the induction, we have
  \begin{align}\label{dbar star3}
    % &\lnm[\bar{\partial} ^{*}, D _{X}^{k}] N _{\delta} u\rnm ^{2}\\
     &\lnm\left[\bar{\partial} ^{*}, D _{X_{\alpha, \varepsilon}} ^{k}\right] N _{\delta} u\rnm ^{2}\\\nonumber
    \leq &\varepsilon^2 C_k \(\lnm \bar{\partial} D _{X_{\alpha, \varepsilon}} ^{k} N _{\delta}u \rnm ^2 + \lnm \bar{\partial}^* D _{X_{\alpha, \varepsilon}} ^{k}N _{\delta}u \rnm ^2 \) + C _{\varepsilon, k} \(\lnm \bar{\partial}  N _{\delta}u \rnm  _{k - 1}^2 + \lnm \bar{\partial}^* N _{\delta}u \rnm _ {k - 1}^2 +\lnm N _{\delta}u \rnm  _{k - 1}^2 \)\\\nonumber
 %   \leq & \varepsilon k^{2} \(\lnm  D _{X_{\alpha, \varepsilon}} ^{k} \bar{\partial} N _{\delta}u \rnm ^2 + \lnm  D _{X_{\alpha, \varepsilon}} ^{k}\bar{\partial}^*N _{\delta}u \rnm ^2  + \lnm \left[\bar{\partial}, D _{X_{\alpha, \varepsilon}} ^{k}\right] N _{\delta}u \rnm ^2 + \lnm \left[\bar{\partial}^* ,D _{X_{\alpha, \varepsilon}} ^{k}\right]N _{\delta}u \rnm ^2\) \\
 %   &+ C _{\varepsilon, k} \(\lnm \bar{\partial}  N _{\delta}u \rnm  _{k - 1}^2 + \lnm \bar{\partial}^* N _{\delta}u \rnm _ {k - 1}^2 +\lnm N _{\delta}u \rnm  _{k - 1}^2\)\\
     \leq & \varepsilon^2 C_k \(\lnm  D _{X_{\alpha, \varepsilon}} ^{k} \bar{\partial} N _{\delta}u \rnm ^2 + \lnm  D _{X_{\alpha, \varepsilon}} ^{k}\bar{\partial}^*N _{\delta}u \rnm ^2  + \lnm \left[\bar{\partial}, D _{X_{\alpha, \varepsilon}} ^{k}\right] N _{\delta}u \rnm ^2 + \lnm \left[\bar{\partial}^* ,D _{X_{\alpha, \varepsilon}} ^{k}\right]N _{\delta}u \rnm ^2\) + C _{\varepsilon, k} \lnm u \rnm _{k - 1}^2.
  \end{align}
%  we apply the induction estimate (\ref{induction}) to the $k - 1$-order Sobolev norms. Thus,
%
%  \begin{align}
 %   &\lnm\left[\bar{\partial} ^{*}, D _{X_{\alpha, \varepsilon}}^{k}\right]  N _{\delta} u\rnm ^{2}\\\nonumber
    % \lnm[\bar{\partial} ^{*}, D _{X}] D _{X} ^{k - 1}  N _{\delta} u\rnm ^{2}
 %   \lesssim & k^{2}\varepsilon \(\lnm D _{X_{\alpha, \varepsilon}} ^{k} \bar{\partial} N _{\delta}u \rnm ^2 + \lnm  D _{X_{\alpha, \varepsilon}} ^{k}\bar{\partial}^*N _{\delta}u \rnm ^2 + \lnm \left[\bar{\partial}, D _{X_{\alpha, \varepsilon}} ^{k}\right] N _{\delta}u \rnm ^2 \)
    % + \lnm \left[\bar{\partial}^* ,D _{X} ^{k}\right]N _{\delta}u \rnm ^2\) 
%    + C _{\varepsilon,k} \lnm u \rnm _{k - 1} ^2.\nonumber
 % \end{align}
%  For the term $\lnm\left[\bar{\partial}, D _{X_{\alpha, \varepsilon}} ^{k}\right]  N _{\delta} u\rnm ^{2}$, 
On the other hand, it follows from (\ref{powers of lie bracket}) that
  \begin{align}\label{555}
    \lnm\left[\bar{\partial}, D _{X_{\alpha, \varepsilon}} ^{k}\right]  N _{\delta} u\rnm ^{2}
    \leq 4k^2 \lnm\left[\bar{\partial}, D _{X_{\alpha, \varepsilon}}\right]D _{X_{\alpha, \varepsilon}} ^{k - 1}  N _{\delta} u\rnm ^{2} + C _{\varepsilon, k}\lnm N _{\delta} u\rnm _{k - 1} ^{2},
  \end{align}
with
  \begin{align}\label{wt7}
    \left[\bar{\partial}, D _{X_{\alpha, \varepsilon}} \right] 
  %  &= \left[\bar{\partial}, D _{v_{\alpha, \varepsilon}^{\#}}\right] - \left[\bar{\partial}, D _{\bar{v}_{\alpha, \varepsilon}}\right] 
%    = \left[\bar{\partial},v_{\alpha, \varepsilon}^{\#}\right] + \left[\bar{\partial}, D _{v_{\alpha, \varepsilon}^{\#}} - v_{\alpha, \varepsilon}^{\#}\right] - \left[\bar{\partial}, D _{\bar{v}_{\alpha, \varepsilon}}\right] \\
    = \left[\bar{\partial},v_{\alpha, \varepsilon}\right] +  \left[\bar{\partial},v_{\alpha, \varepsilon}^{\#} - v_{\alpha, \varepsilon}\right] + \left[\bar{\partial}, D _{v_{\alpha, \varepsilon}^{\#}} - v_{\alpha, \varepsilon}^{\#}\right] -  \left[\bar{\partial}, D _{\bar{v}_{\alpha, \varepsilon}}\right].
  \end{align}
  By the definition,  %of $D _{Z}$ (\ref{DZ}), 
  the terms $[\bar{\partial}, D _{v_{\alpha, \varepsilon}^{\#}} - v_{\alpha, \varepsilon}^{\#}]$ and $[\bar{\partial}, D _{\bar{v}_{\alpha, \varepsilon}}]$ are first order differential operators of $\(0, 1\)$-type, %plus $0$-th order operators, 
  and it thus follows from the benign estimate and the induction that 
  \begin{align}\label{wt5}
    &\lnm\left[\bar{\partial}, D _{v_{\alpha, \varepsilon} ^{\#}} - v_{\alpha, \varepsilon} ^{\#}\right] D _{X_{\alpha, \varepsilon}} ^{k - 1}  N _{\delta} u\rnm ^{2} + \lnm\left[\bar{\partial}, D _{\bar{v}_{\alpha, \varepsilon}}\right] D _{X_{\alpha, \varepsilon}} ^{k - 1} N _{\delta} u\rnm ^{2} \\\nonumber
    \lesssim&  C_{\varepsilon} \(\lnm \bar{\partial} D _{X_{\alpha, \varepsilon}} ^{k - 1}  N _{\delta} u\rnm ^{2} + \lnm \bar{\partial}^{*} D _{X_{\alpha, \varepsilon}} ^{k - 1}  N _{\delta} u\rnm ^{2} + \lnm  D _{X_{\alpha, \varepsilon}} ^{k - 1} N _{\delta} u\rnm ^{2}\)\\\nonumber
    % \lesssim&  C_{\varepsilon}\(\lnm \bar{\partial} D _{X_{\alpha, \varepsilon}} ^{k - 1}  N _{\delta} u\rnm ^{2} + \lnm \bar{\partial}^{*} D _{X_{\alpha, \varepsilon}} ^{k - 1}  N _{\delta} u\rnm ^{2} + \lnm  N _{\delta} u\rnm _{k - 1} ^{2}\)\\
    \leq& C _{\varepsilon}\(\lnm D _{X_{\alpha, \varepsilon}} ^{k - 1} \bar{\partial}  N _{\delta} u\rnm ^{2} + \lnm D _{X_{\alpha, \varepsilon}} ^{k - 1}  \bar{\partial}^{*} N _{\delta} u\rnm ^{2} +\lnm \left[\bar{\partial}, D _{X_{\alpha, \varepsilon}} ^{k - 1} \right]  N _{\delta} u\rnm ^{2} + \lnm \left[\bar{\partial}^{*}, D _{X_{\alpha, \varepsilon}} ^{k - 1} \right]  N _{\delta} u\rnm ^{2}+ \lnm  N _{\delta} u\rnm _{k - 1} ^{2}\)\\\nonumber
    \leq & C _{\varepsilon}\(\lnm \bar{\partial} N _{\delta} u\rnm _{k - 1} ^{2} + \lnm \bar{\partial} ^{* }N _{\delta} u\rnm _{k - 1} ^{2} + \lnm  N _{\delta} u\rnm _{k - 1} ^{2}\) \\\nonumber
    \leq & C _{\varepsilon, k} \lnm u\rnm _{k - 1} ^{2}.
  \end{align}
%  Using (\ref{Nu}) and the induction estimate (\ref{induction}), we have
 %  \begin{align*}
 %  \lnm\left[\bar{\partial}, D _{v_{\alpha, \varepsilon} ^{\#}} - v_{\alpha, \varepsilon} ^{\#}\right] D _{X_{\alpha, \varepsilon}} ^{k - 1}  N _{\delta} u\rnm ^{2} + \lnm\left[\bar{\partial}, D _{\bar{v}_{\alpha, \varepsilon}}\right] D _{X_{\alpha, \varepsilon}} ^{k - 1} N _{\delta} u\rnm ^{2}
%  \end{align*}
  % By the condition (i) of definition of $\theta ^{v}$, we have $\left|[\bar{\partial},v^{\#} - v]u\right| \leq 2\varepsilon \left|v u\right|$. And by the definition of $D _{v^{\#}}$, we have $\left|[\bar{\partial}, D _{v^{\#}} - v^{\#}] u\right| \leq C \left|u\right|$. The last term is a $\(0, 1\)$-type, thus
  % \begin{align*}
  %   \lnm[\bar{\partial}, D _{X}] D _{X} ^{k - 1}  N _{\delta} u\rnm ^{2}
  %   &\leq \lnm \left[\bar{\partial}, v\right] D_{T} ^{k-1} N _{\delta} u \rnm ^{2} + 2\varepsilon \lnm N _{\delta} u \rnm_{k} + C\lnm N _{\delta} u\rnm _{k - 1} \\
  %   &+ C\(\lnm \bar{\partial} N _{\delta} u \rnm _{k - 1} ^{2} + \lnm \bar{\partial} ^{*} N _{\delta} u \rnm _{k - 1} ^{2} + \lnm N _{\delta} u \rnm _{k - 1} ^{2}\).
  % \end{align*}
%  For the second term $ \lnm \left[\bar{\partial},v_{\alpha, \varepsilon}^{\#} - v_{\alpha, \varepsilon}\right]D _{X_{\alpha, \varepsilon}} ^{k - 1}  N _{\delta} u\rnm ^{2}$, i
  It follows from %(\ref{AAA}), (\ref{BBB}) and 
  the condition (i) in \autoref{property p} that
  \begin{align}\label{wt8}
    &\lnm \left[\bar{\partial},v_{\alpha, \varepsilon}^{\#} - v_{\alpha, \varepsilon}\right]D _{X_{\alpha, \varepsilon}} ^{k - 1}  N _{\delta} u\rnm ^{2}\\\nonumber
    % = &\lnm \left[\bar{\partial},2 \arg \(d \rho \(v_{\alpha, \varepsilon}\) \)v_{\alpha, \varepsilon}\right]D _{X_{\alpha, \varepsilon}} ^{k - 1}  N _{\delta} u\rnm ^{2}\\
    = &\lnm \left[\bar{\partial},\(1 - e ^{-2i\text{arg}  d \rho \(v_{\alpha, \varepsilon}\)}\)v_{\alpha, \varepsilon}\right]D _{X_{\alpha, \varepsilon}} ^{k - 1}  N _{\delta} u\rnm ^{2}\\\nonumber
    \leq& 2 \lnm \(1 - e ^{-2i\text{arg}  d \rho \(v_{\alpha, \varepsilon}\)}\)\left[\bar{\partial}, v_{\alpha, \varepsilon}\right]D _{X_{\alpha, \varepsilon}} ^{k - 1}  N _{\delta} u\rnm ^{2} + 2 \lnm \bar{\partial}\(1 - e ^{-2i\text{arg}  d \rho \(v_{\alpha, \varepsilon}\)}\)  \wedge v_{\alpha, \varepsilon}D _{X_{\alpha, \varepsilon}} ^{k - 1}  N _{\delta} u\rnm ^{2}\\\nonumber
 %   \leq& \lnm 2\left|\text{arg}  d \rho \(v_{\alpha, \varepsilon}\)\right| \left[\bar{\partial}, v_{\alpha, \varepsilon}\right]D _{X_{\alpha, \varepsilon}} ^{k - 1}  N _{\delta} u\rnm ^{2} + \lnm 2\left|\bar{\partial}\(\arg \(d \rho \(v_{\alpha, \varepsilon}\) \)\)\right|v_{\alpha, \varepsilon}D _{X_{\alpha, \varepsilon}} ^{k - 1}  N _{\delta} u\rnm ^{2}\\
    % \leq & 4\varepsilon^{2} \lnm \left[\bar{\partial}, v_{\alpha, \varepsilon}\right]D _{X_{\alpha, \varepsilon}} ^{k - 1}  N _{\delta} u\rnm ^{2} + \lnm \bar{\partial}\(2 \arg \(d \rho \(v_{\alpha, \varepsilon}\) \)\)v_{\alpha, \varepsilon}D _{X_{\alpha, \varepsilon}} ^{k - 1}  N _{\delta} u\rnm ^{2}\\
     \lesssim & \varepsilon^{2} \lnm \left[\bar{\partial}, v_{\alpha, \varepsilon}\right]D _{X_{\alpha, \varepsilon}} ^{k - 1}  N _{\delta} u\rnm ^{2} + \varepsilon^{2}\lnm v_{\alpha, \varepsilon}D _{X_{\alpha, \varepsilon}} ^{k - 1}  N _{\delta} u\rnm ^{2}.
  \end{align} 
%  For the secon term $\lnm v_{\alpha, \varepsilon}D _{X_{\alpha, \varepsilon}} ^{k - 1}  N _{\delta} u\rnm ^{2}$, we have
It follows from  (\ref{duan123}),   the standard Morrey-Kohn-H\"{o}rmander estimate,  the benign estimates and the induction in  (\ref{induction}) that 
%  and thus 
 % where the second inequality follows from (\ref{BBB}), while the third is obtained via (\ref{DZ-Z}). Then applying the Morrey-Kohn-H\"{o}rmander estimate to the first term, and the benign estimate to the second term and (\ref{induction}) to the lower-order terms, we have
  \begin{align}\label{666}
    &\lnm v_{\alpha, \varepsilon}D _{X_{\alpha, \varepsilon}} ^{k - 1}  N _{\delta} u\rnm ^{2}\\\nonumber
    \lesssim & \lnm D _{X_{\alpha, \varepsilon}} ^{k}  N _{\delta} u\rnm ^{2} + \lnm D _{\bar{v}_{\alpha, \varepsilon}} D _{X_{\alpha, \varepsilon}} ^{k - 1}  N _{\delta} u\rnm ^{2}+ C _{\varepsilon}\lnm  N _{\delta}  u\rnm _{k - 1} ^{2} \\\nonumber
    \lesssim &\lnm \bar{\partial} D _{X_{\alpha, \varepsilon}} ^{k}  N _{\delta} u\rnm ^{2} + \lnm \bar{\partial}^{*} D _{X_{\alpha, \varepsilon}} ^{k}  N _{\delta} u\rnm ^{2} + \lnm \bar{\partial} D _{X_{\alpha, \varepsilon}} ^{k - 1}  N _{\delta} u\rnm ^{2} + \lnm \bar{\partial} ^{*}D _{X_{\alpha, \varepsilon}} ^{k - 1}  N _{\delta} u\rnm ^{2} +C _{\varepsilon}\lnm N _{\delta}  u\rnm _{k - 1} ^{2}\\\nonumber
%    \leq & \lnm  D _{X_{\alpha, \varepsilon}} ^{k} \bar{\partial} N _{\delta} u\rnm ^{2} + \lnm D _{X_{\alpha, \varepsilon}} ^{k} \bar{\partial}^{*}   N _{\delta} u\rnm ^{2} + \lnm \left[\bar{\partial}, D _{X_{\alpha, \varepsilon}} ^{k} \right] N _{\delta} u\rnm ^{2} + \lnm \left[\bar{\partial}^{*}, D _{X_{\alpha, \varepsilon}} ^{k} \right] N _{\delta} u\rnm ^{2}\\\nonumber
   % &+ \lnm  D _{X_{\alpha, \varepsilon}} ^{k - 1} \bar{\partial} N _{\delta} u\rnm ^{2} + \lnm D _{X_{\alpha, \varepsilon}} ^{k - 1}\bar{\partial} ^{*}  N _{\delta} u\rnm ^{2} + C _{\varepsilon}\lnm N _{\delta}u\rnm _{k - 1} ^{2} + C _{\varepsilon}\lnm u\rnm _{k - 1} ^{2}\\\nonumber
    \lesssim &\lnm  D _{X_{\alpha, \varepsilon}} ^{k} \bar{\partial} N _{\delta} u\rnm ^{2} + \lnm D _{X_{\alpha, \varepsilon}} ^{k} \bar{\partial}^{*}   N _{\delta} u\rnm ^{2} + \lnm \left[\bar{\partial}, D _{X_{\alpha, \varepsilon}} ^{k} \right] N _{\delta} u\rnm ^{2} + \lnm \left[\bar{\partial}^{*}, D _{X_{\alpha, \varepsilon}} ^{k} \right] N _{\delta} u\rnm ^{2} + C _{\varepsilon}\lnm u\rnm _{k - 1} ^{2}.
  \end{align} 
%  
  %
%  Combine (\ref{555}) to (\ref{666}), we have
%  \begin{align*}
%    &\lnm\left[\bar{\partial}, D _{X_{\alpha, \varepsilon}} ^{k}\right]  N _{\delta} u\rnm ^{2}\\\nonumber
 %   \lesssim & 2k \lnm \left[\bar{\partial}, v_{\alpha, \varepsilon}\right] D _{X_{\alpha, \varepsilon}} ^{k-1} N _{\delta} u \rnm ^{2} + C _{\varepsilon, k}\lnm u\rnm _{k - 1} ^{2}\\\nonumber
 %   &+8k\varepsilon^{2}\(\lnm  D _{X_{\alpha, \varepsilon}} ^{k} \bar{\partial} N _{\delta} u\rnm ^{2} + \lnm D _{X_{\alpha, \varepsilon}} ^{k} \bar{\partial}^{*}   N _{\delta} u\rnm ^{2} + \lnm \left[\bar{\partial}, D _{X_{\alpha, \varepsilon}} ^{k} \right] N _{\delta} u\rnm ^{2} + \lnm \left[\bar{\partial}^{*}, D _{X_{\alpha, \varepsilon}} ^{k} \right] N _{\delta} u\rnm ^{2}\),\nonumber
%  \end{align*}
%  that is,
%  \begin{align}\label{777}
 %   &\lnm\left[\bar{\partial}, D _{X_{\alpha, \varepsilon}} ^{k}\right]  N _{\delta} u\rnm ^{2}\\\nonumber
 %   \lesssim & 2k \lnm \left[\bar{\partial}, v_{\alpha, \varepsilon}\right] D _{X_{\alpha, \varepsilon}} ^{k-1} N _{\delta} u \rnm ^{2} + C _{\varepsilon, k}\lnm u\rnm _{k - 1} ^{2} +8k\varepsilon ^{2}\(\lnm  D _{X_{\alpha, \varepsilon}} ^{k} \bar{\partial} N _{\delta} u\rnm ^{2} + \lnm D _{X_{\alpha, \varepsilon}} ^{k} \bar{\partial}^{*}   N _{\delta} u\rnm ^{2} +  \lnm \left[\bar{\partial}^{*}, D _{X_{\alpha, \varepsilon}} ^{k} \right] N _{\delta} u\rnm ^{2}\),\nonumber
 % \end{align}
 % where the last term can be observerd to the left-side hand, thus we only need to estimate the first term. For the first term, 
 %
In addition, it follows from \autoref{dbar lie bracket} and the induction in (\ref{induction}) that
  \begin{align}\label{888888}
    &~~~~~~ \lnm \left[\bar{\partial}, v_{\alpha, \varepsilon}\right]D _{X_{\alpha, \varepsilon}} ^{k - 1}  N _{\delta} u\rnm ^{2} \\\nonumber 
       \lesssim & \varepsilon^2 C_k \(\lnm \bar{\partial} D _{X_{\alpha, \varepsilon}} ^{k} N _{\delta} u \rnm ^{2} +\lnm \bar{\partial} ^{*} D _{X_{\alpha, \varepsilon}} ^{k} N _{\delta} u   \rnm ^{2} \) + C _{\varepsilon, k} \left( \lnm  \bar{\partial} ^{*}  N _{\delta} u  \rnm_{k-1}^{2} +  \lnm  \bar{\partial}  N _{\delta} u  \rnm_{k-1}^{2}+ \lnm N _{\delta} u \rnm_{k-1}^2+ {\rm s.c.} \lnm N _{\delta} u \rnm_{k}^2  \right) \\\nonumber
          \lesssim &  \varepsilon^2 C_k \(\lnm  D _{X_{\alpha, \varepsilon}} ^{k} \bar{\partial} N _{\delta} u \rnm ^{2} +\lnm  D _{X_{\alpha, \varepsilon}} ^{k} \bar{\partial} ^{*} N _{\delta} u   \rnm ^{2} +  \lnm\left[\bar{\partial}^{*}, D _{X_{\alpha, \varepsilon}} ^{k}\right]  N _{\delta} u\rnm ^{2}+ \lnm\left[\bar{\partial}, D _{X_{\alpha, \varepsilon}} ^{k}\right]  N _{\delta} u\rnm ^{2} \) \\\nonumber 
          &+  C _{\varepsilon, k}\left( \lnm  u \rnm_{k-1}^2 + {\rm s.c.} \lnm N _{\delta} u \rnm_{k}^2 \right) . \\\nonumber
  \end{align} 
%where the second inequality follows from (\ref{duan123}), the benign estimate  and the induction in (\ref{induction}), the fourth inequality from (\ref{MKH of uT}), the fifth  inequality from Corollary \ref{uN estimate} and sixth  inequality from the induction in (\ref{induction}). 
Combining (\ref{555}) to (\ref{888888}), we have
  \begin{align}\label{dbar 3}
    \lnm\left[\bar{\partial}, D _{X_{\alpha, \varepsilon}} ^{k}\right]  N _{\delta} u\rnm ^{2}
   & \lesssim  C_{\varepsilon, k} \lnm u \rnm _{k - 1} ^{2} + \text{s.c.} C _{\varepsilon, k}\lnm N _{\delta} u \rnm^2 _{k} \\\nonumber
   &+
    C_k \varepsilon^{2}\(\lnm  D _{X_{\alpha, \varepsilon}} ^{k} \bar{\partial} N _{\delta} u\rnm ^{2} + \lnm D _{X_{\alpha, \varepsilon}} ^{k} \bar{\partial}^{*}   N _{\delta} u\rnm ^{2} +  \lnm \left[\bar{\partial}^{*}, D _{X_{\alpha, \varepsilon}} ^{k} \right] N _{\delta} u\rnm ^{2} +  \lnm \left[\bar\partial, D _{X_{\alpha, \varepsilon}} ^{k} \right] N _{\delta} u\rnm ^{2}\).
  \end{align}
  Combining (\ref{dbar star3}) and (\ref{dbar 3}), we obtain
 % \begin{align*}
  %  &\lnm\left[\bar{\partial}^{*}, D _{X_{\alpha, \varepsilon}} ^{k}\right]  N _{\delta} u\rnm ^{2}+ \lnm\left[\bar{\partial}, D _{X_{\alpha, \varepsilon}} ^{k}\right]  N _{\delta} u\rnm ^{2}\\\nonumber
  %  \lesssim &\(2k^{2}\varepsilon + 8k\varepsilon^{2}\) \(\lnm  D _{X_{\alpha, \varepsilon}} ^{k} \bar{\partial} N _{\delta} u\rnm ^{2} + \lnm D _{X_{\alpha, \varepsilon}} ^{k} \bar{\partial}^{*}   N _{\delta} u\rnm ^{2}\) + C_{\varepsilon, k} \lnm u \rnm _{k - 1} ^{2} + \text{s.c.} C _{\varepsilon}\lnm N _{\delta} u \rnm _{k}\\
  %  &+ 2k^{2}\varepsilon \lnm \left[\bar{\partial}, D _{X_{\alpha, \varepsilon}} ^{k} \right] N _{\delta} u\rnm ^{2}+ 8k\varepsilon^{2} \lnm \left[\bar{\partial}^{*}, D _{X_{\alpha, \varepsilon}} ^{k} \right] N _{\delta} u\rnm ^{2},
 % \end{align*}
%  that is,
  \begin{align}\label{lie bracket combine}
    &\lnm\left[\bar{\partial}^{*}, D _{X_{\alpha, \varepsilon}} ^{k}\right]  N _{\delta} u\rnm ^{2}+ \lnm\left[\bar{\partial}, D _{X_{\alpha, \varepsilon}} ^{k}\right]  N _{\delta} u\rnm ^{2}\\\nonumber
    \lesssim & C_k \varepsilon^2  \(\lnm  D _{X_{\alpha, \varepsilon}} ^{k} \bar{\partial} N _{\delta} u\rnm ^{2} + \lnm D _{X_{\alpha, \varepsilon}} ^{k} \bar{\partial}^{*}   N _{\delta} u\rnm ^{2}\) + C_{\varepsilon, k} \lnm u \rnm _{k - 1} ^{2} + \text{s.c.} C _{\varepsilon}\lnm N _{\delta} u \rnm^2 _{k}\nonumber.
  \end{align}
 Combining (\ref{nabla4}), (\ref{dbar star4}), (\ref{dbar4}), (\ref{Q delta}) and (\ref{lie bracket combine}), we have
  \begin{align}\label{all DX}
    &\delta \lnm  D _{X_{\alpha, \varepsilon}} ^{k} \nabla N _{\delta} u \rnm ^{2} + \lnm D_{X_{\alpha, \varepsilon}} ^{k} \bar{\partial} ^{*} N _{\delta} u \rnm ^{2} + \lnm D_{X_{\alpha, \varepsilon}} ^{k} \bar{\partial} N _{\delta} u \rnm ^{2}\\\nonumber
    \lesssim &C_k \varepsilon^{2} \(\delta\lnm v_{\alpha, \varepsilon} ^{k} \nabla_{T} N _{\delta} u \rnm^{2} + \lnm v_{\alpha, \varepsilon} ^{k} \bar{\partial} ^{*} N _{\delta} u \rnm^{2} + \lnm v_{\alpha, \varepsilon} ^{k} \bar{\partial} N _{\delta} u \rnm^{2}\)+ (\text{s.c.}+\delta) C _{\varepsilon, k}\(\lnm \bar{\partial}N _{\delta} u\rnm _{k} ^{2} +\lnm \bar{\partial} ^{*}N _{\delta} u\rnm _{k} ^{2}\) \\\nonumber
    &+ (\text{s.c.}+ \delta)  C _{\varepsilon, k}\delta\lnm \nabla N _{\delta} u\rnm_k ^{2} + C _{\varepsilon, k} \lnm u\rnm _{k} ^{2}.\nonumber
  \end{align}
  Combining (\ref{nabla}), (\ref{dbar star}), (\ref{dbar}) and (\ref{all DX}), we obtain:
   \begin{align}\label{duan1}
   & \delta \lnm  v^k_{\alpha, \varepsilon} \nabla_T N _{\delta} u \rnm ^{2} + \lnm v^k_{\alpha, \varepsilon} \bar{\partial} ^{*} N _{\delta} u \rnm ^{2} + \lnm v^k_{\alpha, \varepsilon} \bar{\partial} N _{\delta} u \rnm ^{2} \\\nonumber
    \lesssim & (\text{s.c.}+\delta) C _{\varepsilon, k}\(\lnm \bar{\partial}N _{\delta} u\rnm _{k} ^{2} +\lnm \bar{\partial} ^{*}N _{\delta} u\rnm _{k} ^{2}\) + (\text{s.c.}+ \delta)  C _{\varepsilon, k}\delta\lnm \nabla N _{\delta} u\rnm_k ^{2} + C _{\varepsilon, k} \lnm u\rnm _{k} ^{2}.\nonumber
  \end{align}
  Finally, combining  (\ref{wt2}), %(\ref{nabla}), 
  (\ref{wt3}), %(\ref{dbar star}), 
  (\ref{wt4}), %(\ref{dbar}), 
  (\ref{duan1}) and letting $\delta C_{\varepsilon, k} \ll 1$, $\text{s.c.}  C_{\varepsilon, k} \ll 1$, we have
  \begin{align*}
    \lnm \bar{\partial} N _{\delta} u \rnm _{k} ^{2} + \lnm \bar{\partial} ^{*} N _{\delta}u \rnm _{k}^{2} + \delta\lnm \nabla N _{\delta} u \rnm _{k}^{2}
    \leq C _{\varepsilon, k} \lnm u \rnm _{k} ^{2}, 
  \end{align*}
  which is (\ref{induction})  for $\beta = k$ and finishes the induction.
\end{proof}

\subsection{Global regularity}
The genuine estimates follow from \autoref{priori estiamte} by the standard approximation argument as $\delta \rightarrow 0^+$ (cf. \cite{KN65, S08}). 

\begin{thm}\label{genuine estimate}
  Let $n > 2$ and $\Om \subset \bbC ^{n}$ be a smooth bounded pseudoconvex domain 
 that possesses a family of transverse vector fields satisfying property (${\widetilde P}_{n-1}^{\#}$). %or property ($P _{n-1} ^{\#}$). 
  Then the $\bar{\partial}$-Neumann operator $N _{n - 1}$ is continuous on $W _{\(0, n - 1\)}  ^{k} \(\Om\)$ for any $k \geq 0$.
\end{thm} 

\section{Exact regularity for $\(0, q\)$-forms}

\begin{prop}\label{condition 2 basic estimate2} 
  Let $\Om$ be a bounded smooth pseudoconvex domain in $\mathbb{C}^n$ with $n \geq 3$ and $U$ be a small open neighborhood of some boundary point. Assume that $\lambda \in C^{2} (\cl{\Om})$ satisfies $\mathcal{Y}^{q} \(\lambda \) \geq \left( C(\gamma, \eta, n)+10\right) \sum_{j \leq n} \lmdl L _{j} \lambda  \rmdl ^{2} $ on $\partial\Om \cap U$, 
 for some $1 \leq t \leq n-1$. Let $u = \sum'_J u_J^{\dagger} \bar{w}_J \in L ^{2}_{\(0, q\)}\(\Om\) \cap \dom\(\bar{\partial}\) \cap \dom\(\bar{\partial} ^{*}\)$ with support in $\cl{\Om} \cap U$. Then there exists $C>0$, such that 
  \begin{align*}%\label{MKH of u}
\int_\Om \mathcal{X}^{q} \(\lambda\) \big(u, u \big) \, dV \lesssim     \lnm \bar{\partial} u \rnm ^2 + \lnm \bar{\partial}^* u  \rnm^2 .
  \end{align*}
\end{prop}

\begin{proof}
  If $\lambda$ satisfies $\mathcal{Y}^{q} \(\lambda\) \geq \left( C(\gamma, \eta, n)+10\right) \sum_{j \leq n} \lmdl L _{j} \lambda  \rmdl ^{2} $ on $\partial\Om$, then it follows from applying  \autoref{modify basic estimate} to $u$ that
  \begin{equation*}
  \begin{split}
    \lnm \bar{\partial} u \rnm^2 + \lnm \bar{\partial}^* u  \rnm^2  
    \gtrsim & \sideset{}{'}{\sum} _{\lmdl J \rmdl = q} \sum_{j \in I_s} \( - \lambda^\dagger_{jj}  u_J^{\dagger},  u_J^{\dagger} \) + \sideset{}{'}{\sum} _{\lmdl K \rmdl = q - 1} \sum_{i,j = 1} ^{n} \( \lambda^{\dagger}_{ij}   u_{iK}^{\dagger}, u_{jK}^{\dagger} \)
    = \int_\Om \mathcal{X}^{q} \(\lambda\) \big(u, u \big) \, dV.
    \end{split}
  \end{equation*}
%Thus  (\ref{MKH of u}) follows by combining the standard Morrey-Kohn-H\"ormander estimates.
\end{proof}

We now derive  the exact regularity for $N_q$.

\begin{thm}\label{priori estiamte2}
  Let $n > 2$ and $\Om \subset \bbC ^{n}$ be a smooth bounded pseudoconvex domain 
 that possesses a family of transverse vector fields satisfying property strong (${\widetilde P}_{q}^{\#}$). %or property ($P _{q} ^{\#}$) for $1 < q <n-1$. 
  Then the $\bar{\partial}$-Neumann operator $N _{q'}$ is continuous on $W _{\(0, q'\)}  ^{k} \(\Om\)$ for any $k \geq 0$ and any $n-1 \geq q'\geq q$.
\end{thm} 

\begin{proof}
It suffices to derive the a priori estimates. 
  By \autoref{increasing}, $\Om$ possesses a family of transverse vector fields satisfying property strong (${\widetilde P}_{q'}^{\#}$) %or property ($P _{q'} ^{\#}$) 
  for all $q' \geq q$. Use the downward induction on $q'$ starting with $q'=n-1$. %For $q'=n-1$, 
  Since $B_{n-1}$ is automatically exactly regular on the $L^2$-Sobolev space, \autoref{B regular} can be applied to $(0, n-1)$-forms. By the discussion below, we can show that $N_{n-1}$ is exactly regular on the $L^2$-Sobolev space. This further implies that $B_{n-2}$ is exactly regular on the $L^2$-Sobolev space. By the induction, since $N_{q+1}$ is exactly regular, so is $B_q$ and thus $N_q$ is also exactly regular. This proves the theorem. 
  
  %For fixed $s$, $q _{0}$ and $I _{s}$, we claim that property $\title{P} _{q} ^{\#}$ implies $\title{P} _{q + 1} ^{\#}$. According to Schur’s majorization theorem, if conditions (ii) and (iii) in \autoref{property Pq} hold, then the $q$-th smallest eigenvalue of the complex Hessian $\lambda _{ij}$ minus $\sum _{j \in I _{s}} \lambda _{jj}$ is denoted by $\mathcal{X}_{q} \(\lambda\)$ must be positive. This means that the $\(q + 1\)$-th smallest eigenvalue of $\lambda _{ij}$ minus $\sum _{j \in I _{s}} \lambda _{jj}$ is denoted by $\mathcal{X}_{q} \(\lambda\)$ must also be positive. Therefore, for the case of $q+1 $with the same function family $\lambda$, conditions (ii) and (iii) still hold.
  
 For the proof of the exact regularity of $N_{q'}$ assuming the exact regularity of  $B_{q'}$ and that $\Omega$ possesses a family of transverse vector fields satisfying property strong (${\widetilde P}_{q'}^{\#}$), we only point out the difference from the proof of  \autoref{genuine estimate}, which is in the estimates of $\left[\bar{\partial} ^{*}, D _{X_{\alpha, \varepsilon}} ^{k}\right] N _{\delta}u $ and $ \theta ^{v_{\alpha, \varepsilon}} \wedge v_{\alpha, \varepsilon} D _{X_{\alpha, \varepsilon}} ^{k - 1} N _{\delta} u$.
 %
% from the case of $\(0, n - 1\)$-forms, which is in the estimates of $\lnm[\bar{\partial} ^{*}, D _{X_{\alpha, \varepsilon}} ^{k}] N _{\delta} u\rnm ^{2}$ and $\lnm \theta ^{v_{\alpha, \varepsilon}} \wedge v D _{X_{\alpha, \varepsilon}} ^{k - 1} N _{\delta} u \rnm ^{2}$. 
%
%
For $\left[\bar{\partial} ^{*}, D _{X_{\alpha, \varepsilon}} ^{k}\right] N _{\delta} u$, %we use the following estimates to replace \autoref{lie bracket basic estimate}. 
similar to the proof of \autoref{lie bracket basic estimate},
it follows from (ii), (iii) in  \autoref{property Pq},  \autoref{modify basic estimate} and \autoref{condition 2 basic estimate2} that
 \begin{align*}
    \lnm \left[D_{X_{\alpha, \varepsilon}} ^{k}, \bar{\partial}^*\right] N _{\delta} u\rnm^2
   & \lesssim k^2  \sideset{}{'}{\sum} _{I} \sum_{j=1}^n \lnm \overline{\theta ^{v _{\alpha, \varepsilon}} _{j}} \left(D_{X_{\alpha, \varepsilon}} ^{k}N _{\delta} u\right)_{jI} \rnm^2 + C _{\varepsilon, k}\(\lnm \bar{\partial} N _{\delta}u \rnm _{k -1} ^{2}+\lnm \bar{\partial} ^{*} N _{\delta}u \rnm_{k -1} ^{2} + \lnm N _{\delta}u \rnm_{k -1}^2 \) \\
%  \end{align*}
 % Thus, we only need to estimate $\lnm \sum’ _{I} \sum_{j=1}^n \bar{\theta} ^{v _{\alpha, \varepsilon}} _{j} D_{X_{\alpha, \varepsilon}} ^{k} u_{jI} \rnm ^{2}$. By changing of coordinates and (\ref{change of coordinates}), we have
%  \begin{align*}
  %  \lnm \sideset{}{'}{\sum} _{I} \sum_{j=1}^n \bar{\theta} ^{v _{\alpha, \varepsilon}} _{j} D_{X_{\alpha, \varepsilon}} ^{k} u_{jI} \rnm ^{2}
%   & \leq\lnm \sideset{}{'}{\sum} _{I} \sum_{j=1}^n \(\bar{\theta} ^{v _{\alpha, \varepsilon}} _{j}\)^{\dagger}D_{X_{\alpha, \varepsilon}} ^{k} u_{j I}^{\dagger} \rnm ^{2} + C _{\varepsilon} \lnm u\rnm_{k-1} \\
%  \end{align*}
%  where $u = \sum'_{I} \sum_{j=1}^{n} u _{jI}^{\dagger} \bar{\omega} _{j} \wedge \bar{\omega} _{I} = u _{T} + u _{N}$, where $u _{T}$ is the tangential component of $u$ and $u_{N}$ is the normal component of $u$. Then we have
%  \begin{align*}
%    \lnm \sideset{}{'}{\sum} _{I} \sum_{j=1}^n \(\bar{\theta} ^{v _{\alpha, \varepsilon}} _{j}\)^{\dagger}D_{X_{\alpha, \varepsilon}} ^{k} u_{j I}^{\dagger} \rnm ^{2}
    &\lesssim   k^2 \lnm \left| \theta  ^{v_{\alpha, \varepsilon}} \right| \left| D _{X_{\alpha, \varepsilon}} ^{k} N _{\delta} u  \right| \rnm ^{2} + C _{\varepsilon, k} \lnm u \rnm_{k -1}^2 \\
    %
    % %+ \lnm \bar{\theta}^{v _{\alpha, \varepsilon}} _{N} D_{X_{\alpha, \varepsilon}} ^{k} u\rnm ^{2}\\
    &  \lesssim  k^2\varepsilon^2 \left( \mathcal{X}^{q}_{ij} \(\lambda_{\alpha, \varepsilon}\)(D _{X_{\alpha, \varepsilon}} ^{k}N _{\delta}  u),  D _{X_{\alpha, \varepsilon}} ^{k}  N _{\delta}u \right) +  C _{\varepsilon, k} \lnm u \rnm_{k -1}^2  \\
 %    \lesssim & \varepsilon^2 \(\lnm \bar{\partial} \left( D _{X_{\alpha, \varepsilon}} ^{k} N _{\delta} u\right)_{\rm Tan} \rnm ^{2} +\lnm \bar{\partial} ^{*} \left( D _{X_{\alpha, \varepsilon}} ^{k} N _{\delta} u\right)_{\rm Tan}  \rnm ^{2} \) + C _{\varepsilon} \lnm \(   D _{X_{\alpha, \varepsilon}}^{k} N _{\delta} u \) _{\rm Norm}  \rnm^{2} + C _{\varepsilon, k} \lnm  u \rnm_{k-1}^2 \\\nonumber
  %   C _{\varepsilon, k} \(\lnm \bar{\partial} \left( D _{X_{\alpha, \varepsilon}} ^{k} N _{\delta} u\right)_{\rm Tan} \rnm ^{2} +\lnm \bar{\partial} ^{*} \left( D _{X_{\alpha, \varepsilon}} ^{k} N _{\delta} u\right)_{\rm Tan}  \rnm ^{2} +  \lnm u \rnm^2 _{k-1} \)\nonumber\\
    & \lesssim  k^2\varepsilon^2 \(\lnm \bar{\partial} D _{X_{\alpha, \varepsilon}} ^{k} N _{\delta} u \rnm ^{2} + \lnm \bar{\partial} ^{*} D _{X_{\alpha, \varepsilon}} ^{k} N _{\delta}u   \rnm ^{2} \) + C _{\varepsilon, k} \lnm u \rnm_{k -1}^2  .
  \end{align*}
  For $ \theta ^{v_{\alpha, \varepsilon}} \wedge v_{\alpha, \varepsilon} D _{X_{\alpha, \varepsilon}} ^{k - 1} N _{\delta} u $, we use similar estimates as % replace the corresponding estimates % in (\ref{888}
   in the proof of \autoref{dbar lie bracket}: 
    \begin{align*}\label{wentian888}
 \lnm \theta ^{v_{\alpha, \varepsilon}} \wedge v_{\alpha, \varepsilon} D _{X_{\alpha, \varepsilon}} ^{k-1 } N _{\delta} u \rnm ^{2}
    \leq  \lnm \left| \theta ^{v_{\alpha, \varepsilon}} \right| \left| v_{\alpha, \varepsilon} D _{X_{\alpha, \varepsilon}} ^{k-1 } N _{\delta} u \right|\rnm ^{2} 
     \lesssim \lnm \left| \theta  ^{v_{\alpha, \varepsilon}} \right| \left| D _{X_{\alpha, \varepsilon}} ^{k} N _{\delta} u  \right| \rnm ^{2} +  C _{\varepsilon, k}\lnm  u\rnm _{k -1} ^{2} %\\\nonumber
  \end{align*} 
  and the rest of the proof is the same as above.
  %to replace the corresponding estimates in (\ref{888}), 
%  where the third inequality follows from (iii) in  \autoref{property Pq} and the fourth inequality from \autoref{condition 2 basic estimate2}. %and the fourth inequality from \autoref{modify basic estimate}.
   \end{proof} 

{\bf Acknowledgement}: We thank Xinchen Duan, Zheng-Chao Han, Phillip S. Harrington, Mijia Lai, Yue Zhang for very helpful discussions. The second author is supported in part by Zhejiang Provincial Natural Science Foundation of China (Grant No. LQKWL26A0201).  
The third author is supported in part by the China
Postdoctoral Science Foundation (Grant No. 2024M762395).

\bigskip

{\bf Declarations:} The authors declare no conflict of interest.

\end{document}